\documentclass[a4j,12pt,showkeys]{article}

\usepackage{amsmath,amsthm,amssymb}
\usepackage{geometry}
\usepackage{hyperref}
\usepackage{cancel}
\usepackage{mathtools}
\usepackage{bm}
\usepackage{url}

\date{}

\newtheorem{theo}{Theorem}[section]
\newtheorem{lemm}[theo]{Lemma}
\newtheorem{prop}[theo]{Proposition}
\newtheorem{cor}[theo]{Corollary}

\theoremstyle{definition}
\newtheorem{defi}[theo]{Definition}
\newtheorem{rem}[theo]{Remark}
\newtheorem{exam}[theo]{Example}

\newtheorem{assum}{Assumption}
\newcommand{\ep}{\varepsilon}

\newcommand{\relmiddle}[1]{\mathrel{}\middle#1\mathrel{}}
 
\newcommand{\1}{\mbox{\rm{1}}\hspace{-0.25em}\mbox{\rm{l}}}

\newcommand{\pert}{{\tau,\ep}}
\providecommand{\keywords}[1]{\textbf{Keywords:} #1}

\allowdisplaybreaks[0]

\makeatletter

\@addtoreset{equation}{section}
\makeatother
\def\widebar{\accentset{{\cc@style\underline{\mskip10mu}}}}
\numberwithin{equation}{section}
\def\rnum#1{\expandafter{\romannumeral #1}} 
\def\Rnum#1{\uppercase\expandafter{\romannumeral #1}}

\title{Extended backward stochastic Volterra integral equations and their applications to time-inconsistent stochastic recursive control problems\thanks{This is a pre-copy-editing, author-produced PDF of an article accepted for publication in \emph{Math.\ Control Relat.\ Fields} following peer review. The definitive publisher-authenticated version is available online at: \url{https://www.aimsciences.org/article/doi/10.3934/mcrf.2020043}.}}

\author{Yushi Hamaguchi\thanks{Department of Mathematics, Kyoto University, Kyoto 606--8502, Japan, \href{mailto:hamaguchi@math.kyoto-u.ac.jp}{hamaguchi@math.kyoto-u.ac.jp}}}

\begin{document}
\maketitle

%%%%%%%%%%%%%%%%%%%%%%%%%%%%%%%%%%%%%%%%%%%%%%%%%%%%%%%%%%%%%%%%%%%%%%%%%%%%%%%%%%%%%%%%%%%%%%
%%%% Abstract
%%%%%%%%%%%%%%%%%%%%%%%%%%%%%%%%%%%%%%%%%%%%%%%%%%%%%%%%%%%%%%%%%%%%%%%%%%%%%%%%%%%%%%%%%%%%%%

\begin{abstract}
In this paper, we study extended backward stochastic Volterra integral equations (EBSVIEs, for short). We establish the well-posedness under weaker assumptions than those of known results, and prove a new kind of regularity property for the solutions. As an application, we investigate, in the open-loop framework, a time-inconsistent stochastic recursive control problem where the cost functional is defined by the solution to a backward stochastic Volterra integral equation (BSVIE, for short). We show that the corresponding adjoint equations become EBSVIEs, and provide a necessary and sufficient condition for an open-loop equilibrium control via variational methods.
\end{abstract}

\keywords{Extended backward stochastic Volterra integral equation; backward stochastic Volterra integral equation; time-inconsistent stochastic recursive control problem; open-loop equilibrium control}

%%%%%%%%%%%%%%%
%% Section
%%%%%%%%%%%%%%%

\section{Introduction}

Throughout this paper, we let $W(\cdot)$ be a $d$-dimensional Brownian motion on a complete probability space $(\Omega,\mathcal{F},\mathbb{P})$. $\mathbb{F}=(\mathcal{F}_t)_{t\geq0}$ denotes the $\mathbb{P}$-augmentation of the filtration generated by $W(\cdot)$. Let $0\leq S<T<\infty$ be fixed. In this paper, we study, together with an application to stochastic control, the following \emph{extended backward stochastic Volterra integral equation} (EBSVIE, for short):
\begin{equation}\label{EBSVIE}
\begin{split}
	&Y(t,s)=\psi(t)+\int^T_sg(t,r,Y(r,r),Y(t,r),Z(t,r))\,dr-\int^T_sZ(t,r)\,dW(r),\\
	&\hspace{5cm}s\in[S,T],\ t\in[S,T],
\end{split}
\end{equation}
where $\psi:\Omega\times[S,T]\to\mathbb{R}^m$ and $g:\Omega\times[S,T]^2\times\mathbb{R}^m\times\mathbb{R}^m\times\mathbb{R}^{m\times d}\to\mathbb{R}^m$ are given maps. By an adapted solution to \eqref{EBSVIE}, we mean a pair of $\mathbb{R}^m\times\mathbb{R}^{m\times d}$-valued random fields $(Y(\cdot,\cdot),Z(\cdot,\cdot))=\{(Y(t,s),Z(t,s))\}_{(t,s)\in[S,T]^2}$ such that
\begin{itemize}
\item
the map $[S,T]^2\times\Omega\ni(t,s,\omega)\mapsto(Y(t,s,\omega),Z(t,s,\omega))\in\mathbb{R}^m\times\mathbb{R}^{m\times d}$ is measurable,
\item
for each fixed $t\in[S,T]$, the process $Y(t,\cdot)=(Y(t,s))_{s\in[S,T]}$ is continuous and $\mathbb{F}$-adapted,
\item
for each fixed $t\in[S,T]$, the process $Z(t,\cdot)=(Z(t,s))_{s\in[S,T]}$ is $\mathbb{F}$-progressively measurable, and
\item
the equality in \eqref{EBSVIE} holds a.s.\ for any $s\in[S,T]$ and $t\in[S,T]$.
\end{itemize}
We call $\psi$ the free term and $g$ the generator of EBSVIE~\eqref{EBSVIE}. If the generator $g(t,r,\eta,y,z)$ does not depend on $y$, then EBSVIE~\eqref{EBSVIE} can be seen as an integral equation for $\eta(t)=Y(t,t)$ and $\zeta(t,s)=Z(t,s)$, and it reduces to the so-called Type-\Rnum{1} \emph{backward stochastic Volterra integral equation} (BSVIE, for short) of the following form:
\begin{equation*}
	\eta(t)=\psi(t)+\int^T_tg(t,s,\eta(s),\zeta(t,s))\,ds-\int^T_t\zeta(t,s)\,dW(s),\ t\in[S,T].
\end{equation*}
If moreover $\psi$ and $g$ do not depend on $t$, then the above equation reduces to a well-known \emph{backward stochastic differential equation} (BSDE, for short) with the adapted solution $(\eta(\cdot),\zeta(\cdot))$:
\begin{equation*}
	\eta(t)=\psi+\int^T_tg(s,\eta(s),\zeta(s))\,ds-\int^T_t\zeta(s)\,dW(s),\ t\in[S,T],
\end{equation*}
which can be rewritten in the differential form:
\begin{equation*}
	\begin{cases}
		d\eta(s)=-g(s,\eta(s),\zeta(s))\,ds+\zeta(s)\,dW(s),\ s\in[S,T],\\
		\eta(T)=\psi.
	\end{cases}
\end{equation*}
\par
BSDEs have been extensively researched, and established as a fundamental object in mathematical finance and stochastic control; see for example the survey paper \cite{a_ElKaroui-_97} and the textbook \cite{b_Zhang_17}. BSVIEs were firstly studied by Lin~\cite{a_Lin_02}, and further investigated by Yong~\cite{a_Yong_06,a_Yong_08}, Shi--Wang~\cite{a_Shi-Wang_12}, Wang--Yong~\cite{a_Wang-Yong_19}, Shi--Wen--Xiong~\cite{a_Shi-Wen-Xiong_20}, and so on. BSVIEs have become a popular tool for studying some problems in mathematical finance. Yong~\cite{a_Yong_07} applied BSVIEs to dynamic risk measures. Wang--Sun--Yong~\cite{a_Wang-Sun-Yong_19} established the well-posedness of quadratic BSVIEs, and explored the applications of quadratic BSVIEs to equilibrium dynamic risk measures and equilibrium recursive utility processes. Shi--Wang--Yong~\cite{a_Shi-Wang-Yong_15} and Wang--Zhang~\cite{a_Wang-Zhang_17} investigated optimal control problems of BSVIEs. Recently, as a generalization of BSVIEs, Wang~\cite{a_Wang_20} introduced EBSVIEs, and investigated the Feynman--Kac formula for a non-local quasilinear parabolic partial differential equation (PDE, for short). A similar equation was considered by the author's work~\cite{a_Hamaguchi_20}. The author established the well-posedness of a ``flow of forward-backward stochastic differential equations'' over small time horizon, which is a coupled system of a stochastic differential equation (SDE, for short) and an EBSVIE. In this paper, we deal with the well-posedness and a regularity property of EBSVIE~\eqref{EBSVIE} under weaker assumptions than the literature. Furthermore, we show that EBSVIEs naturally arise in a time-inconsistent stochastic recursive control problem. To the best of our knowledge, the present paper is of the first result to show the applicability of EBSVIEs to stochastic control.
\par
\medskip
In recent years, time-inconsistent stochastic control problems have received remarkable attentions in stochastic control, mathematical finance and economics. Time-inconsistency for a dynamic control problem means that the so-called Bellman's principle of optimality does not hold. In other words, a restriction of an optimal control for a specific initial pair on a later time interval might not be optimal for that corresponding initial pair. Such a situation occurs for example in dynamic mean-variance control problems, and in utility maximization problems for consumption-investment strategies under non-exponential discounting. In order to deal with a time-inconsistent problem in a sophisticated way, Strotz~\cite{a_Strotz_73} introduced an approach which regards the dynamic problem as a non-cooperative game, where decisions at every instant of time are selected by different players (which represent the incarnations of the controller). Nash equilibria are therefore considered instead of optimal controls. This approach was adopted and further developed by Bj\"{o}rk--Khapko--Murgoci~\cite{a_Bjork-_17}, Djehiche--Huang~\cite{a_Djehiche-Huang_16}, Yong~\cite{a_Yong_12,a_Yong_17}, Wei--Yong--Yu~\cite{a_Wei_17}, Yan--Yong~\cite{a_Yan-Yong_19}, Wang--Yong~\cite{a_HWang-Yong_19}, Hu--Jin--Zhou~\cite{a_Hu-Jin-Zhou_12,a_Hu-Jin-Zhou_17}, Hu--Huang--Li~\cite{a_Hu-Huang-Li_17}, Alia~\cite{a_Alia_19}, and so on. Time-inconsistent consumption-investment problems under non-exponential discounting were studied by, for example, Ekeland--Pirvu~\cite{a_Ekeland-Pirvu_08}, Alia et al.~\cite{a_Alia-_17}, and Hamaguchi~\cite{a_Hamaguchi_19}. The equilibrium concepts investigated in the literature can be roughly divided into two different types, that is, (\rnum{1}) a \emph{closed-loop equilibrium strategy} and (\rnum{2}) an \emph{open-loop equilibrium control}. Let us briefly review these two concepts.
\begin{enumerate}
\renewcommand{\labelenumi}{(\roman{enumi})}
\item
A closed-loop equilibrium strategy is an equilibrium concept for a ``decision rule'' that a controller uses to select a ``control action'' based on each state. Mathematically, a strategy is a mapping from states to control actions, which is chosen independently of initial conditions. Concerning this formulation, Yong~\cite{a_Yong_12} performed a multi-person differential game approach for a general discounting time-inconsistent stochastic control problem, and characterized the closed-loop equilibrium strategy via the so-called equilibrium Hamilton--Jacobi--Bellman (HJB, for short) equation. This approach was further developed in \cite{a_Yong_17,a_Wei_17,a_Yan-Yong_19,a_HWang-Yong_19}.
\item
An open-loop equilibrium control is an equilibrium concept for a ``control process'' that a controller chooses based on the initial condition. Hu--Jin--Zhou~\cite{a_Hu-Jin-Zhou_12,a_Hu-Jin-Zhou_17} introduced and investigated an open-loop equilibrium control for a time-inconsistent stochastic linear-quadratic control problem, together with an application to a dynamic mean-variance control problem. They characterized an open-loop equilibrium control by using a variational method, which is a natural generalization of the stochastic maximum principle of Peng~\cite{a_Peng_90} to the time-inconsistent problem. This approach was further developed in \cite{a_Hu-Huang-Li_17,a_Yong_17,a_Yan-Yong_19,a_TWang_19,a_TWang_20,a_Alia-_17,a_Alia_19,a_Hamaguchi_19}.
\end{enumerate}
For a stochastic control problem, a recursive cost functional with exponential discounting can be described by the solution of a BSDE. On the other hand, as discussed in \cite{a_HWang-Yong_19}, when we consider a stochastic recursive control problem with general (non-exponential) discounting, then the proper definition of the recursive cost functional is the solution of a Type-\Rnum{1} BSVIE. In the closed-loop framework, Wang--Yong~\cite{a_HWang-Yong_19} and Yan--Yong~\cite{a_Yan-Yong_19} (Section 5) adopted the multi-person differential game approach, and studied a time-inconsistent stochastic recursive control problem where the cost functional was defined by the solution of a Type-\Rnum{1} BSVIE. A similar problem was studied by Wei--Yong--Yu~\cite{a_Wei_17} in the closed-loop framework, where the cost functional was defined by a family of parametrized BSDEs. On the other hand, to the best of our knowledge, time-inconsistent stochastic recursive control problems have not been studied in the open-loop framework.
\par
\medskip
In Sections~\ref{section: control problem} and \ref{section: proof of main result} of the present paper, we investigate, in the open-loop framework, a time-inconsistent stochastic recursive control problem where the cost functional is defined by the solution of a Type-\Rnum{1} BSVIE. We define an open-loop equilibrium control by a similar way to \cite{a_Hu-Jin-Zhou_12,a_Hu-Jin-Zhou_17}, and characterize it via variational methods. The key point is to derive the \emph{first-order adjoint equation} and the \emph{second-order adjoint equation}. In this paper, we show that the proper choices of the adjoint equations are EBSVIEs; see equations~\eqref{EBSVIE 1} and \eqref{EBSVIE 2}. For this reason, we see that EBSVIEs are important tools to deal with a time-inconsistent stochastic recursive control problem in the open-loop framework. Our method to derive the adjoint equations is inspired by Hu~\cite{a_Hu_17}. He investigated a (time-consistent) stochastic recursive control problem with the cost functional defined by the solution to a BSDE, and developed a global maximum principle. In this paper, we generalize his idea to a time-inconsistent setting with the cost functional defined by the solution to a BSVIE. It is also worth to mention that the paper \cite{a_Hu_17} provided a necessary condition for an optimal control in a time-consistent stochastic recursive control problem, while the papers \cite{a_Yong_12,a_Wei_17,a_Yan-Yong_19,a_HWang-Yong_19} provided sufficient conditions for closed-loop equilibrium strategies in time-inconsistent stochastic recursive control problems with general discounting. Compared with the above papers, we provide a \emph{necessary and sufficient condition} for an open-loop equilibrium control in a time-inconsistent stochastic recursive control problem with general discounting. We also refer to relevant works of Wang~\cite{a_TWang_19,a_TWang_20}, where the author characterized open-loop equilibrium controls in a linear-quadratic time-inconsistent mean-field control problem (which is different from our setting) by a system of conditions named as first-order and second-order equilibrium conditions.
\par
Unfortunately, the coefficients of the adjoint equations \eqref{EBSVIE 1} and \eqref{EBSVIE 2} do not satisfy the assumptions considered in \cite{a_Wang_20}; see Remark~\ref{remark: t-continuity}. Therefore, in order to justify the arguments in Sections~\ref{section: control problem} and \ref{section: proof of main result}, we need further observations on EBSVIEs. This is a motivation of Section~\ref{section: EBSVIE}. For the sake of the well-posedness of the adjoint equations, in Section~\ref{section: EBSVIE}, we prove the well-posedness of the general EBSVIE~\eqref{EBSVIE} under weaker assumptions than the literature. We provide a direct proof which is different from the original method of \cite{a_Wang_20}. Moreover, we show a new type of regularity property of the solution $(Y(\cdot,\cdot),Z(\cdot,\cdot))=\{(Y(t,s),Z(t,s))\}_{(t,s)\in[S,T]^2}$ to an EBSVIE with respect to the $t$-variable. This regularity result plays an interesting role in the study of time-inconsistent stochastic control problems.
\par
In the studies of EBSVIEs and time-inconsistent stochastic control problems, the ``diagonal process'' $Z(s,s)$ of a process $Z(t,s)$ with two time-parameters plays a crucial role. In some previous works on time-inconsistent stochastic control problems (for example, in \cite{a_Ekeland-Pirvu_08,a_Djehiche-Huang_16}), such a ``diagonal process'' was used without rigorous discussions, although even the well-definedness is not clear and questionable. Indeed, we show a counter example (Example~\ref{example: diagonal}) which says that there exists a (deterministic) process $Z(\cdot,\cdot)$ such that the term $Z(s,s)$ cannot be defined. Due to this technical difficulty, in some time-inconsistent stochastic control problems, the full characterization of an open-loop equilibrium control has been an open problem. In Section~4.1 of \cite{a_Yong_17}, a strong assumption, that is, the a.s.\ continuity of the map $(t,s)\mapsto Z(t,s)$, was imposed in the sufficient condition for an open-loop equilibrium control, but the a.s.\ continuity is difficult to check in general. Also, in Section~4 of \cite{a_Yan-Yong_19}, the characterization of an open-loop equilibrium control remained to include a limit procedure, and hence they did not provide a full characterization in a local form. In this paper, in order to overcome such difficulties arising in the existing literature, we show some abstract results on stochastic processes with two time-parameters, and provide a useful approach to treat the ``diagonal processes''. This observation is interesting by its own right, and plays a key role in our study. Indeed, this approach helps to solve the open problem arising in the full characterization (via a necessary and sufficient condition) of the open-loop equilibrium control in a time-inconsistent stochastic control problem under reasonable assumptions.
\par
\medskip
The contributions of this paper are summarized as follows:
\begin{itemize}
\item
We establish the well-posedness of the general EBSVIE~\eqref{EBSVIE} under weaker assumptions than the literature, and prove a new kind of regularity property of the solution (Theorems~\ref{theorem: EBSVIE} and \ref{theorem: EBSVIE derivative}).
\item
We provide a necessary and sufficient condition for an open-loop equilibrium control of a time-inconsistent stochastic recursive control problem via variational methods (Theorem~\ref{theorem: main result}).
\item
We derive the corresponding adjoint equations, which turn out to be EBSVIEs (equations~\eqref{EBSVIE 1}--\eqref{EBSVIE 2}).
\item
We provide a rigorous approach to deal with the ``diagonal process'' of a stochastic process which has two time-parameters (Lemma~\ref{lemma: Z diagonal}). This abstract result plays an important role in the studies of EBSVIEs and time-inconsistent control problems.
\end{itemize}
\par
The paper is organized as follows: In Section~\ref{section: preliminaries}, we introduce some notation and recall some known results. In Subsection~\ref{subsection: two time-parameters}, we investigate stochastic processes with two time-parameters. In Section~\ref{section: EBSVIE}, we prove the well-posedness of EBSVIE~\eqref{EBSVIE} and study the regularity of the solution $(Y(\cdot,\cdot),Z(\cdot,\cdot))=\{(Y(t,s),Z(t,s))\}_{(t,s)\in[S,T]^2}$ with respect to $t$. In Section~\ref{section: control problem}, we investigate a time-inconsistent stochastic recursive control problem in the open-loop framework; the main result of this section is Theorem~\ref{theorem: main result}. In Section~\ref{section: proof of main result}, we prove Theorem~\ref{theorem: main result} via variational methods. Some technical estimates needed in Section~\ref{section: proof of main result} are proved in Appendix~\ref{appendix}.

%%%%%%%%%%%%%%%
%% Section
%%%%%%%%%%%%%%%

\section{Preliminaries}\label{section: preliminaries}

Throughout this paper, $\mathrm{Leb}_{[S,T]}$ denotes the Lebesgue measure on an interval $[S,T]$, and $\1_A$ denotes the indicator function for a given set $A$. $\mathbb{E}[\cdot]$ denotes the expectation, and $\mathbb{E}_t[\cdot]:=\mathbb{E}[\cdot|\mathcal{F}_t]$ denotes the conditional expectation given by $\mathcal{F}_t$ for each $t\geq0$. We say that a function $\rho:[0,\infty)\to[0,\infty)$ is a modulus of continuity if $\rho$ is continuous, increasing, and satisfies $\rho(0)=0$. Let $p,q\geq1$, $0\leq S<T<\infty$, and let $\mathbb{H}$ be a Euclidean space. We define the following spaces of (equivalent classes of) functions and random variables:
\begin{align*}
	&L^p(S,T;\mathbb{H}):=\left\{\varphi:[S,T]\to\mathbb{H}\relmiddle|\varphi\ \text{is measurable}, \int^T_S|\varphi(s)|^p\,ds<\infty\right\},\displaybreak[1]\\
	&L^0_{\mathcal{F}_T}(\Omega;\mathbb{H}):=\{\varphi:\Omega\to\mathbb{H}\,|\,\varphi\ \text{is}\ \mathcal{F}_T\text{-measurable}\},\displaybreak[1]\\
	&L^p_{\mathcal{F}_T}(\Omega;\mathbb{H}):=\{\varphi\in L^0_{\mathcal{F}_T}(\Omega;\mathbb{H})\,|\,\mathbb{E}\bigl[|\varphi|^p\bigr]<\infty\}.
\end{align*}
Furthermore, we introduce the following spaces of (equivalent classes of) processes:
\begin{align*}
	&L^0_\mathbb{F}(S,T;\mathbb{H}):=\{\varphi:\Omega\times[S,T]\to\mathbb{H}\,|\,\varphi(\cdot)\ \text{is progressively measurable}\},\displaybreak[1]\\
	&L^{p,q}_\mathbb{F}(S,T;\mathbb{H}):=\left\{\varphi(\cdot)\in L^0_\mathbb{F}(S,T;\mathbb{H})\relmiddle|\mathbb{E}\Bigl[\Bigl(\int^T_S|\varphi(s)|^q\,ds\Bigr)^{p/q}\Bigr]<\infty\right\},\displaybreak[1]\\
	&L^p_\mathbb{F}(\Omega;C([S,T];\mathbb{H})):=\left\{\varphi(\cdot)\in L^0_\mathbb{F}(S,T;\mathbb{H})\relmiddle|
		\begin{aligned}
		&\varphi(\cdot)\ \text{has continuous paths and satisfies}\\
		&\mathbb{E}\Bigl[\sup_{s\in[S,T]}|\varphi(s)|^p\Bigr]<\infty
		\end{aligned}\right\}.
\end{align*}
Define $L^p_\mathbb{F}(S,T;\mathbb{H}):=L^{p,p}_\mathbb{F}(S,T;\mathbb{H})$. Note that $L^0_\mathbb{F}(S,T;\mathbb{H})$ is a complete metric space with the metric
\begin{equation*}
	(\varphi_1(\cdot),\varphi_2(\cdot))\mapsto\mathbb{E}\Bigl[\int^T_S\min\{|\varphi_1(s)-\varphi_2(s)|,1\}\,ds\Bigr],\ \varphi_1(\cdot),\varphi_2(\cdot)\in L^0_\mathbb{F}(S,T;\mathbb{H}).
\end{equation*}
The induced topology coincides with the one of convergence in measure $\mathrm{Leb}_{[S,T]}\otimes\mathbb{P}$. That is, for $\varphi(\cdot),\varphi_n(\cdot)\in L^0_\mathbb{F}(S,T;\mathbb{H})$, $n\in\mathbb{N}$, $\varphi_n(\cdot)\to \varphi(\cdot)$ in $L^0_\mathbb{F}(S,T;\mathbb{H})$ if and only if
\begin{equation*}
	\lim_{n\to\infty}\mathbb{E}\Bigl[\int^T_S\1_{\{|\varphi_n(s)-\varphi(s)|\geq\ep\}}\,ds\Bigr]=0,\ \forall\,\ep>0.
\end{equation*}
$L^{p,q}_\mathbb{F}(S,T;\mathbb{H})$ and $L^p_\mathbb{F}(\Omega;C([S,T];\mathbb{H}))$ are Banach spaces with the norms
\begin{align*}
	&\|z(\cdot)\|_{L^{p,q}_\mathbb{F}(S,T;\mathbb{H})}:=\mathbb{E}\Bigl[\Bigl(\int^T_S|z(s)|^q\,ds\Bigr)^{p/q}\Bigr]^{1/p},\ z(\cdot)\in L^{p,q}_\mathbb{F}(S,T;\mathbb{H}),
\shortintertext{and}
	&\|y(\cdot)\|_{L^p_\mathbb{F}(\Omega;C([S,T];\mathbb{H}))}:=\mathbb{E}\Bigl[\sup_{s\in[S,T]}|y(s)|^p\Bigr]^{1/p},\ y(\cdot)\in L^p_\mathbb{F}(\Omega;C([S,T];\mathbb{H})),
\end{align*}
respectively.
\par
The following lemma is standard, but plays an interesting role in our study.

%% Lemma

\begin{lemm}\label{lemma: regularity}
Let $p,q\geq1$, $0\leq S<T<\infty$, and let $\mathbb{H}$ be a Euclidean space.
\begin{enumerate}
\renewcommand{\labelenumi}{(\roman{enumi})}
\item
It holds that
\begin{equation*}
	L^p_\mathbb{F}(\Omega;C([S,T];\mathbb{H}))\subset L^{p,q}_\mathbb{F}(S,T;\mathbb{H})\subset L^{p,1}_\mathbb{F}(S,T;\mathbb{H})\subset L^1_\mathbb{F}(S,T;\mathbb{H})\subset L^0_\mathbb{F}(S,T;\mathbb{H}),
\end{equation*}
and the embeddings are continuous.
\item
Let $\varphi_n(\cdot)\in L^0_\mathbb{F}(S,T;\mathbb{R})$, $n\in\mathbb{N}$, be uniformly bounded and $\lim_{n\to\infty}\varphi_n(\cdot)=0$ in $L^0_\mathbb{F}(S,T;\mathbb{R})$. Then for any $z(\cdot)\in L^{p,1}_\mathbb{F}(S,T;\mathbb{H})$, it holds that $\lim_{n\to\infty}\varphi_n(\cdot)z(\cdot)=0$ in $L^{p,1}_\mathbb{F}(S,T;\mathbb{H})$. 
\end{enumerate}
\end{lemm}

For each $p\geq1$ and Euclidean spaces $\mathbb{H}$ and $\mathbb{G}$, we define
\begin{equation*}
	\mathcal{H}^p_\mathbb{F}(S,T;\mathbb{H}\times\mathbb{G}):=L^p_\mathbb{F}(\Omega;C([S,T];\mathbb{H}))\times L^{p,2}_\mathbb{F}(S,T;\mathbb{G}).
\end{equation*}
Note that $\mathcal{H}^p_\mathbb{F}(S,T;\mathbb{H}\times\mathbb{G})$ is a Banach space with the norm defined by
\begin{equation*}
	\|y(\cdot),z(\cdot)\|_{\mathcal{H}^p_\mathbb{F}(S,T;\mathbb{H}\times\mathbb{G})}:=\mathbb{E}\Bigl[\sup_{s\in[S,T]}|y(s)|^p+\Bigl(\int^T_S|z(s)|^2\,ds\Bigr)^{p/2}\Bigr]^{1/p}
\end{equation*}
for $(y(\cdot),z(\cdot))\in\mathcal{H}^p_\mathbb{F}(S,T;\mathbb{H}\times\mathbb{G})$.

%%%%
%% Subsection
%%%%

\subsection{Stochastic processes with two time-parameters}\label{subsection: two time-parameters}

In order to study EBSVIEs and time-inconsistent stochastic control problems, we have to consider stochastic processes which have two time-parameters. Now we observe such processes rigorously. Let $p,q\geq1$, $0\leq S<T<\infty$, and let $\mathbb{H}$ be a Euclidean space. For $\mathcal{L}_\mathbb{F}(\mathbb{H})=L^0_\mathbb{F}(S,T;\mathbb{H}),\,L^{p,q}_\mathbb{F}(S,T;\mathbb{H}),\,L^p_\mathbb{F}(\Omega;C([S,T];\mathbb{H}))$, we denote by $C([S,T];\mathcal{L}_\mathbb{F}(\mathbb{H}))$ the space of $\mathcal{L}_\mathbb{F}(\mathbb{H})$-valued continuous functions. Furthermore, we define
\begin{equation*}
	\tilde{C}([S,T];\mathcal{L}_\mathbb{F}(\mathbb{H})):=\left\{\varphi:\Omega\times[S,T]^2\to\mathbb{H}\relmiddle|
		\begin{aligned}
		&\varphi\ \text{is measurable},\ \varphi(t,\cdot)\in\mathcal{L}_\mathbb{F}(\mathbb{H}),\ \forall\,t\in[S,T],\\
		&\text{and}\ t\mapsto\varphi(t,\cdot)\in\mathcal{L}_\mathbb{F}(\mathbb{H})\ \text{is continuous}
		\end{aligned}\right\}.
\end{equation*}
Note that each element of $\tilde{C}([S,T];\mathcal{L}_\mathbb{F}(\mathbb{H}))$ is \emph{jointly measurable} on $\Omega\times[S,T]^2$. Let us discuss a relationship between $C([S,T];\mathcal{L}_\mathbb{F}(\mathbb{H}))$ and $\tilde{C}([S,T];\mathcal{L}_\mathbb{F}(\mathbb{H}))$. To do so, we show the following abstract lemma.

%% Lemma

\begin{lemm}\label{lemma: measurable version}
Let $(X,\Sigma,\mu)$ be a finite measure space and $(E,\|\cdot\|_E)$ be a Banach space. Denote by $\mathcal{B}(E)$ the Borel $\sigma$-field with respect to the norm topology of $E$. We denote by $\mathcal{L}(X;E)$ the space of (equivalent classes of) $E$-valued measurable functions, which is a complete metric space with the topology of convergence in measure $\mu$. Then, for any $\varphi\in C([S,T];\mathcal{L}(X;E))$, there exists a jointly measurable function $\tilde{\varphi}:[S,T]\times X\to E$ such that, for any $t\in[S,T]$, $\varphi(t)(x)=\tilde{\varphi}(t,x)$ in $E$ for $\mu$-a.e.\,$x\in X$.
\end{lemm}

The above lemma is standard, but let us prove that fact for self-containedness.

%% Proof

\begin{proof}
Since the function $[S,T]\ni t\mapsto\varphi(t)\in\mathcal{L}(X;E)$ is (uniformly) continuous, there exists a sequence $\{\Pi_n\}_{n\in\mathbb{N}}$ of finite partitions $\Pi_n=\{t^n_k\,|\,k=0,1,\dots,m_n\}$ of $[S,T]$ such that the mesh size of $\Pi_n$ tends to zero as $n\to\infty$, and
\begin{equation*}
	\mu\{x\in X\,|\,\|\varphi(t)(x)-\tilde{\varphi}_n(t,x)\|_E>2^{-n}\}\leq2^{-n},\ \forall\,t\in[S,T],\ \forall\,n\in\mathbb{N},
\end{equation*}
where, for each $n\in\mathbb{N}$, $\tilde{\varphi}_n$ is defined by
\begin{equation*}
	\tilde{\varphi}_n(t,x):=\sum^{m_n}_{k=1}\1_{[t^n_{k-1},t^n_k)}(t)\varphi(t^n_{k-1})(x)+\1_{\{T\}}(t)\varphi(T)(x),\ (t,x)\in[S,T]\times X.
\end{equation*}
Since $\tilde{\varphi}_n:[S,T]\times X\to E$ is jointly measurable for all $n\in\mathbb{N}$, the limit
\begin{equation*}
	\tilde{\varphi}(t,x):=\begin{cases}
		\lim_{n\to\infty}\tilde{\varphi}_n(t,x)\ &\text{if the limit exists},\\
		0\ &\text{otherwise},
	\end{cases}
\end{equation*}
is also jointly measurable. Let $t\in[S,T]$ be fixed. The Borel--Cantelli lemma yields that, for $\mu$-a.e.\,$x\in X$, there exists a number $N(x)\in\mathbb{N}$ such that $\|\varphi(t)(x)-\tilde{\varphi}_n(t,x)\|_E\leq2^{-n}$ for any $n\geq N(x)$, and hence $\varphi(t)(x)=\tilde{\varphi}(t,x)$. This completes the proof.
\end{proof}

Now we show three examples of the above lemma.
\begin{enumerate}
\renewcommand{\labelenumi}{(\roman{enumi})}
\item
Take $(X,\Sigma,\mu)=(\Omega,\mathcal{F}_T,\mathbb{P})$ and $E=\mathbb{H}$. For any $\psi\in C([S,T];L^0_{\mathcal{F}_T}(\Omega;\mathbb{H}))$, there exists a $\tilde{\psi}\in\tilde{C}([S,T];L^0_{\mathcal{F}_T}(\Omega;\mathbb{H}))$ (which is jointly measurable) such that
\begin{equation*}
	\psi(t,\omega)=\tilde{\psi}(t,\omega),\ \text{for}\ \mathbb{P}\text{-a.e.}\,\omega\in\Omega,\ \forall\,t\in[S,T].
\end{equation*}
\item
Take $(X,\Sigma,\mu)=(\Omega\times[S,T],\mathcal{P},\mathrm{Leb}_{[S,T]}\otimes\mathbb{P})$ and $E=\mathbb{H}$, where $\mathcal{P}$ is the progressive $\sigma$-field. For any $Z\in C([S,T];L^0_\mathbb{F}(S,T;\mathbb{H}))$, there exists a $\tilde{Z}\in\tilde{C}([S,T];L^0_\mathbb{F}(S,T;\mathbb{H}))$ (which is jointly measurable) such that
\begin{equation}\label{version Z}
	Z(t,s,\omega)=\tilde{Z}(t,s,\omega),\ \text{for}\ \mathrm{Leb}_{[S,T]}\otimes\mathbb{P}\text{-a.e.}\,(s,\omega)\in[S,T]\times\Omega,\ \forall\,t\in[S,T].
\end{equation}
\item
Take $(X,\Sigma,\mu)=(\Omega,\mathcal{F},\mathbb{P})$ and $E=C([S,T];\mathbb{H})$. For any $Y\in C([S,T];L^p_\mathbb{F}(\Omega;C([S,T];\mathbb{H})))$, there exists a $\tilde{Y}\in\tilde{C}([S,T];L^p_\mathbb{F}(\Omega;C([S,T];\mathbb{H})))$ (which is jointly measurable) such that
\begin{equation*}
	Y(t,s,\omega)=\tilde{Y}(t,s,\omega),\ \forall\,s\in[S,T],\ \text{for}\ \mathbb{P}\text{-a.e.}\,\omega\in\Omega,\ \forall\,t\in[S,T].
\end{equation*}
\end{enumerate}
In the following, for each element of $C([S,T];\mathcal{L}(X;E))$, we always consider a jointly measurable version in the above sense, and we identify ``$C$'' and ``$\tilde{C}$''.
\par
We define
\begin{equation*}
	C_b([S,T];L^0_\mathbb{F}(S,T;\mathbb{H})):=\{\varphi(\cdot,\cdot)\in C([S,T];L^0_\mathbb{F}(S,T;\mathbb{H}))\,|\,\varphi\ \text{is uniformly bounded}\}.
\end{equation*}
Lastly, for each $p\geq1$ and Euclidean spaces $\mathbb{H}$ and $\mathbb{G}$, we define
\begin{equation*}
	\mathfrak{H}^p_\mathbb{F}(S,T;\mathbb{H}\times\mathbb{G}):=C([S,T];L^p_\mathbb{F}(\Omega;C([S,T];\mathbb{H})))\times C([S,T];L^{p,2}_\mathbb{F}(S,T;\mathbb{G})).
\end{equation*}
Note that $\mathfrak{H}^p_\mathbb{F}(S,T;\mathbb{H}\times\mathbb{G})$ is a Banach space with the norm defined by
\begin{equation*}
	\|y(\cdot,\cdot),z(\cdot,\cdot)\|_{\mathfrak{H}^p_\mathbb{F}(S,T;\mathbb{H}\times\mathbb{G})}:=\sup_{t\in[S,T]}\mathbb{E}\Bigl[\sup_{s\in[S,T]}|y(t,s)|^p+\Bigl(\int^T_S|z(t,s)|^2\,ds\Bigr)^{p/2}\Bigr]^{1/p}
\end{equation*}
for $(y(\cdot,\cdot),z(\cdot,\cdot))\in \mathfrak{H}^p_\mathbb{F}(S,T;\mathbb{H}\times\mathbb{G})$.

%% Remark

\begin{rem}
The ``diagonal process'' of a process with two time-parameters is crucial in the studies of EBSVIEs and time-inconsistent control problems. Let us remark on that.
\begin{itemize}
\item
For each $Y(\cdot,\cdot)\in C([S,T];L^p_\mathbb{F}(\Omega;C([S,T];\mathbb{H})))$, the diagonal process $(Y(s,s))_{s\in[S,T]}$ is progressively measurable, and the map $[S,T]\ni s\mapsto Y(s,s)\in L^p_{\mathcal{F}_T}(\Omega;\mathbb{H})$ is continuous. Moreover, it can be easily shown that, for any $t\in[S,T)$,
\begin{equation*}
	\lim_{\ep\downarrow0}\frac{1}{\ep}\mathbb{E}\Bigl[\int^{t+\ep}_t\bigl|Y(t,s)-Y(s,s)\bigr|^p\,ds\Bigr]=\lim_{\ep\downarrow0}\frac{1}{\ep}\mathbb{E}\Bigl[\int^{t+\ep}_t\bigl|Y(t,s)-Y(t,t)\bigr|^p\,ds\Bigr]=0.
\end{equation*}
\item
The case of the space $C([S,T];L^{p,q}_\mathbb{F}(S,T;\mathbb{H}))$ is more delicate. In fact, in some previous works on time-inconsistent stochastic control problems, the diagonal process $(Z(s,s))_{s\in[S,T]}$ of $Z(\cdot,\cdot)\in C([S,T];L^{p,q}_\mathbb{F}(S,T;\mathbb{H}))$ was used without rigorous discussions. However, such a process is not well-defined in general. Indeed, for two elements $Z_1(\cdot,\cdot)$ and $Z_2(\cdot,\cdot)$ in $C([S,T];L^{p,q}_\mathbb{F}(S,T;\mathbb{H}))$ such that
\begin{equation*}
	Z_1(t,s,\omega)=Z_2(t,s,\omega),\ \text{for}\ \mathrm{Leb}_{[S,T]}\otimes\mathbb{P}\text{-a.e.}\,(s,\omega)\in[S,T]\times\Omega,\ \forall\,t\in[S,T],
\end{equation*}
the equality $Z_1(s,s,\omega)=Z_2(s,s,\omega)$ for $\mathrm{Leb}_{[S,T]}\otimes\mathbb{P}$-a.e.\,$(s,\omega)\in[S,T]\times\Omega$ does not hold in general. Moreover, the next example shows that the limit
\begin{equation*}
	\lim_{\ep\downarrow0}\frac{1}{\ep}\mathbb{E}\Bigl[\int^{t+\ep}_tZ(t,s)\,ds\Bigr],\ t\in[S,T),
\end{equation*}
does not exist in general. 
\end{itemize}
\end{rem}

%% Example

\begin{exam}\label{example: diagonal}
Let $r>1$ be fixed. Define two (deterministic) processes $Z_1(t,s)$ and $Z_2(t,s)$ for $(t,s)\in[0,1]^2$ by
\begin{equation*}
	Z_1(t,s)=\begin{cases}
		(s-t)^{-1/r}\ &\text{if}\ s>t,\\
		0\ &\text{if}\ s\leq t,
	\end{cases}
	\ \text{and}\ 
	Z_2(t,s)=\begin{cases}
		(s-t)^{-1/r}\ &\text{if}\ s>t,\\
		1\ &\text{if}\ s=t,\\
		0\ &\text{if}\ s<t.
	\end{cases}
\end{equation*}
It can be easily shown that both $Z_1$ and $Z_2$ are jointly measurable, $Z_1(t,s)=Z_2(t,s)$ for a.e.\,$s\in[0,1]$ for any $t\in[0,1]$, and $Z_1(\cdot,\cdot),Z_2(\cdot,\cdot)\in C([0,1];L^q(0,1;\mathbb{R}))$ for any $q\in[1,r)$, but $Z_1(s,s)\neq Z_2(s,s)$ for any $s\in[0,1]$. Furthermore,
\begin{equation*}
	\frac{1}{\ep}\int^{t+\ep}_tZ_1(t,s)\,ds\Bigl(=\frac{1}{\ep}\int^{t+\ep}_tZ_2(t,s)\,ds\Bigr)=\frac{r}{r-1}\ep^{-1/r}\overset{\ep\downarrow0}{\longrightarrow}\infty,\ \forall\,t\in[0,1).
\end{equation*}
\end{exam}

Thus, the case of the space $C([S,T];L^{p,q}_\mathbb{F}(S,T;\mathbb{H}))$ needs a more careful observation. Firstly, let us define a property which the ``diagonal process'' of $Z(\cdot,\cdot)\in C([S,T];L^1_\mathbb{F}(S,T;\mathbb{H}))$ should satisfy, in view of applications to time-inconsistent stochastic control problems.

%% Definition

\begin{defi}
Let $Z(\cdot,\cdot)\in C([S,T];L^1_\mathbb{F}(S,T;\mathbb{H}))$ be given. We say that a process $\mathcal{Z}(\cdot)\in L^1_\mathbb{F}(S,T;\mathbb{H})$ satisfies Property~(D) with respect to $Z(\cdot,\cdot)$ if it holds that
\begin{equation*}
	\lim_{\ep\downarrow0}\frac{1}{\ep}\mathbb{E}\Bigl[\int^{t+\ep}_t|Z(t,s)-\mathcal{Z}(s)|\,ds\Bigr]=0,\ \forall\,t\in[S,T).
\end{equation*}
\end{defi}

Here ``D'' is named after ``Diagonal''. Note that the above definition does not depend on the choice of a ``version'' (in the sense of \eqref{version Z}) of $Z(\cdot,\cdot)$. That is, if $\mathcal{Z}(\cdot)\in L^1_\mathbb{F}(S,T;\mathbb{H})$ satisfies Property~(D) with respect to $Z(\cdot,\cdot)$, then for any $\tilde{Z}(\cdot,\cdot)\in C([S,T];L^1_\mathbb{F}(S,T;\mathbb{H}))$ such that $\tilde{Z}(t,s)=Z(t,s)$ for $\mathrm{Leb}_{[S,T]}\otimes\mathbb{P}$-a.e.\,$(s,\omega)\in[S,T]\times\Omega$, $\forall\,t\in[S,T]$, the process $\mathcal{Z}(\cdot)$ also satisfies Property~(D) with respect to $\tilde{Z}(\cdot,\cdot)$. Furthermore, the following lemma shows that, for each $Z(\cdot,\cdot)\in C([S,T];L^1_\mathbb{F}(S,T;\mathbb{H}))$, the process $\mathcal{Z}(\cdot)\in L^1_\mathbb{F}(S,T;\mathbb{H})$ satisfying Property~(D) with respect to $Z(\cdot,\cdot)$ is, if it exists, unique.

%% Lemma

\begin{lemm}
Let $Z(\cdot,\cdot)\in C([S,T];L^{1}_\mathbb{F}(S,T;\mathbb{H}))$ be given. Assume that both two processes $\mathcal{Z}_1(\cdot),\mathcal{Z}_2(\cdot)\in L^1_\mathbb{F}(S,T;\mathbb{H})$ satisfy Property~(D) with respect to $Z(\cdot,\cdot)$. Then it holds that $\mathcal{Z}_1(s)=\mathcal{Z}_2(s)$ for $\mathrm{Leb}_{[S,T]}\otimes\mathbb{P}$-a.e.\,$(s,\omega)\in[S,T]\times\Omega$.
\end{lemm}

%% Proof

\begin{proof}
Since the function $s\mapsto\mathbb{E}\bigl[|\mathcal{Z}_1(s)-\mathcal{Z}_2(s)|\bigr]$ is in $L^1(S,T;\mathbb{R})$, by the Lebesgue differentiation theorem, it holds that
\begin{equation*}
	\lim_{\ep\downarrow0}\frac{1}{\ep}\int^{t+\ep}_t\mathbb{E}\bigl[|\mathcal{Z}_1(s)-\mathcal{Z}_2(s)|\bigr]\,ds=\mathbb{E}\bigl[|\mathcal{Z}_1(t)-\mathcal{Z}_2(t)|\bigr]
\end{equation*}
for a.e.\,$t\in[S,T)$. On the other hand, for any $t\in[S,T)$, it holds that
\begin{align*}
	&\frac{1}{\ep}\int^{t+\ep}_t\mathbb{E}\bigl[|\mathcal{Z}_1(s)-\mathcal{Z}_2(s)|\bigr]\,ds\\
	&\leq\frac{1}{\ep}\mathbb{E}\Bigl[\int^{t+\ep}_t|\mathcal{Z}_1(s)-Z(t,s)|\,ds\Bigr]+\frac{1}{\ep}\mathbb{E}\Bigl[\int^{t+\ep}_t|Z(t,s)-\mathcal{Z}_2(s)|\,ds\Bigr]\overset{\ep\downarrow0}{\longrightarrow}0.
\end{align*}
Thus, we get $\mathbb{E}\bigl[|\mathcal{Z}_1(t)-\mathcal{Z}_2(t)|\bigr]=0$ for a.e.\,$t\in[S,T]$. This implies that $\mathcal{Z}_1(s)=\mathcal{Z}_2(s)$ for $\mathrm{Leb}_{[S,T]}\otimes\mathbb{P}$-a.e.\,$(s,\omega)\in[S,T]\times\Omega$.
\end{proof}

Then, when does the process satisfying Property~(D) exist? Example~\ref{example: diagonal} shows that there does not exist such processes in general even in the case of deterministic processes.
\par
First, assume that there exist a process $\varphi(\cdot)\in L^1_\mathbb{F}(\Omega;C([S,T];\mathbb{R}))$ and a uniformly bounded process $z(\cdot)\in L^0_\mathbb{F}(S,T;\mathbb{H})$ such that $Z(t,s)=\varphi(t)z(s)$ for $\mathrm{Leb}_{[t,T]}\otimes\mathbb{P}$-a.e.\,$(s,\omega)\in[t,T]\times\Omega$, $\forall\,t\in[S,T]$. Then it can be easily shown that the process $\mathcal{Z}(s):=\varphi(s)z(s)$, $s\in[S,T]$, satisfies property~(D) with respect to $Z(\cdot,\cdot)$.
This technique arises in the literature of time-inconsistent stochastic linear-quadratic control problems; see for example \cite{a_Hu-Jin-Zhou_17}. However, in most control problems, we cannot use this method due to the generality of the process $Z(\cdot,\cdot)$. We investigate another approach to deal with a general $Z(\cdot,\cdot)$ by imposing a regularity assumption on the map $t\mapsto Z(t,\cdot)$.
\par
Let $p,q\geq1$ be fixed. Suppose that $Z(\cdot,\cdot)$ is in $C^1([S,T];L^{p,q}_\mathbb{F}(S,T;\mathbb{H}))$, that is, $[S,T]\ni t\mapsto Z(t,\cdot)$ is continuously differentiable as an $L^{p,q}_\mathbb{F}(S,T;\mathbb{H})$-valued function. Then there exists a process $\partial_tZ(\cdot,\cdot)\in C([S,T];L^{p,q}_\mathbb{F}(S,T;\mathbb{H}))$ such that
\begin{equation*}
	Z(t_1,\cdot)-Z(t_2,\cdot)=\int^{t_1}_{t_2}\partial_tZ(\tau,\cdot)\,d\tau,\ \forall\,t_1,t_2\in[S,T],
\end{equation*}
where the integral in the right-hand side is the Bochner integral on the Banach space $L^{p,q}_\mathbb{F}(S,T;\mathbb{H})$. By Fubini's theorem, for $\mathrm{Leb}_{[S,T]}\otimes\mathbb{P}$-a.e.\,$(s,\omega)\in[S,T]\times\Omega$, the function $\tau\mapsto \partial_tZ(\tau,s,\omega)$ is well-defined as an element of $L^1(S,T;\mathbb{H})$. Furthermore, for any $t_1,t_2\in[S,T]$, we have
\begin{equation}\label{Z integral representation}
	Z(t_1,s,\omega)-Z(t_2,s,\omega)=\Bigl(\int^{t_1}_{t_2}\partial_t Z(\tau,\cdot,\cdot)\,d\tau\Bigr)(s,\omega)=\int^{t_1}_{t_2}\partial_t Z(\tau,s,\omega)\,d\tau,
\end{equation}
for $\mathrm{Leb}_{[S,T]}\otimes\mathbb{P}\text{-a.e.}\,(s,\omega)\in[S,T]\times\Omega$ (where the null set may depend on $t_1$ and $t_2$). Now we define a progressively measurable process $\mathrm{Diag}[Z](\cdot)$ by
\begin{equation}\label{Z diagonal}
	\mathrm{Diag}[Z](s,\omega):=
	\begin{cases}
		Z(S,s,\omega)+\int^s_S\partial_tZ(\tau,s,\omega)\,d\tau\ &\text{if}\ \partial_tZ(\cdot,s,\omega)\in L^1(S,T;\mathbb{H}),\\
		0\ &\text{otherwise}.
	\end{cases}
\end{equation}
By \eqref{Z integral representation}, for any $t\in[S,T]$, it holds that
\begin{equation}\label{Z diagonal equality}
	\mathrm{Diag}[Z](s)=Z(t,s)+\int^s_t\partial_tZ(\tau,s)\,d\tau,\ \text{for}\ \mathrm{Leb}_{[S,T]}\otimes\mathbb{P}\text{-a.e.}\,(s,\omega)\in[S,T]\times\Omega.
\end{equation}
We emphasize that, in the above expression, the $\mathrm{Leb}_{[S,T]}\otimes\mathbb{P}$-null set is allowed to depend on $t$. Now let us show important properties of $\mathrm{Diag}[Z](\cdot)$.

%% Lemma

\begin{lemm}\label{lemma: Z diagonal}
For a given $Z(\cdot,\cdot)\in C^1([S,T];L^{p,q}_\mathbb{F}(S,T;\mathbb{H}))$, define $\mathrm{Diag}[Z](\cdot)\in L^0_\mathbb{F}(S,T;\mathbb{H})$ by \eqref{Z diagonal}. Then the following hold.
\begin{enumerate}
\renewcommand{\labelenumi}{(\roman{enumi})}
\item
$\mathrm{Diag}[Z](\cdot)\in L^{p,q}_\mathbb{F}(S,T;\mathbb{H})$. Moreover, the following estimate holds:
\begin{equation*}
	\|\mathrm{Diag}[Z](\cdot)\|_{L^{p,q}_\mathbb{F}(S,T;\mathbb{H})}\leq\|Z(S,\cdot)\|_{L^{p,q}_\mathbb{F}(S,T;\mathbb{H})}+(T-S)\sup_{t\in[S,T]}\|\partial_tZ(t,\cdot)\|_{L^{p,q}_\mathbb{F}(S,T;\mathbb{H})}.
\end{equation*}
\item
For any $1\leq p'\leq p$ and $1\leq q'\leq q$, it holds that
\begin{equation*}
	\mathbb{E}\Bigl[\Bigl(\int^{t+\ep}_t\bigl|Z(t,s)-\mathrm{Diag}[Z](s)\bigr|^{q'}\,ds\Bigr)^{p'/q'}\Bigr]^{1/p'}=o(\ep^{1+1/q'-1/q}),\ \forall\,t\in[S,T).
\end{equation*}
In particular, $\mathrm{Diag}[Z](\cdot)$ is the (unique) process satisfying Property~(D) with respect to $Z(\cdot,\cdot)$. 
\end{enumerate}
\end{lemm}

%% Proof

\begin{proof}
(\rnum{1})\ \,By using Minkowski's integral inequality (the integral form of Minkowski's inequality) repeatedly, we see that
\begin{align*}
	\mathbb{E}\Bigl[\Bigl(\int^T_S\Bigl|\int^s_S\partial_tZ(\tau,s)\,d\tau\Bigr|^q\,ds\Bigr)^{p/q}\Bigr]^{1/p}&\leq\mathbb{E}\Bigl[\Bigl(\int^T_S\Bigl(\int^T_S|\partial_tZ(\tau,s)|\,d\tau\Bigr)^q\,ds\Bigr)^{p/q}\Bigr]^{1/p}\\
	&\leq\mathbb{E}\Bigl[\Bigl(\int^T_S\Bigl(\int^T_S|\partial_tZ(\tau,s)|^q\,ds\Bigr)^{1/q}\,d\tau\Bigr)^p\Bigr]^{1/p}\displaybreak[1]\\
	&\leq\int^T_S\mathbb{E}\Bigl[\Bigl(\int^T_S|\partial_tZ(\tau,s)|^q\,ds\Bigr)^{p/q}\Bigr]^{1/p}\,d\tau\\
	&\leq(T-S)\sup_{\tau\in[S,T]}\|\partial_tZ(\tau,\cdot)\|_{L^{p,q}_\mathbb{F}(S,T;\mathbb{H})}<\infty.
\end{align*}
From this estimate and the fact that $Z(S,\cdot)\in L^{p,q}_\mathbb{F}(S,T;\mathbb{H})$, we obtain the assertions in (\rnum{1}).\\
(\rnum{2})\ \,Without loss of generality we may assume that $p'=p$. Let $t\in[S,T)$ be fixed. Noting the equality~\ref{Z diagonal equality}, by using Minkowski's integral inequality and H\"{o}lder's inequality, we see that, for any $\ep\in(0,T-t]$,
\begin{align*}
	&\mathbb{E}\Bigl[\Bigl(\int^{t+\ep}_t\bigl|Z(t,s)-\mathrm{Diag}[Z](s)\bigr|^{q'}\,ds\Bigr)^{p/q'}\Bigr]^{1/p}\\
	&=\mathbb{E}\Bigl[\Bigl(\int^{t+\ep}_t\Bigl|\int^s_t\partial_tZ(\tau,s)\,d\tau\Bigr|^{q'}\,ds\Bigr)^{p/q'}\Bigr]^{1/p}\displaybreak[1]\\
	&\leq\mathbb{E}\Bigl[\Bigl(\int^{t+\ep}_t\Bigl(\int^{t+\ep}_\tau|\partial_tZ(\tau,s)|^{q'}\,ds\Bigr)^{1/q'}\,d\tau\Bigr)^p\Bigr]^{1/p}\displaybreak[1]\\
	&\leq\int^{t+\ep}_t\mathbb{E}\Bigl[\Bigl(\int^{t+\ep}_\tau|\partial_tZ(\tau,s)|^{q'}\,ds\Bigr)^{p/q'}\Bigr]^{1/p}\,d\tau\displaybreak[1]\\
	&\leq\ep^{1/q'-1/q}\int^{t+\ep}_t\mathbb{E}\Bigl[\Bigl(\int^{t+\ep}_t|\partial_tZ(\tau,s)|^q\,ds\Bigr)^{p/q}\Bigr]^{1/p}\,d\tau.
\end{align*}
Furthermore, by Minkowski's inequality, we obtain
\begin{align*}
	&\int^{t+\ep}_t\mathbb{E}\Bigl[\Bigl(\int^{t+\ep}_t|\partial_tZ(\tau,s)|^q\,ds\Bigr)^{p/q}\Bigr]^{1/p}\,d\tau\\
	&\leq\int^{t+\ep}_t\Bigl\{\mathbb{E}\Bigl[\Bigl(\int^{t+\ep}_t|\partial_tZ(t,s)|^q\,ds\Bigr)^{p/q}\Bigr]^{1/p}+\mathbb{E}\Bigl[\Bigl(\int^{t+\ep}_t|\partial_tZ(\tau,s)-\partial_tZ(t,s)|^q\,ds\Bigr)^{p/q}\Bigr]^{1/p}\Bigr\}\,d\tau\\
	&\leq\ep\Bigl\{\mathbb{E}\Bigl[\Bigl(\int^{t+\ep}_t|\partial_tZ(t,s)|^q\,ds\Bigr)^{p/q}\Bigr]^{1/p}+\sup_{\tau\in[t,t+\ep]}\|\partial_tZ(\tau,\cdot)-\partial_tZ(t,\cdot)\|_{L^{p,q}_\mathbb{F}(S,T;\mathbb{H})}\Bigr\}.
\end{align*}
Since $\partial_tZ(t,\cdot)\in L^{p,q}_\mathbb{F}(S,T;\mathbb{H})$, we have $\lim_{\ep\downarrow0}\mathbb{E}\Bigl[\Bigl(\int^{t+\ep}_t|\partial_tZ(t,s)|^q\,ds\Bigr)^{p/q}\Bigr]^{1/p}=0$. On the other hand, since the map $[S,T]\ni\tau\mapsto\partial_tZ(\tau,\cdot)\in L^{p,q}_\mathbb{F}(S,T;\mathbb{H})$ is continuous, we have $\lim_{\ep\downarrow0}\sup_{\tau\in[t,t+\ep]}\|\partial_tZ(\tau,\cdot)-\partial_tZ(t,\cdot)\|_{L^{p,q}_\mathbb{F}(S,T;\mathbb{H})}=0$.
Thus, the first assertion in (\rnum{2}) holds. In particular, if we take $p'=q'=1$, then it holds that
\begin{equation*}
	\mathbb{E}\Bigl[\int^{t+\ep}_t|Z(t,s)-\mathrm{Diag}[Z](s)|\,ds\Bigr]=o(\ep^{2-1/q}),\ \forall\,t\in[S,T).
\end{equation*}
This implies that $\mathrm{Diag}[Z](\cdot)$ satisfies Property~(D) with respect to $Z(\cdot,\cdot)$.
\end{proof}

%% Remark

\begin{rem}
\begin{enumerate}
\renewcommand{\labelenumi}{(\roman{enumi})}
\item
We emphasize that, for a given $Z(\cdot,\cdot)\in C^1([S,T];L^{p,q}_\mathbb{F}(S,T;\mathbb{H}))$, the naive definition ``$Z(s,s):=Z(s,u)|_{u=s}$'' still depends on the choice of a ``version'' (in the sense of \eqref{version Z}) of $Z(\cdot,\cdot)$, while Property~(D) does not. The above lemma implies that, if we define $\tilde{Z}(\cdot,\cdot)\in C^1([S,T];L^{p,q}_\mathbb{F}(S,T;\mathbb{H}))$ by
\begin{equation*}
	\tilde{Z}(t,s,\omega):=
	\begin{cases}
		Z(S,s,\omega)+\int^t_S\partial_tZ(\tau,s,\omega)\,d\tau\ &\text{if}\ \partial_tZ(\cdot,s,\omega)\in L^1(S,T;\mathbb{H}),\\
		0\ &\text{otherwise},
	\end{cases}
\end{equation*}
then it is a ``version'' of $Z(\cdot,\cdot)$ such that the process $\tilde{Z}(s,s)=\mathrm{Diag}[Z](s)$, $s\in[S,T]$, is well-defined and satisfies Property~(D) with respect to $Z(\cdot,\cdot)$. Furthermore, the above discussions are consistent with the arguments of the recent work by Hern\'{a}ndez--Possama\"{i}~\cite{a_Hernandez-Possamai_20'}.
\item
By the same arguments as in the above proof, we can also show that
\begin{align*}
	&\mathbb{E}\Bigl[\Bigl(\int^t_{t-\ep}\bigl|Z(t,s)-\mathrm{Diag}[Z](s)\bigr|^{q'}\,ds\Bigr)^{p'/q'}\Bigr]^{1/p'}=o(\ep^{1+1/q'-1/q}),\ \forall\,t\in(S,T],
\shortintertext{and}
	&\mathbb{E}\Bigl[\Bigl(\int^{t+\ep}_{t-\ep}\bigl|Z(t,s)-\mathrm{Diag}[Z](s)\bigr|^{q'}\,ds\Bigr)^{p'/q'}\Bigr]^{1/p'}=o(\ep^{1+1/q'-1/q}),\ \forall\,t\in(S,T),
\end{align*}
for any $1\leq p'\leq p$ and $1\leq q'\leq q$.
\end{enumerate}
\end{rem}

%%%%
%% Subsection
%%%%

\subsection{Known results for BSDEs}

For $0\leq S<T<\infty$, consider the following BSDE on $[S,T]$:
\begin{equation}\label{BSDE}
	Y(s)=\psi+\int^T_sg(r,Y(r),Z(r))\,dr-\int^T_sZ(r)\,dW(r),\ s\in[S,T],
\end{equation}
where $(\psi,g)$ satisfies the following assumptions:

%% Assumption

\begin{assum}\label{assumption: BSDE}
Fix $p\geq2$.
\begin{enumerate}
\renewcommand{\labelenumi}{(\roman{enumi})}
\item
$\psi\in L^p_{\mathcal{F}_T}(\Omega;\mathbb{R}^m)$.
\item
$g:\Omega\times[S,T]\times\mathbb{R}^m\times\mathbb{R}^{m\times d}\to\mathbb{R}^m$ is a measurable map such that
\begin{itemize}
\item
The process $(g(s,y,z))_{s\in[S,T]}$ is progressively measurable for each $y\in\mathbb{R}^m$ and $z\in\mathbb{R}^{m\times d}$;
\item
$g(\cdot,0,0)\in L^{p,1}_\mathbb{F}(S,T;\mathbb{R}^m)$;
\item
There exists a constant $L>0$ such that, for $\mathrm{Leb}_{[S,T]}\otimes\mathbb{P}$-a.e.\,$(s,\omega)\in[S,T]\times\Omega$, it holds that
\begin{equation*}
	|g(s,y_1,z_1)-g(s,y_2,z_2)|\leq L(|y_1-y_2|+|z_1-z_2|)
\end{equation*}
for any $y_1,y_2\in\mathbb{R}^m$ and $z_1,z_2\in\mathbb{R}^{m\times d}$.
\end{itemize}
\end{enumerate}
\end{assum}

We say that a pair $(Y(\cdot),Z(\cdot))$ is an $L^p$-adapted solution of BSDE~\eqref{BSDE} if $(Y(\cdot),Z(\cdot))\in\mathcal{H}^p_\mathbb{F}(S,T;\mathbb{R}^m\times\mathbb{R}^{m\times d})$ and the equality~\eqref{BSDE} holds a.s. for any $s\in[S,T]$. The following fact is well-known; see for example~\cite{b_Zhang_17}.

%% Lemma

\begin{lemm}\label{lemma: BSDE}
Under Assumption~\ref{assumption: BSDE}, there exists a unique $L^p$-adapted solution $(Y(\cdot),Z(\cdot))\in\mathcal{H}^p_\mathbb{F}(S,T;\mathbb{R}^m\times\mathbb{R}^{m\times d})$ of BSDE~\eqref{BSDE}, and the following estimate holds:
\begin{equation*}
	\mathbb{E}\Bigl[\sup_{s\in[S,T]}|Y(s)|^p+\Bigl(\int^T_S|Z(s)|^2\,ds\Bigr)^{p/2}\Bigr]\leq C\mathbb{E}\Bigl[|\psi|^p+\Bigl(\int^T_S|g(s,0,0)|\,ds\Bigr)^p\Bigr].
\end{equation*}
For $i=1,2$, let $(\psi_i,g_i)$ satisfy Assumption~\ref{assumption: BSDE} and $(Y_i(\cdot),Z_i(\cdot))\in\mathcal{H}^p_\mathbb{F}(S,T;\mathbb{R}^m\times\mathbb{R}^{m\times d})$ be the unique $L^p$-adapted solution of BSDE~\eqref{BSDE} corresponding to $(\psi_i,g_i)$, respectively. Then it holds that
\begin{equation*}
\begin{split}
	&\mathbb{E}\Bigl[\sup_{s\in[S,T]}|Y_1(s)-Y_2(s)|^p+\Bigl(\int^T_S|Z_1(s)-Z_2(s)|^2\,ds\Bigr)^{p/2}\Bigr]\\
	&\leq C\mathbb{E}\Bigl[|\psi_1-\psi_2|^p+\Bigl(\int^T_S|g_1(s,Y_1(s),Z_1(s))-g_2(s,Y_1(s),Z_1(s))|\,ds\Bigr)^p\Bigr].
\end{split}
\end{equation*}
\end{lemm}

%%%%%%%%%%%%%%%
%% Section
%%%%%%%%%%%%%%%

\section{Well-posedness and regularity of EBSVIEs}\label{section: EBSVIE}

Consider EBSVIE~\eqref{EBSVIE}. We impose the following assumptions:

%% Assumption

\begin{assum}\label{assumption: EBSVIE}
Fix $p\geq2$.
\begin{enumerate}
\renewcommand{\labelenumi}{(\roman{enumi})}
\item
$\psi(\cdot)\in C([S,T];L^p_{\mathcal{F}_T}(\Omega;\mathbb{R}^m))$.
\item
$g:\Omega\times[S,T]^2\times\mathbb{R}^m\times\mathbb{R}^m\times\mathbb{R}^{m\times d}\to\mathbb{R}^m$ is a measurable map such that
\begin{itemize}
\item
The process $(g(t,s,\eta,y,z))_{s\in[S,T]}$ is progressively measurable for each $t\in[S,T]$, $\eta,y\in\mathbb{R}^m$ and $z\in\mathbb{R}^{m\times d}$;
\item
$\sup_{t\in[S,T]}\mathbb{E}\Bigl[\Bigl(\int^T_S|g(t,s,0,0,0)|\,ds\Bigr)^p\Bigr]<\infty$;
\item
There exists a constant $L>0$ such that, for any $t\in[S,T]$, for $\mathrm{Leb}_{[S,T]}\otimes\mathbb{P}$-a.e.\,$(s,\omega)\in[S,T]\times\Omega$, it holds that
\begin{equation*}
	|g(t,s,\eta_1,y_1,z_1)-g(t,s,\eta_2,y_2,z_2)|\leq L(|\eta_1-\eta_2|+|y_1-y_2|+|z_1-z_2|)
\end{equation*}
for any $\eta_1,\eta_2,y_1,y_2\in\mathbb{R}^m$ and $z_1,z_2\in\mathbb{R}^{m\times d}$;
\item
There exist two processes
\begin{equation*}
	k(\cdot,\cdot)\in C([S,T];L^{p,1}_\mathbb{F}(S,T;\mathbb{H}))\ \text{and}\ l(\cdot,\cdot)\in C_b([S,T];L^0_\mathbb{F}(S,T;\mathbb{G}))
\end{equation*}
with Euclidean spaces $\mathbb{H}$ and $\mathbb{G}$ such that, for any $t_1,t_2\in[S,T]$, for $\mathrm{Leb}_{[S,T]}\otimes\mathbb{P}$-a.e.\,$(s,\omega)\in[S,T]\times\Omega$, it holds that
\begin{equation}\label{weak continuity}
\begin{split}
	&|g(t_1,s,\eta,y,z)-g(t_2,s,\eta,y,z)|\\
	&\leq|k(t_1,s)-k(t_2,s)|+|l(t_1,s)-l(t_2,s)|(|\eta|+|y|+|z|)
\end{split}
\end{equation}
for any $\eta,y\in\mathbb{R}^m$ and $z\in\mathbb{R}^{m\times d}$.
\end{itemize}
\end{enumerate}
\end{assum}

%% Remark

\begin{rem}\label{remark: t-continuity}
Compared with \cite{a_Wang_20} and other previous researches on BSVIEs, the last assumption~\eqref{weak continuity} on the continuity of the generator $g$ with respect to $t\in[S,T]$ is new and weaker. In the literature, the continuity of $g$ with respect to $t\in[S,T]$ is assumed to be pointwise, that is,
\begin{equation}\label{pointwise continuity}
	|g(t_1,s,\eta,y,z)-g(t_2,s,\eta,y,z)|\leq\rho(|t_1-t_2|)(1+|\eta|+|y|+|z|)
\end{equation}
for some modulus of continuity $\rho:[0,\infty)\to[0,\infty)$. However, the EBSVIEs arising in Section~\ref{section: control problem} do not satisfy the continuity assumption~\eqref{pointwise continuity}, and hence they are beyond the literature. This is why we introduced the weaker continuity assumption with respect to $t$ in Assumption~\ref{assumption: EBSVIE}.
\end{rem}

We now introduce a concept of the solution of EBSVIE~\eqref{EBSVIE}.

%% Definition

\begin{defi}
We say that a pair $(Y(\cdot,\cdot),Z(\cdot,\cdot))$ is an $L^p$-adapted C-solution of EBSVIE~\eqref{EBSVIE} if $(Y(\cdot,\cdot),Z(\cdot,\cdot))\in \mathfrak{H}^p_\mathbb{F}(S,T;\mathbb{R}^m\times\mathbb{R}^{m\times d})$ and the equality~\eqref{EBSVIE} holds a.s. for any $s\in[S,T]$ and $t\in[S,T]$.
\end{defi}

%% Remark

\begin{rem}\label{remark: EBSVIE}
\begin{enumerate}
\renewcommand{\labelenumi}{(\roman{enumi})}
\item
Unlike \cite{a_Wang_20}, we consider the values not only on $S\leq t\leq s\leq T$ but also on $S\leq s<t\leq T$, because it clarifies the discussions for regularity of solutions with respect to $t\in[S,T]$. The term ``C'' is named after the continuity of the solution with respect to $t\in[S,T]$ (in the $L^p$-sense). The above definition of solutions is a generalization of the concept of adapted C-solutions of Type-\Rnum{1} BSVIEs introduced in \cite{a_Wang-Zhang_17} to EBSVIEs.
\item
If the generator $g(t,s,\eta,y,z)$ is independent of $\eta$, then EBSVIE~\eqref{EBSVIE} reduces to the (decoupled) family of BSDEs for $(Y(t,\cdot),Z(t,\cdot))$ on $[S,T]$ parametrized by $t\in[S,T]$. 
\item
If the generator $g(t,s,\eta,y,z)$ is independent of $y$, then EBSVIE~\eqref{EBSVIE} reduces to the following Type-\Rnum{1} BSVIE:
\begin{equation}\label{BSVIE}
	\eta(t)=\psi(t)+\int^T_tg(t,s,\eta(s),\zeta(t,s))\,ds-\int^T_t\zeta(t,s)\,dW(s),\ t\in[S,T].
\end{equation}
In this case, the $L^p$-adapted C-solution $(Y(\cdot,\cdot),Z(\cdot,\cdot))$ of EBSVIE~\eqref{EBSVIE} corresponds to the following. For each $t\in[S,T]$,
\begin{equation*}
	\begin{cases}
		Y(t,s)=\mathbb{E}_s\Bigl[\psi(t)+\int^T_sg(t,r,\eta(r),\zeta(t,r))\,dr\Bigr],\\
		Z(t,s)=\zeta(t,s),
	\end{cases}
	s\in[t,T],
\end{equation*}
and $(Y(t,s),Z(t,s))_{s\in[S,t]}\in\mathcal{H}^p_\mathbb{F}(S,t;\mathbb{R}^m\times\mathbb{R}^{m\times d})$ is the unique adapted solution of the BSDE
\begin{equation*}
	Y(t,s)=\eta(t)+\int^t_sg(t,r,\eta(r),Z(t,r))\,dr-\int^t_sZ(t,r)\,dW(r),\ s\in[S,t].
\end{equation*}
When $g(t,s,\eta,y,z)$ is independent of $y$, we say that a pair $(\eta(\cdot),\zeta(\cdot,\cdot))$ is an $L^p$-adapted C-solution of Type-\Rnum{1} BSVIE~\eqref{BSVIE} if
\begin{equation*}
	\begin{cases}
		\eta(t)=Y(t,t)\ &\text{a.s.},\ \forall\,t\in[S,T],\\
		\zeta(t,s)=Z(t,s)\ &\text{for}\ \mathrm{Leb}_{[S,T]}\otimes\mathbb{P}\text{-a.e.}\,(s,\omega)\in[S,T]\times\Omega,\ \forall\,t\in[S,T],
	\end{cases}
\end{equation*}
where $(Y(\cdot,\cdot),Z(\cdot,\cdot))\in\mathfrak{H}^p_\mathbb{F}(S,T;\mathbb{R}^m\times\mathbb{R}^{m\times d})$ is the $L^p$-adapted C-solution of EBSVIE~\eqref{EBSVIE} with the corresponding generator $g$.
\end{enumerate}
\end{rem}

The following theorem shows the existence, uniqueness, and a priori estimates of the $L^p$-adapted C-solution of EBSVIE~\eqref{EBSVIE}.

%% Theorem

\begin{theo}\label{theorem: EBSVIE}
Let Assumption~\ref{assumption: EBSVIE} hold. Then there exists a unique $L^p$-adapted C-solution $(Y(\cdot,\cdot),Z(\cdot,\cdot))\in\mathfrak{H}^p_\mathbb{F}(S,T;\mathbb{R}^m\times\mathbb{R}^{m\times d})$ of EBSVIE~\eqref{EBSVIE}. Moreover, for any $t,t'\in[S,T]$, the following estimate holds:
\begin{equation}\label{EBSVIE estimate}
\begin{split}
	&\mathbb{E}\Bigl[\sup_{s\in[t',T]}|Y(t,s)|^p+\Bigl(\int^T_{t'}|Z(t,s)|^2\,ds\Bigr)^{p/2}\Bigr]\\
	&\leq C\Biggl\{\mathbb{E}\Bigl[|\psi(t)|^p+\Bigl(\int^T_{t'}|g_0(t,s)|\,ds\Bigr)^p\Bigr]+\int^T_{t'}\mathbb{E}\Bigl[|\psi(\tau)|^p+\Bigl(\int^T_{\tau}|g_0(\tau,s)|\,ds\Bigr)^p\Bigr]\,d\tau\Biggr\},
\end{split}
\end{equation}
where $g_0(t,s):=g(t,s,0,0,0)$.\\
For $i=1,2$, let $(\psi_i,g_i)$ satisfy Assumption~\ref{assumption: EBSVIE} and let $(Y_i(\cdot,\cdot),Z_i(\cdot,\cdot))\in\mathfrak{H}^p_\mathbb{F}(S,T;\mathbb{R}^m\times\mathbb{R}^{m\times d})$ be the unique $L^p$-adapted C-solution of EBSVIE~\eqref{EBSVIE} corresponding to $(\psi_i,g_i)$, respectively. Then it holds that, for any $t,t'\in[S,T]$,
\begin{equation}\label{EBSVIE stability}
\begin{split}
	&\mathbb{E}\Bigl[\sup_{s\in[t',T]}|Y_1(t,s)-Y_2(t,s)|^p+\Bigl(\int^T_{t'}|Z_1(t,s)-Z_2(t,s)|^2\,ds\Bigr)^{p/2}\Bigr]\\
	&\leq C\Biggl\{\mathbb{E}\Bigl[|\Delta\psi(t)|^p+\Bigl(\int^T_{t'}|\Delta g(t,s)|\,ds\Bigr)^p\Bigr]+\int^T_{t'}\mathbb{E}\Bigl[|\Delta\psi(\tau)|^p+\Bigl(\int^T_\tau|\Delta g(\tau,s)|\,ds\Bigr)^p\Bigr]\,d\tau\Biggr\},
\end{split}
\end{equation}
where $\Delta\psi(t):=\psi_1(t)-\psi_2(t)$, and
\begin{equation*}
	\Delta g(t,s):=g_1(t,s,Y_1(s,s),Y_1(t,s),Z_1(t,s))-g_2(t,s,Y_1(s,s),Y_1(t,s),Z_1(t,s)).
\end{equation*}
\end{theo}

%% Remark

\begin{rem}
Wang~\cite{a_Wang_20} showed the well-posedness of EBSVIE~\eqref{EBSVIE} under a stronger assumption. His method is firstly showing the existence and uniqueness of the solution of EBSVIE~\eqref{EBSVIE} (defined on $S\leq t\leq s\leq T$) when $T-S$ is small, and then connecting them inductively by considering an associated family of BSDEs (or stochastic Fredholm equations). On the other hand, our proof relies on a simple observation based on an equivalent norm $\|\cdot\|_\beta$ defined below. A similar technique can be seen in the literature of BSDEs (see for example \cite{a_ElKaroui-_97}), and in the literature of BSVIEs (see for example \cite{a_Shi-Wang_12,a_Shi-Wang-Yong_15}). We remark that the estimates~\eqref{EBSVIE estimate} and \eqref{EBSVIE stability} are more detailed than \cite{a_Wang_20}. Indeed, by letting $t'=t$ and then taking the supremum over $t\in[S,T]$, we get the estimates in Theorem~3.1 of \cite{a_Wang_20}.
\end{rem}

%% Proof

\begin{proof}[Proof of Theorem~\ref{theorem: EBSVIE}]
In this proof, $C>0$ denotes a universal constant which may vary from line to line. Let $(y(\cdot,\cdot),z(\cdot,\cdot))\in\mathfrak{H}^p_\mathbb{F}(S,T;\mathbb{R}^m\times\mathbb{R}^{m\times d})$ be given. For each $t\in[S,T]$, consider the following BSDE:
\begin{equation}\label{BSDE map}
	Y(t,s)=\psi(t)+\int^T_sg(t,r,y(r,r),Y(t,r),Z(t,r))\,dr-\int^T_sZ(t,r)\,dW(r),\ s\in[S,T].
\end{equation}
By Lemma~\ref{lemma: BSDE}, there exists a unique $L^p$-adapted solution $(Y(t,\cdot),Z(t,\cdot))\in\mathcal{H}^p_\mathbb{F}(S,T;\mathbb{R}^m\times\mathbb{R}^{m\times d})$ for any $t\in[S,T]$. Furthermore, by the stability estimate of $L^p$-adapted solutions of BSDEs, we have, for each $t,t_0\in[S,T]$,
\begin{align*}
	&\mathbb{E}\Bigl[\sup_{s\in[S,T]}|Y(t,s)-Y(t_0,s)|^p+\Bigl(\int^T_S|Z(t,s)-Z(t_0,s)|^2\,ds\Bigr)^{p/2}\Bigr]\\
	&\leq C\mathbb{E}\Bigl[|\psi(t)-\psi(t_0)|^p\\
	&\hspace{1.5cm}+\Bigl(\int^T_S|g(t,s,y(s,s),Y(t_0,s),Z(t_0,s))-g(t_0,s,y(s,s),Y(t_0,s),Z(t_0,s))|\,ds\Bigr)^p\Bigr]\displaybreak[1]\\
	&\leq C\mathbb{E}\Bigl[|\psi(t)-\psi(t_0)|^p+\Bigl(\int^T_S|k(t,s)-k(t_0,s)|\,ds\Bigr)^p\\
	&\hspace{1.5cm}+\Bigl(\int^T_S|l(t,s)-l(t_0,s)|(|y(s,s)|+|Y(t_0,s)|+|Z(t_0,s)|)\,ds\Bigr)^p\Bigr],
\end{align*}
where we used \eqref{weak continuity} in the second inequality. Thus, by using Lemma~\ref{lemma: regularity}, we get
\begin{equation*}
	\lim_{t\to t_0}\mathbb{E}\Bigl[\sup_{s\in[S,T]}|Y(t,s)-Y(t_0,s)|^p+\Bigl(\int^T_S|Z(t,s)-Z(t_0,s)|^2\,ds\Bigr)^{p/2}\Bigr]=0
\end{equation*}
for each $t_0\in[S,T]$. This implies that the maps $[S,T]\ni t\mapsto Y(t,\cdot)\in L^p_\mathbb{F}(\Omega;C([S,T];\mathbb{R}^m))$ and $[S,T]\ni t\mapsto Z(t,\cdot)\in L^{p,2}_\mathbb{F}(S,T;\mathbb{R}^{m\times d})$ are continuous. By replacing them with jointly measurable versions (see Lemma~\ref{lemma: measurable version}) if necessary, we have that $(Y(\cdot,\cdot),Z(\cdot,\cdot))\in\mathfrak{H}^p_\mathbb{F}(S,T;\mathbb{R}^m\times\mathbb{R}^{m\times d})$. Therefore, we can define the mapping $\Theta:\mathfrak{H}^p_\mathbb{F}(S,T;\mathbb{R}^m\times\mathbb{R}^{m\times d})\to\mathfrak{H}^p_\mathbb{F}(S,T;\mathbb{R}^m\times\mathbb{R}^{m\times d})$ by $\Theta((y(\cdot,\cdot),z(\cdot,\cdot)):=(Y(\cdot,\cdot),Z(\cdot,\cdot))$. It suffices to show that $\Theta$ has a unique fixed point. To show that, we introduce the following norm on $\mathfrak{H}^p_\mathbb{F}(S,T;\mathbb{R}^m\times\mathbb{R}^{m\times d})$ parametrized by $\beta>0$:
\begin{align*}
	&\|(y(\cdot,\cdot),z(\cdot,\cdot))\|_\beta:=\sup_{t\in[S,T]}\Bigl\{e^{\beta t}\mathbb{E}\Bigl[\sup_{s\in[t,T]}|y(t,s)|^p+\Bigl(\int^T_t|z(t,s)|^2\,ds\Bigr)^{p/2}\Bigr]\\
	&\hspace{6cm}+\mathbb{E}\Bigl[\sup_{s\in[S,t]}|y(t,s)|^p+\Bigl(\int^t_S|z(t,s)|^2\,ds\Bigr)^{p/2}\Bigr]\Bigr\}^{1/p}
\end{align*}
for $(y(\cdot,\cdot),z(\cdot,\cdot))\in\mathfrak{H}^p_\mathbb{F}(S,T;\mathbb{R}^m\times\mathbb{R}^{m\times d})$. It can be easily shown that, for any $\beta>0$, $\|\cdot\|_\beta$ is equivalent to the original norm $\|\cdot\|_{\mathfrak{H}^p_\mathbb{F}(S,T;\mathbb{R}^m\times\mathbb{R}^{m\times d})}$. Furthermore, for each $(y(\cdot,\cdot),z(\cdot,\cdot))\in\mathfrak{H}^p_\mathbb{F}(S,T;\mathbb{R}^m\times\mathbb{R}^{m\times d})$ and $t\in[S,T]$, it holds that
\begin{equation}\label{technical idea}
	\mathbb{E}\bigl[|y(t,t)|^p\bigr]\leq \mathbb{E}\Bigl[\sup_{s\in[t,T]}|y(t,s)|^p\Bigr]\leq e^{-\beta t}\|(y(\cdot,\cdot),z(\cdot,\cdot)\|^p_\beta.
\end{equation}
We prove that $\Theta$ is contractive under the norm $\|\cdot\|_\beta$ when $\beta>0$ is large enough. To do so, take arbitrary $(y(\cdot,\cdot),z(\cdot,\cdot))$ and $(\bar{y}(\cdot,\cdot),\bar{z}(\cdot,\cdot))$ from $\mathfrak{H}^p_\mathbb{F}(S,T;\mathbb{R}^m\times\mathbb{R}^{m\times d})$, and define
\begin{equation*}
	\begin{cases}
		(Y(\cdot,\cdot),Z(\cdot,\cdot)):=\Theta((y(\cdot,\cdot),z(\cdot,\cdot))),\\
		(\bar{Y}(\cdot,\cdot),\bar{Z}(\cdot,\cdot)):=\Theta((\bar{y}(\cdot,\cdot),\bar{z}(\cdot,\cdot))).
	\end{cases}
\end{equation*}
Let $t\in[S,T]$ be fixed. Then by Lemma~\ref{lemma: BSDE}, we have
\begin{align}
	\nonumber&e^{\beta t}\mathbb{E}\Bigl[\sup_{s\in[t,T]}|Y(t,s)-\bar{Y}(t,s)|^p+\Bigl(\int^T_t|Z(t,s)-\bar{Z}(t,s)|^2\,ds\Bigr)^{p/2}\Bigr]\\
	\nonumber&\leq Ce^{\beta t}\mathbb{E}\Bigl[\Bigl(\int^T_t|g(t,s,y(s,s),Y(t,s),Z(t,s))-g(t,s,\bar{y}(s,s),Y(t,s),Z(t,s))|\,ds\Bigr)^p\Bigr]\displaybreak[1]\\
	\nonumber&\leq Ce^{\beta t}\int^T_t\mathbb{E}\bigl[|y(s,s)-\bar{y}(s,s)|^p\bigr]\,ds\displaybreak[1]\\
	\nonumber&\leq Ce^{\beta t}\int^T_te^{-\beta s}\,ds\,\|(y(\cdot,\cdot),z(\cdot,\cdot))-(\bar{y}(\cdot,\cdot),\bar{z}(\cdot,\cdot))\|^p_\beta\displaybreak[1]\\
	\label{estimate 1}&\leq \frac{C}{\beta}\|(y(\cdot,\cdot),z(\cdot,\cdot))-(\bar{y}(\cdot,\cdot),\bar{z}(\cdot,\cdot))\|^p_\beta,
\end{align}
where we used \eqref{technical idea} in the third inequality. On the other hand, since $(Y(t,s),Z(t,s))_{s\in[S,t]}$ and $(\bar{Y}(t,s),\bar{Z}(t,s))_{s\in[S,t]}$ are the unique $L^p$-adapted solutions of BSDEs
\begin{align*}
	&Y(t,s)=Y(t,t)+\int^t_sg(t,r,y(r,r),Y(t,r),Z(t,r))\,dr-\int^t_sZ(t,r)\,dW(r),\ s\in[S,t],
\shortintertext{and}
	&\bar{Y}(t,s)=\bar{Y}(t,t)+\int^t_sg(t,r,\bar{y}(r,r),\bar{Y}(t,r),\bar{Z}(t,r))\,dr-\int^t_s\bar{Z}(t,r)\,dW(r),\ s\in[S,t],
\end{align*}
respectively, again by Lemma~\ref{lemma: BSDE}, we get
\begin{align*}
	&\mathbb{E}\Bigl[\sup_{s\in[S,t]}|Y(t,s)-\bar{Y}(t,s)|^p+\Bigl(\int^t_S|Z(t,s)-\bar{Z}(t,s)|^2\,ds\Bigr)^{p/2}\Bigr]\\
	&\leq C\mathbb{E}\Bigl[|Y(t,t)-\bar{Y}(t,t)|^p\\
	&\hspace{1cm}+\Bigl(\int^t_S|g(t,s,y(s,s),Y(t,s),Z(t,s))-g(t,s,\bar{y}(s,s),Y(t,s),Z(t,s))|\,ds\Bigr)^p\Bigr].
\end{align*}
By the estimate~\eqref{estimate 1}, it holds that, in particular,
\begin{equation*}
	\mathbb{E}\bigl[|Y(t,t)-\bar{Y}(t,t)|^p\bigr]\leq\frac{C}{\beta}\|(y(\cdot,\cdot),z(\cdot,\cdot))-(\bar{y}(\cdot,\cdot),\bar{z}(\cdot,\cdot))\|^p_\beta.
\end{equation*}
Moreover, we have
\begin{align*}
	&\mathbb{E}\Bigl[\Bigl(\int^t_S|g(t,s,y(s,s),Y(t,s),Z(t,s))-g(t,s,\bar{y}(s,s),Y(t,s),Z(t,s))|\,ds\Bigr)^p\Bigr]\\
	&\leq C\int^t_S\mathbb{E}\bigl[|y(s,s)-\bar{y}(s,s)|^p\bigr]\,ds\displaybreak[1]\\
	&\leq C\int^t_Se^{-\beta s}\,ds\,\|(y(\cdot,\cdot),z(\cdot,\cdot))-(\bar{y}(\cdot,\cdot),\bar{z}(\cdot,\cdot))\|^p_\beta\displaybreak[1]\\
	&\leq\frac{C}{\beta}\|(y(\cdot,\cdot),z(\cdot,\cdot))-(\bar{y}(\cdot,\cdot),\bar{z}(\cdot,\cdot))\|^p_\beta,
\end{align*}
where we used \eqref{technical idea} in the second inequality. Thus, we get
\begin{equation}\label{estimate 2}
\begin{split}
	&\mathbb{E}\Bigl[\sup_{s\in[S,t]}|Y(t,s)-\bar{Y}(t,s)|^p+\Bigl(\int^t_S|Z(t,s)-\bar{Z}(t,s)|^2\,ds\Bigr)^{p/2}\Bigr]\\
	&\leq\frac{C}{\beta}\|(y(\cdot,\cdot),z(\cdot,\cdot))-(\bar{y}(\cdot,\cdot),\bar{z}(\cdot,\cdot))\|^p_\beta.
\end{split}
\end{equation}
Note that, in the estimates~\eqref{estimate 1} and \eqref{estimate 2}, the constant $C>0$ does not depend on $t\in[S,T]$ and $\beta>0$. Consequently, we obtain
\begin{equation*}
	\|(Y(\cdot,\cdot),Z(\cdot,\cdot))-(\bar{Y}(\cdot,\cdot),\bar{Z}(\cdot,\cdot))\|^p_\beta\leq\frac{C}{\beta}\|(y(\cdot,\cdot),z(\cdot,\cdot))-(\bar{y}(\cdot,\cdot),\bar{z}(\cdot,\cdot))\|^p_\beta.
\end{equation*}
Therefore, if we take the parameter $\beta>0$ large enough, then the map $\Theta$ is contractive under the norm $\|\cdot\|_\beta$. Consequently, we see that EBSVIE~\eqref{EBSVIE} has a unique $L^p$-adapted C-solution.
\par
Next, we prove the estimate \eqref{EBSVIE estimate}. Let $(Y(\cdot,\cdot),Z(\cdot,\cdot))\in\mathfrak{H}^p_\mathbb{F}(S,T;\mathbb{R}^m\times\mathbb{R}^{m\times d})$ be the unique $L^p$-adapted C-solution of EBSVIE~\eqref{EBSVIE}. By letting $\eta(t):=Y(t,t)$, $t\in[S,T]$, we see that, for each $t\in[S,T]$, $(Y(t,\cdot),Z(t,\cdot))\in\mathcal{H}^p_\mathbb{F}(S,T;\mathbb{R}^m\times\mathbb{R}^{m\times d})$ is the unique $L^p$-adapted solution of the BSDE
\begin{equation*}
	Y(t,s)=\psi(t)+\int^T_sg(t,r,\eta(r),Y(t,r),Z(t,r))\,dr-\int^T_sZ(t,r)\,dW(r),\ s\in[S,T].
\end{equation*}
Thus, by Lemma~\ref{lemma: BSDE}, for any $t,t'\in[S,T]$,
\begin{align}\label{estimate YZ}
	\nonumber&\mathbb{E}\Bigl[\sup_{s\in[t',T]}|Y(t,s)|^p+\Bigl(\int^T_{t'}|Z(t,s)|^2\,ds\Bigr)^{p/2}\Bigr]\\
	\nonumber&\leq C\mathbb{E}\Bigl[|\psi(t)|^p+\Bigl(\int^T_{t'}|g(t,s,\eta(s),0,0)|\,ds\Bigr)^p\Bigr]\\
	&\leq C\mathbb{E}\Bigl[|\psi(t)|^p+\Bigl(\int^T_{t'}|g_0(t,s)|\,ds\Bigr)^p\Bigr]+C\int^T_{t'}\mathbb{E}\bigl[|\eta(\tau)|^p\bigr]\,d\tau,
\end{align}
where $g_0(t,s):=g(t,s,0,0,0)$. In particular, if we let $t'=t$, then we obtain
\begin{equation*}
	\mathbb{E}\bigl[|\eta(t)|^p\bigr]\leq C\mathbb{E}\Bigl[|\psi(t)|^p+\Bigl(\int^T_t|g_0(t,s)|\,ds\Bigr)^p\Bigr]+C\int^T_t\mathbb{E}\bigl[|\eta(\tau)|^p\bigr]\,d\tau,\ \forall\,t\in[S,T].
\end{equation*}
Then Gronwall's inequality yields that, for any $\tau\in[S,T]$,
\begin{equation}\label{estimate eta}
\begin{split}
	&\mathbb{E}\bigl[|\eta(\tau)|^p\bigr]\\
	&\leq C\Biggl\{\mathbb{E}\Bigl[|\psi(\tau)|^p+\Bigl(\int^T_\tau|g_0(\tau,s)|\,ds\Bigr)^p\Bigr]+\int^T_\tau\mathbb{E}\Bigl[|\psi(\tau')|^p+\Bigl(\int^T_{\tau'}|g_0(\tau',s)|\,ds\Bigr)^p\Bigr]\,d\tau'\Biggr\}.
\end{split}
\end{equation}
By inserting the estimate~\eqref{estimate eta} into \eqref{estimate YZ}, we obtain the estimate \eqref{EBSVIE estimate}. Similarly we can show the stability estimate \eqref{EBSVIE stability}.
\end{proof}

As a corollary, we obtain a similar result for a Type-\Rnum{1} BSVIE. For the solution concept of such a equation, see Remark~\ref{remark: EBSVIE}~(\rnum{3}).

%% Corollary

\begin{cor}\label{corollary: BSVIE}
Let Assumption~\ref{assumption: EBSVIE} hold. Furthermore, assume that the generator $g(t,s,\eta,y,z)$ does not depend on $y$. Then there exists a unique $L^p$-adapted C-solution $(\eta(\cdot),\zeta(\cdot,\cdot))$ of BSVIE~\eqref{BSVIE}, and the following estimate holds:
\begin{equation}\label{BSVIE estimate}
	\sup_{t\in[S,T]}\mathbb{E}\Bigl[|\eta(t)|^p+\Bigl(\int^T_t|\zeta(t,s)|^2\,ds\Bigr)^{p/2}\Bigr]\leq C\sup_{t\in[S,T]}\mathbb{E}\Bigl[|\psi(t)|^p+\Bigl(\int^T_t|g(t,s,0,0)|\,ds\Bigr)^p\Bigr].
\end{equation}
For $i=1,2$, let $(\psi_i,g_i)$ satisfy Assumption~\ref{assumption: EBSVIE} with the generator $g_i(t,s,\eta,y,z)$ being independent of $y$. Let $(\eta_i(\cdot),\zeta_i(\cdot,\cdot))$ be the unique $L^p$-adapted C-solution of BSVIE~\eqref{BSVIE} corresponding to $(\psi_i,g_i)$. Then it holds that
\begin{equation}\label{BSVIE stability}
\begin{split}
	&\sup_{t\in[S,T]}\mathbb{E}\Bigl[|\eta_1(t)-\eta_2(t)|^p+\Bigl(\int^T_t|\zeta_1(t,s)-\zeta_2(t,s)|^2\,ds\Bigr)^{p/2}\Bigr]\\
	&\leq C\sup_{t\in[S,T]}\mathbb{E}\Bigl[|\psi_1(t)-\psi_2(t)|^p\\
	&\hspace{1.5cm}+\Bigl(\int^T_t|g_1(t,s,\eta_1(s),\zeta_1(t,s))-g_2(t,s,\eta_1(s),\zeta_1(t,s))|\,ds\Bigr)^p\Bigr].
\end{split}
\end{equation}
\end{cor}

\par
\medskip
Next, we study the regularity of the solution $(Y(t,s),Z(t,s))$ of EBSVIE~\eqref{EBSVIE} with respect to $t\in[S,T]$. For the free term $\psi$ and the generator $g$, we further impose the following assumptions.

%% Assumption

\begin{assum}\label{assumption: EBSVIE derivative}
Fix $p\geq2$.
\begin{enumerate}
\renewcommand{\labelenumi}{(\roman{enumi})}
\item
$\psi(\cdot)\in C^1([S,T];L^p_{\mathcal{F}_T}(\Omega;\mathbb{R}^m))$.
\item
$g$ satisfies Assumption~\ref{assumption: EBSVIE}~(\rnum{2}). Moreover, the following hold.
\begin{itemize}
\item
For any $t\in[S,T]$, for $\mathrm{Leb}_{[S,T]}\otimes\mathbb{P}$-a.e.\,$(s,\omega)\in[S,T]\times\Omega$, and for any $\eta\in\mathbb{R}^m$, the function $(y,z)\mapsto g(t,s,\eta,y,z)$ is differentiable. Moreover, there exist a process $l(\cdot,\cdot)\in C_b([S,T];L^0_\mathbb{F}(S,T;\mathbb{H}))$ with a Euclidean space $\mathbb{H}$ and a modulus of continuity $\rho:[0,\infty)\to[0,\infty)$ such that, for any $t_1,t_2\in[S,T]$, for $\mathrm{Leb}_{[S,T]}\otimes\mathbb{P}$-a.e.\,$(s,\omega)\in[S,T]\times\Omega$, it holds that
\begin{align*}
	&|\partial_{(y,z)}g(t_1,s,\eta,y_1,z_1)-\partial_{(y,z)}g(t_2,s,\eta,y_2,z_2)|\\
	&\leq |l(t_1,s)-l(t_2,s)|+\rho(|y_1-y_2|+|z_1-z_2|)
\end{align*}
for any $\eta,y_1,y_2\in\mathbb{R}^m$ and $z_1,z_2\in\mathbb{R}^{m\times d}$;
\item
There exists a measurable function $\partial_tg$ satisfying Assumption~\ref{assumption: EBSVIE}~(\rnum{2}) such that, for any $t_1,t_2\in[S,T]$, for $\mathrm{Leb}_{[S,T]}\otimes\mathbb{P}$-a.e.\,$(s,\omega)\in[S,T]\times\Omega$, it holds that
\begin{equation*}
	g(t_1,s,\eta,y,z)-g(t_2,s,\eta,y,z)=\int^{t_1}_{t_2}\partial_tg(\tau,s,\eta,y,z)\,d\tau
\end{equation*}
for any $\eta,y\in\mathbb{R}^m$ and $z\in\mathbb{R}^{m\times d}$.
\end{itemize}
\end{enumerate}
\end{assum}

Suppose that Assumption~\ref{assumption: EBSVIE derivative} holds. Let $(Y(\cdot,\cdot),Z(\cdot,\cdot))\in\mathfrak{H}^p_\mathbb{F}(S,T;\mathbb{R}^m\times\mathbb{R}^{m\times d})$ be the unique $L^p$-adapted C-solution of EBSVIE~\eqref{EBSVIE}. Consider the following linear EBSVIE for $(\mathcal{Y}(\cdot,\cdot),\mathcal{Z}(\cdot,\cdot))$:
\begin{equation}\label{EBSVIE derivative}
\begin{split}
	\mathcal{Y}(t,s)=&\partial_t\psi(t)+\int^T_s\bigl(g_t(t,r)+g_y(t,r)\mathcal{Y}(t,r)+\sum^d_{j=1}g_{z_j}(t,r)\mathcal{Z}_j(t,r)\bigr)\,dr\\
	&\hspace{2cm}-\int^T_s\mathcal{Z}(t,r)\,dW(r),\ s\in[S,T],\ t\in[S,T],
\end{split}
\end{equation}
where, for each $z\in\mathbb{R}^{m\times d}$, $z_j\in\mathbb{R}^m$ denotes the $j$-th column, and
\begin{equation*}
	\begin{cases}
		g_t(t,r):=\partial_tg(t,r,Y(r,r),Y(t,r),Z(t,r)),\\
		g_y(t,r):=\partial_yg(t,r,Y(r,r),Y(t,r),Z(t,r)),\\
		g_{z_j}(t,r):=\partial_{z_j}g(t,r,Y(r,r),Y(t,r),Z(t,r)),\ j=1,\dots,d.
	\end{cases}
\end{equation*}
Observe that $\partial_t\psi(\cdot)\in C([S,T];L^p_{\mathcal{F}_T}(\Omega;\mathbb{R}^m))$, $g_t(\cdot,\cdot)\in C([S,T];L^{p,1}_\mathbb{F}(S,T;\mathbb{R}^m))$, and
\begin{equation*}
	g_y(\cdot,\cdot),\,g_{z_j}(\cdot,\cdot)\in C_b([S,T];L^0_\mathbb{F}(S,T;\mathbb{R}^{m\times m})),\ j=1,\dots,d.
\end{equation*}
Thus, the coefficients of EBSVIE~\eqref{EBSVIE derivative} satisfy Assumption~\ref{assumption: EBSVIE}, and hence there exists a unique $L^p$-adapted C-solution $(\mathcal{Y}(\cdot,\cdot),\mathcal{Z}(\cdot,\cdot))\in\mathfrak{H}^p_\mathbb{F}(S,T;\mathbb{R}^m\times\mathbb{R}^{m\times d})$. In fact, the equation \eqref{EBSVIE derivative} is just a family of (decoupled) BSDEs parametrized by $t\in[S,T]$. The next result shows that, under the above assumption, the function $t\mapsto(Y(t,\cdot),Z(t,\cdot))$ is differentiable (as a Banach space-valued function) and $(\mathcal{Y}(\cdot,\cdot),\mathcal{Z}(\cdot,\cdot))$ coincides with its derivative.

%% Theorem

\begin{theo}\label{theorem: EBSVIE derivative}
Let Assumption~\ref{assumption: EBSVIE derivative} hold. Let $(Y(\cdot,\cdot),Z(\cdot,\cdot))\in\mathfrak{H}^p_\mathbb{F}(S,T;\mathbb{R}^m\times\mathbb{R}^{m\times d})$ and $(\mathcal{Y}(\cdot,\cdot),\mathcal{Z}(\cdot,\cdot))\in\mathfrak{H}^p_\mathbb{F}(S,T;\mathbb{R}^m\times\mathbb{R}^{m\times d})$ be the $L^p$-adapted C-solutions of EBSVIE~\eqref{EBSVIE} and \eqref{EBSVIE derivative}, respectively. Then
\begin{itemize}
\item
the Banach space-valued function $t\mapsto Y(t,\cdot)$ is in $C^1([S,T];L^p_\mathbb{F}(\Omega;C([S,T];\mathbb{R}^m)))$ with the derivative $\partial_tY(\cdot,\cdot)=\mathcal{Y}(\cdot,\cdot)$, and
\item
the Banach space-valued function $t\mapsto Z(t,\cdot)$ is in $C^1([S,T];L^{p,2}_\mathbb{F}(S,T;\mathbb{R}^{m\times d}))$ with the derivative $\partial_tZ(\cdot,\cdot)=\mathcal{Z}(\cdot,\cdot)$.
\end{itemize}
\end{theo}

%% Proof

\begin{proof}
Fix $t\in[S,T]$. For each $h\neq0$ such that $t+h\in[S,T]$, define
\begin{equation*}
	\begin{cases}
		\Delta^h\mathcal{Y}(t,s):=\frac{1}{h}\bigl(Y(t+h,s)-Y(t,s)\bigr)-\mathcal{Y}(t,s),\\
		\Delta^h\mathcal{Z}(t,s):=\frac{1}{h}\bigl(Z(t+h,s)-Z(t,s)\bigr)-\mathcal{Z}(t,s).
	\end{cases}
\end{equation*}
We show that
\begin{equation}\label{differentiation limit}
	\lim_{h\to0}\mathbb{E}\Bigl[\sup_{s\in[S,T]}\bigl|\Delta^h\mathcal{Y}(t,s)\bigr|^p+\Bigl(\int^T_S\bigl|\Delta^h\mathcal{Z}(t,s)\bigr|^2\,ds\Bigr)^{p/2}\Bigr]=0.
\end{equation}
Simple calculations show that, for any $s\in[S,T]$,
\begin{align*}
	&\Delta^h\mathcal{Y}(t,s)\\
	&=\frac{1}{h}\bigl(\psi(t+h)-\psi(t)\bigr)-\partial_t\psi(t)-\int^T_s\Delta^h\mathcal{Z}(t,r)\,dW(r)\\
	&\hspace{0.5cm}+\int^T_s\Bigl\{\frac{1}{h}\bigl(g(t+h,r,Y(r,r),Y(t+h,r),Z(t+h,r))-g(t,r,Y(r,r),Y(t,r),Z(t,r))\bigr)\\
	&\hspace{2cm}-g_t(t,r)-g_y(t,r)\mathcal{Y}(t,r)-\sum^d_{j=1}g_{z_j}(t,r)\mathcal{Z}_j(t,r)\Bigr\}\,dr\displaybreak[1]\\
	&=\frac{1}{h}\bigl(\psi(t+h)-\psi(t)\bigr)-\partial_t\psi(t)-\int^T_s\Delta^h\mathcal{Z}(t,r)\,dW(r)\\
	&\hspace{0.5cm}+\int^T_s\Bigl\{\tilde{g}^h_t(t,r)-g_t(t,r)+\bigl(\tilde{g}^h_y(t,r)-g_y(t,r)\bigr)\mathcal{Y}(t,r)+\sum^d_{j=1}\bigl(\tilde{g}^h_{z_j}(t,r)-g_{z_j}(t,r)\bigr)\mathcal{Z}_j(t,r)\\
	&\hspace{2cm}+\tilde{g}^h_y(t,r)\Delta^h\mathcal{Y}(t,r)+\sum^d_{j=1}\tilde{g}^h_{z_j}(t,r)\Delta^h\mathcal{Z}_j(t,r)\Bigr\}\,dr,
\end{align*}
where
\begin{align*}
	&\tilde{g}^h_t(t,r):=\frac{1}{h}\int^{t+h}_t\partial_tg(\tau,r,Y(r,r),Y(t+h,r),Z(t+h,r))\,d\tau,\displaybreak[1]\\
	&\tilde{g}^h_y(t,r):=\int^1_0\partial_yg\bigl(t,r,Y(r,r),Y(t,r)+\lambda\bigl(Y(t+h,r)-Y(t,r)\bigr),\\
	&\hspace{5cm}Z(t,r)+\lambda\bigl(Z(t+h,r)-Z(t,r)\bigr)\bigr)\,d\lambda
\shortintertext{and}
	&\tilde{g}^h_{z_j}(t,r):=\int^1_0\partial_{z_j}g\bigl(t,r,Y(r,r),Y(t,r)+\lambda\bigl(Y(t+h,r)-Y(t,r)\bigr),\\
	&\hspace{5cm}Z(t,r)+\lambda\bigl(Z(t+h,r)-Z(t,r)\bigr)\bigr)\,d\lambda.
\end{align*}
Note that $\tilde{g}^h_y(\cdot,\cdot)$ and $\tilde{g}^h_{z_j}(\cdot,\cdot)$ are bounded uniformly in $h$. Thus, by the standard estimate of the solution of the BSDE, we see that there exists a constant $C>0$ such that, for any $h\neq0$,
\begin{align*}
	&\mathbb{E}\Bigl[\sup_{s\in[S,T]}\bigl|\Delta^h\mathcal{Y}(t,s)\bigr|^p+\Bigl(\int^T_S\bigl|\Delta^h\mathcal{Z}(t,s)\bigr|^2\,ds\Bigr)^{p/2}\Bigr]\\
	&\leq C\mathbb{E}\Bigl[\Bigl|\frac{1}{h}\bigl(\psi(t+h)-\psi(t)\bigr)-\partial_t\psi(t)\Bigr|^p+\Bigl(\int^T_S\bigl|\tilde{g}^h_t(t,s)-g_t(t,s)\bigr|\,ds\Bigr)^p\\
	&\hspace{1.5cm}+\Bigl(\int^T_S\Bigl\{\bigl|\tilde{g}^h_y(t,s)-g_y(t,s)\bigr|\bigl|\mathcal{Y}(t,s)\bigr|+\sum^d_{j=1}\bigl|\tilde{g}^h_{z_j}(t,s)-g_{z_j}(t,s)\bigr|\bigl|\mathcal{Z}(t,s)\bigr|\Bigr\}\,ds\Bigr)^p\Bigr].
\end{align*}
Clearly it holds that $\lim_{h\to0}\mathbb{E}\Bigl[\Bigl|\frac{1}{h}\bigl(\psi(t+h)-\psi(t)\bigr)-\partial_t\psi(t)\Bigr|^p\Bigr]=0$. Suppose that $h>0$. Since $\partial_tg$ satisfies Assumption~\ref{assumption: EBSVIE}~(\rnum{2}), we see that
\begin{align*}
	&\mathbb{E}\Bigl[\Bigl(\int^T_S\bigl|\tilde{g}^h_t(t,s)-g_t(t,s)\bigr|\,ds\Bigr)^p\Bigr]\\
	&\leq\mathbb{E}\Bigl[\Bigl(\int^T_S\frac{1}{h}\int^{t+h}_t\bigl|\partial_tg(\tau,s,Y(s,s),Y(t+h,s),Z(t+h,s))\\
	&\hspace{4cm}-\partial_tg(t,s,Y(s,s),Y(t,s),Z(t,s))\bigr|\,d\tau\,ds\Bigr)^p\Bigr]\displaybreak[1]\\
	&\leq\mathbb{E}\Bigl[\Bigl(\int^T_S\Bigl\{L\bigl(|Y(t+h,s)-Y(t,s)|+|Z(t+h,s)-Z(t,s)|\bigr)\\
	&\hspace{2cm}+\frac{1}{h}\int^{t+h}_t|k(\tau,s)-k(t,s)|\,d\tau\\
	&\hspace{2cm}+\frac{1}{h}\int^{t+h}_t|l(\tau,s)-l(t,s)|\bigl(|Y(s,s)|+|Y(t,s)|+|Z(t,s)|\bigr)\,d\tau\Bigr\}\,ds\Bigr)^p\Bigr],
\end{align*}
where $L>0$, $k(\cdot,\cdot)\in C([S,T];L^{p,1}_\mathbb{F}(S,T;\mathbb{H}))$ and $l(\cdot,\cdot)\in C_b([S,T];L^0_\mathbb{F}(S,T;\mathbb{G}))$ are given in Assumption~\ref{assumption: EBSVIE}~(\rnum{2}) with respect to $\partial_tg$. By using the continuity property of the map $t\mapsto(Y(t,\cdot),Z(t,\cdot))$, we see that
\begin{equation*}
	\lim_{h\downarrow0}\mathbb{E}\Bigl[\Bigl(\int^T_S\bigl(|Y(t+h,s)-Y(t,s)|+|Z(t+h,s)-Z(t,s)|\bigr)\,ds\Bigr)^p\Bigr]=0.
\end{equation*}
Furthermore, by using Fubini's theorem and H\"{o}lder's inequality, we have
\begin{align*}
	\mathbb{E}\Bigl[\Bigl(\int^T_S\frac{1}{h}\int^{t+h}_t|k(\tau,s)-k(t,s)|\,d\tau\,ds\Bigr)^p\Bigr]&=\mathbb{E}\Bigl[\Bigl(\frac{1}{h}\int^{t+h}_t\int^T_S|k(\tau,s)-k(t,s)|\,ds\,d\tau\Bigr)^p\Bigr]\\
	&\leq\frac{1}{h}\int^{t+h}_t\mathbb{E}\Bigl[\Bigl(\int^T_S|k(\tau,s)-k(t,s)|\,ds\Bigr)^p\Bigr]\,d\tau\\
	&\leq\sup_{\tau\in[t,t+h]}\mathbb{E}\Bigl[\Bigl(\int^T_S|k(\tau,s)-k(t,s)|\,ds\Bigr)^p\Bigr]\overset{h\downarrow0}{\longrightarrow}0.
\end{align*}
By the same calculation as above and Lemma~\ref{lemma: regularity}, we see that
\begin{align*}
	&\mathbb{E}\Bigl[\Bigl(\int^T_S\frac{1}{h}\int^{t+h}_t|l(\tau,s)-l(t,s)|\bigl(|Y(s,s)|+|Y(t,s)|+|Z(t,s)|\bigr)\,d\tau\,ds\Bigr)^p\Bigr]\\
	&\leq\sup_{\tau\in[t,t+h]}\mathbb{E}\Bigl[\Bigl(\int^T_S|l(\tau,s)-l(t,s)|\bigl(|Y(s,s)|+|Y(t,s)|+|Z(t,s)|\bigr)\,ds\Bigr)^p\Bigr]\overset{h\downarrow0}{\longrightarrow}0.
\end{align*}
Therefore, we obtain
\begin{equation*}
	\lim_{h\downarrow0}\mathbb{E}\Bigl[\Bigl(\int^T_S\bigl|\tilde{g}^h_t(t,s)-g_t(t,s)\bigr|\,ds\Bigr)^p\Bigr]=0.
\end{equation*}
By the same way, we can show that $\lim_{h\uparrow0}\mathbb{E}\Bigl[\Bigl(\int^T_S\bigl|\tilde{g}^h_t(t,s)-g_t(t,s)\bigr|\,ds\Bigr)^p\Bigr]=0$. On the other hand, since the map $t\mapsto(Y(t,\cdot),Z(t,\cdot))$ is continuous in $L^0_\mathbb{F}(S,T;\mathbb{R}^m\times\mathbb{R}^{m\times d})$, by the assumption, we see that $|\tilde{g}^h_y(t,\cdot)-g_y(t,\cdot)\bigr|$ and $\sum^d_{j=1}|\tilde{g}^h_{z_j}(t,\cdot)-g_{z_j}(t,\cdot)\bigr|$ tend to zero as $h\to0$ in $L^0_\mathbb{F}(S,T;\mathbb{R})$. Since these two terms are uniformly bounded, again by Lemma~\ref{lemma: regularity}, we see that
\begin{equation*}
	\lim_{h\to0}\mathbb{E}\Bigl[\Bigl(\int^T_S\Bigl\{\bigl|\tilde{g}^h_y(t,s)-g_y(t,s)\bigr|\bigl|\mathcal{Y}(t,s)\bigr|+\sum^d_{j=1}\bigl|\tilde{g}^h_{z_j}(t,s)-g_{z_j}(t,s)\bigr|\bigl|\mathcal{Z}(t,s)\bigr|\Bigr\}\,ds\Bigr)^p\Bigr]=0.
\end{equation*}
Consequently, we get \eqref{differentiation limit} and finish the proof.
\end{proof}

%% Remark

\begin{rem}
From the above result and Lemma~\ref{lemma: Z diagonal}, under Assumption~\ref{assumption: EBSVIE derivative}, the diagonal process $\mathrm{Diag}[Z](\cdot)\in L^{p,2}_\mathbb{F}(S,T;\mathbb{R}^{m\times d})$ is well-defined and satisfies Property~(D) with respect to $Z(\cdot,\cdot)$. This consequence provides an interesting generalization where the generator $g$ depends also on $\mathrm{Diag}[Z](\cdot)$, that is, the following equation:
\begin{equation}\label{generalization}
\begin{split}
	&Y(t,s)=\psi(t)+\int^T_sg(t,r,Y(r,r),\mathrm{Diag}[Z](r),Y(t,r),Z(t,r))\,dr-\int^T_sZ(t,r)\,dW(r),\\
	&\hspace{5cm}s\in[S,T],\ t\in[S,T].
\end{split}
\end{equation}
Wang--Yong~\cite{a_HWang-Yong_19} studied a similar equation (in a Markovian setting) in view of a generalization of the Feynman--Kac formula. From our discussions, for the sake of the well-definedness of $\mathrm{Diag}[Z](\cdot)$, we can guess that the ``solution'' $(Y(\cdot,\cdot),Z(\cdot,\cdot))$ of \eqref{generalization} should be regular in an appropriate sense. Indeed, after we submitted the first version of this paper, Hern\'{a}ndez--Possama\"{i}~\cite{a_Hernandez-Possamai_20'} reported a relevant result on this issue. They assumed differentiability conditions which are similar to Assumption~\ref{assumption: EBSVIE derivative}, and showed the well-posedness of the generalized equation \eqref{generalization} by considering a coupled system of $(Y(\cdot,\cdot),Z(\cdot,\cdot))$ and their derivatives $(\partial_tY(\cdot,\cdot),\partial_tZ(\cdot,\cdot))$, together with an auxiliary BSDE which corresponds to the dynamics of $Y(s,s)$. Compared with the discussions of \cite{a_Hernandez-Possamai_20'}, in this paper, we firstly established the well-posedness of \eqref{EBSVIE} under Assumption~\ref{assumption: EBSVIE} without differentiability conditions, and then proved the regularity of the solution under Assumption~\ref{assumption: EBSVIE derivative}. We remark that, in our setting, EBSVIE~\eqref{EBSVIE derivative} is consistent with the dynamics of the derivatives $(\partial_tY(\cdot,\cdot),\partial_tZ(\cdot,\cdot))$ appearing in \cite{a_Hernandez-Possamai_20'}.
\end{rem}

%%%%%%%%%%%%%%%
%% Section
%%%%%%%%%%%%%%%

\section{Time-inconsistent stochastic recursive control problems}\label{section: control problem}

In this and the next sections, we investigate, in the open-loop framework, a time-inconsistent stochastic recursive control problem where the cost functional is defined by the solution of a Type-\Rnum{1} BSVIE. Let $T>0$ be a finite time horizon. In the following, for the sake of simplicity of notation, we assume that $d=1$, that is, the Brownian motion $W(\cdot)$ is one-dimensional. Our results can be easily generalized to the case of a general $d\in\mathbb{N}$.
\par
For each $\tau\in[0,T)$, define the set of admissible controls on $[\tau,T]$ by
\begin{equation*}
	\mathcal{U}[\tau,T]:=\{u:\Omega\times[\tau,T]\to U\,|\,u(\cdot)\ \text{is progressively measurable}\},
\end{equation*}
where $(U,d)$ is a separable metric space. We define the set of initial conditions by
\begin{equation*}
	\mathcal{I}:=\left\{(\tau,x_\tau)\relmiddle|\tau\in[0,T),\ x_\tau\in\bigcap_{p\geq1}L^p_{\mathcal{F}_\tau}(\Omega;\mathbb{R}^n)\right\}.
\end{equation*}
For each initial condition $(\tau,x_\tau)\in\mathcal{I}$ and control process $u(\cdot)\in\mathcal{U}[\tau,T]$, the corresponding ($\mathbb{R}^n$-valued) state process $X(\cdot)=X(\cdot;\tau,x_\tau;u(\cdot))$ is defined by the solution to the following SDE:
\begin{equation}\label{state SDE}
	\begin{cases}
		dX(s)=b\bigl(s,u(s),X(s)\bigr)\,ds+\sigma\bigl(s,u(s),X(s)\bigr)\,dW(s),\ s\in[\tau,T],\\
		X(\tau)=x_\tau.
	\end{cases}
\end{equation}
Define the cost functional by
\begin{equation*}
	J(\tau,x_\tau;u(\cdot)):=Y(\tau)
\end{equation*}
where $(Y(\cdot),Z(\cdot,\cdot))=(Y(\cdot;\tau,x_\tau;u(\cdot)),Z(\cdot,\cdot;\tau,x_\tau;u(\cdot)))$ is the adapted C-solution to the following ($\mathbb{R}$-valued) Type-\Rnum{1} BSVIE:
\begin{equation}\label{cost BSVIE}
	Y(t)=h\bigl(t,X(T)\bigr)+\int^T_tf\bigl(t,s,u(s),X(s),Y(s),Z(t,s)\bigr)\,ds-\int^T_tZ(t,s)\,dW(s),\ t\in[\tau,T].
\end{equation}
Our definition of $J(\tau,x_\tau;u(\cdot))$ is a recursive cost functional with general (non-exponential) discounting; see \cite{a_HWang-Yong_19}. We impose the following assumptions on the coefficients:

%% Assumption

\begin{assum}[on SDE~\eqref{state SDE}]\label{assumption: state SDE}
\begin{enumerate}
\renewcommand{\labelenumi}{(\roman{enumi})}
\item
The maps $b,\,\sigma: [0,T]\times U\times\mathbb{R}^n\to\mathbb{R}^n$ are measurable, and there exist a constant $L>0$ and a modulus of continuity $\rho:[0,\infty)\to[0,\infty)$ such that for $\varphi=b,\sigma$, we have
\begin{equation*}
	\begin{cases}
		|\varphi(s,u_1,x_1)-\varphi(s,u_2,x_2)|\leq L|x_1-x_2|+\rho(d(u_1,u_2)),\\
		\hspace{4cm}\forall\,s\in[0,T],\ x_1,x_2\in\mathbb{R}^n,\ u_1,u_2\in U;\\
		|\varphi(s,u,0)|\leq L,\ \forall\,(s,u)\in[0,T]\times U.
	\end{cases}
\end{equation*}
\item
The maps $b$ and $\sigma$ are $C^2$ in $x$. Moreover, there exist a constant $L>0$ and a modulus of continuity $\rho:[0,\infty)\to[0,\infty)$ such that for $\varphi=b,\sigma$, we have
\begin{equation*}
	\begin{cases}
		|\partial_x\varphi(s,u_1,x_1)-\partial_x\varphi(s,u_2,x_2)|\leq L|x_1-x_2|+\rho(d(u_1,u_2)),\\
		|\partial^2_x\varphi(s,u_1,x_1)-\partial^2_x\varphi(s,u_2,x_2)|\leq \rho(|x_1-x_2|+d(u_1,u_2)),\\
		\hspace{4cm}\forall\,s\in[0,T],\ x_1,x_2\in\mathbb{R}^n,\ u_1,u_2\in U.
	\end{cases}
\end{equation*}
\end{enumerate}
\end{assum}

\begin{assum}[on BSVIE~\eqref{cost BSVIE}]\label{assumption: cost BSVIE}
\begin{enumerate}
\renewcommand{\labelenumi}{(\roman{enumi})}
\item
The maps $h:[0,T]\times\mathbb{R}^n\to\mathbb{R}$ and $f:[0,T]^2\times U\times\mathbb{R}^n\times\mathbb{R}\times\mathbb{R}\to\mathbb{R}$ are measurable, and there exist a constant $L>0$ and a modulus of continuity $\rho:[0,\infty)\to[0,\infty)$ such that for $\varphi(t,s,u,x,y,z)=h(t,x),f(t,s,u,x,y,z)$, we have
\begin{equation}\label{cost assumption 1}
	\begin{cases}
		|\varphi(t,s,u_1,x_1,y_1,z_1)-\varphi(t,s,u_2,x_2,y_2,z_2)|\\
		\leq L(|x_1-x_2|+|y_1-y_2|+|z_1-z_2|)+\rho(d(u_1,u_2)),\\
		\hspace{2cm}\forall\,(t,s)\in[0,T]^2,\ x_1,x_2\in\mathbb{R}^n,\ y_1,y_2,z_1,z_2\in\mathbb{R},\ u_1,u_2\in U;\\
		|\varphi(t,s,u,0,0,0)|\leq L,\ \forall\,(t,s,u)\in[0,T]^2\times U.
	\end{cases}
\end{equation}
\item
The map $h$ is $C^2$ in $x\in\mathbb{R}^n$ and the map $f$ is $C^2$ in $(x,y,z)\in\mathbb{R}^n\times\mathbb{R}\times\mathbb{R}$. Moreover, there exist a constant $L>0$ and a modulus of continuity $\rho:[0,\infty)\to[0,\infty)$ such that for $\varphi(t,s,u,x,y,z)=h(t,x),f(t,s,u,x,y,z)$, we have
\begin{equation*}
	\begin{cases}
		|D\varphi(t,s,u_1,x_1,y_1,z_1)-D\varphi(t,s,u_2,x_2,y_2,z_2)|\\
		\leq L(|x_1-x_2|+|y_1-y_2|+|z_1-z_2|)+\rho(d(u_1,u_2)),\\
		\hspace{2cm}\forall\,(t,s)\in[0,T]^2,\ x_1,x_2\in\mathbb{R}^n,\ y_1,y_2,z_1,z_2\in\mathbb{R},\ u_1,u_2\in U;\\
		|D^2\varphi(t,s,u_1,x_1,y_1,z_1)-D^2\varphi(t,s,u_2,x_2,y_2,z_2)|\\
		\leq\rho(|x_1-x_2|+|y_1-y_2|+|z_1-z_2|+d(u_1,u_2)),\\
		\hspace{2cm}\forall\,(t,s)\in[0,T]^2,\ x_1,x_2\in\mathbb{R}^n,\ y_1,y_2,z_1,z_2\in\mathbb{R},\ u_1,u_2\in U.
	\end{cases}
\end{equation*}
where $D\varphi$ is the gradient of $\varphi$ with respect to $(x,y,z)$ and $D^2\varphi$ is the Hessian matrix of $\varphi$ with respect to $(x,y,z)$.
\item
There exists a modulus of continuity $\rho:[0,\infty)\to[0,\infty)$ such that for $\varphi(t,s,u,x,y,z)=h(t,x),\partial_xh(t,x),\partial^2_xh(t,x),f(t,s,u,x,y,z),Df(t,s,u,x,y,z),D^2f(t,s,u,x,y,z)$, we have
\begin{equation}\label{cost assumption 2}
\begin{split}
	&|\varphi(t_1,s,u,x,y,z)-\varphi(t_2,s,u,x,y,z)|\leq\rho(|t_1-t_2|)\bigl(1+|x|+|y|+|z|\bigr),\\
	&\hspace{3cm}\forall\,t_1,t_2,s\in[0,T],\ (x,y,z)\in\mathbb{R}^n\times\mathbb{R}\times\mathbb{R},\ u\in U.
\end{split}
\end{equation}
\item
The maps $h,\partial_xh,f,Df$ are $C^1$ in $t\in[0,T]$. Moreover, there exist a constant $L>0$ and a modulus of continuity $\rho:[0,\infty)\to[0,\infty)$ such that for $\varphi(t,s,u,x,y,z)=\partial_th(t,x),\partial_t\partial_xh(t,x),\partial_tf(t,s,u,x,y,z),\partial_tDf(t,s,u,x,y,z)$, we have \eqref{cost assumption 1} and \eqref{cost assumption 2}.
\end{enumerate}
\end{assum}

Under Assumption~\ref{assumption: state SDE}, for each initial condition $(\tau,x_\tau)\in\mathcal{I}$ and control process $u(\cdot)\in\mathcal{U}[\tau,T]$, SDE \eqref{state SDE} has a unique strong solution $X(\cdot)$ which satisfies
\begin{equation*}
	\mathbb{E}\Biggl[\sup_{s\in[\tau,T]}|X(s)|^p\Biggr]\leq C\Bigl(1+\mathbb{E}\bigl[|x_\tau|^p\bigr]\Bigr)<\infty, \ \forall\,p\geq2.
\end{equation*}
Moreover, under Assumption~\ref{assumption: cost BSVIE}, the free term $\psi(t)=h(t,X(T))$ and the generator
\begin{equation*}
	g(t,s,\eta,y,z)=f\bigl(t,s,u(s),X(s),\eta,z\bigr)
\end{equation*}
satisfy Assumption~\ref{assumption: EBSVIE derivative} for any $p\geq2$. Therefore, by Corollary~\ref{corollary: BSVIE}, BSVIE~\eqref{cost BSVIE} has a unique $L^p$-adapted C-solution $(Y(\cdot),Z(\cdot,\cdot))$ for any $p\geq2$. Consequently, the cost functional is well-defined and finite a.s.\ for any $\tau\in[0,T)$. Furthermore, by Theorem~\ref{theorem: EBSVIE derivative}, the map $[\tau,T]\ni t\mapsto Z(t,\cdot)\in L^{p,2}_\mathbb{F}(\tau,T;\mathbb{R})$ is continuously differentiable for any $p\geq2$. Therefore, by Lemma~\ref{lemma: Z diagonal}, there exists a unique process $\mathrm{Diag}[Z](\cdot)\in \bigcap_{p\geq2}L^{p,2}_\mathbb{F}(\tau,T;\mathbb{R})$ such that, for any $p'\geq1$ and $1\leq q'\leq2$, it holds that
\begin{equation*}
	\mathbb{E}\Bigl[\Bigl(\int^{t+\ep}_t\bigl|Z(t,s)-\mathrm{Diag}[Z](s)\bigr|^{q'}\,ds\Bigr)^{p'/q'}\Bigr]^{1/p'}=o(\ep^{1/2+1/q'}),\ \forall\,t\in[\tau,T).
\end{equation*}
\par
Our problem is to find a control process $\hat{u}(\cdot)$ which minimizes the cost functional. However, it is well-known that the problem is time-inconsistent in general. That is, even if $\hat{u}(\cdot)\in\mathcal{U}[t_0,T]$ is an optimal control with respect to a given initial condition $(t_0,x_{t_0})\in\mathcal{I}$, for a future time $t_1\in[t_0,T)$, the restriction $\hat{u}|_{[t_1,T]}(\cdot)$ of $\hat{u}(\cdot)$ on the later time interval $[t_1,T]$ is no longer optimal with respect to the corresponding initial condition $(t_1,X(t_1;t_0,x_{t_0};\hat{u}(\cdot)))\in\mathcal{I}$. For more detailed discussions on the time-inconsistency, see for example \cite{a_Yan-Yong_19}. Instead of seeking for a global optimal control (which does not exist in general), we investigate an open-loop equilibrium control defined as follows.

%% Definition

\begin{defi}\label{definition: open-loop equilibrium control}
Let $(t_0,x_{t_0})\in\mathcal{I}$ be given. We say that a control process $\hat{u}(\cdot)\in\mathcal{U}[t_0,T]$ is an open-loop equilibrium control with respect to the initial condition $(t_0,x_{t_0})$ if, for any $v(\cdot)\in\mathcal{U}[t_0,T]$, any $\tau\in[t_0,T)$, and any nonnegative, bounded and $\mathcal{F}_\tau$-measurable random variable $\xi_\tau$, it holds that
\begin{equation*}
	\liminf_{\ep\downarrow0}\mathbb{E}\Biggl[\frac{J(\tau,\hat{X}(\tau);u^\pert(\cdot))-J(\tau,\hat{X}(\tau);\hat{u}|_{[\tau,T]}(\cdot))}{\ep}\,\xi_\tau\Biggr]\geq0,
\end{equation*}
where $\hat{X}(\cdot):=X(\cdot;t_0,x_{t_0};\hat{u}(\cdot))$, $\hat{u}|_{[\tau,T]}(\cdot)\in\mathcal{U}[\tau,T]$ is the restriction of $\hat{u}(\cdot)$ on $[\tau,T]$, and $u^\pert(\cdot)\in\mathcal{U}[\tau,T]$ is defined by
\begin{equation}\label{perturbed control}
	u^\pert(s):=
	\begin{cases}
		v(s)\ \text{for}\ s\in[\tau,\tau+\ep),\\
		\hat{u}(s)\ \text{for}\ s\in[\tau+\ep,T].
	\end{cases}
\end{equation}
\end{defi}

%% Remark

\begin{rem}
\begin{enumerate}
\renewcommand{\labelenumi}{(\roman{enumi})}
\item
The above definition is slightly different from the original definition given by Hu--Jin--Zhou~\cite{a_Hu-Jin-Zhou_12,a_Hu-Jin-Zhou_17}, where the ``definition'' of an open-loop equilibrium control is given by
\begin{equation}\label{definition'}
	\liminf_{\ep\downarrow0}\frac{J(\tau,\hat{X}(\tau);u^\pert(\cdot))-J(\tau,\hat{X}(\tau);\hat{u}|_{[\tau,T]}(\cdot))}{\ep}\geq0\ \text{a.s.}
\end{equation}
However, since $\{J(\tau,\hat{X}(\tau);u^\pert(\cdot))\}_{\ep>0}$ is an uncountable family of random variables, and the a.s.\,limit as $\ep\downarrow0$ along the whole $\ep>0$ may not be well-defined, the above ``definition'' is not suitable for our problem. Note that if there exists a modification of the family $\{J(\tau,\hat{X}(\tau);u^\pert(\cdot))\}_{\ep>0}$ which is a.s.\ continuous (with respect to $\ep$), then \eqref{definition'} makes sense for such a modification. If furthermore the family $\Bigl\{\Bigl(\frac{J(\tau,\hat{X}(\tau);u^\pert(\cdot))-J(\tau,\hat{X}(\tau);\hat{u}|_{[\tau,T]}(\cdot))}{\ep}\Bigr)^{-}\Bigr\}_{\ep>0}$ is uniformly integrable, then by Fatou's lemma we see that \eqref{definition'} implies our definition. However, since the existence of the continuous modification is questionable, we should avoid to use \eqref{definition'} as the definition in our problem. Thus, we defined an open-loop equilibrium control by a weak sense, which is well-defined in general. In fact, it turns out that if $\hat{u}(\cdot)\in\mathcal{U}[t_0,T]$ is an open-loop equilibrium control with respect to $(t_0,x_{t_0})\in\mathcal{I}$ in the sense of Definition~\ref{definition: open-loop equilibrium control}, then for any $v(\cdot)\in\mathcal{U}[t_0,T]$ and $\tau\in[t_0,T)$, there exists a sequence $\{\ep_k\}_{k\in\mathbb{N}}\subset(0,T-\tau)$ (which depends on $\hat{u}(\cdot),\,v(\cdot)$ and $\tau$) such that $\lim_{k\to\infty}\ep_k=0$ and
\begin{equation*}
	\liminf_{k\to\infty}\frac{J(\tau,\hat{X}(\tau);u^{\tau,\ep_k}(\cdot))-J(\tau,\hat{X}(\tau);\hat{u}|_{[\tau,T]}(\cdot))}{\ep_k}\geq0\ \text{a.s.};
\end{equation*}
see Remark~\ref{remark: last remark}. We also remark that our definition is consistent with the game theoretic formulation usually discussed in the literature of continuous-time time-inconsistent stochastic control problems; see for example \cite{a_Bjork-_17}.
\item
The concept of open-loop equilibrium controls is time-consistent. Indeed, if $\hat{u}(\cdot)\in\mathcal{U}[t_0,T]$ is an open-loop equilibrium control with respect to a given initial condition $(t_0,x_{t_0})\in\mathcal{I}$, then, for any future time $t_1\in[t_0,T)$, the restriction $\hat{u}|_{[t_1,T]}(\cdot)\in\mathcal{U}[t_1,T]$ of $\hat{u}(\cdot)$ on the later time interval $[t_1,T]$ is also an open-loop equilibrium control with respect to the corresponding initial condition $(t_1,X(t_1;t_0,x_{t_0};\hat{u}(\cdot)))\in\mathcal{I}$.
\end{enumerate}
\end{rem}

Our goal is to characterize an open-loop equilibrium control by using variational methods. The key point is to derive the first-order and the second-order adjoint equations. Firstly, let us state our main result. The proof will be given in Section~\ref{section: proof of main result}.
\par
\medskip
Let $(t_0,x_{t_0})\in\mathcal{I}$ and $\hat{u}(\cdot)\in\mathcal{U}[t_0,T]$ be given. Denote by $(\hat{X}(\cdot),\hat{Y}(\cdot),\hat{Z}(\cdot,\cdot))$ the corresponding triplet, that is,
\begin{equation}\label{equilibrium triplet}
	(\hat{X}(\cdot),\hat{Y}(\cdot),\hat{Z}(\cdot,\cdot)):=(X(\cdot;t_0,x_{t_0};\hat{u}(\cdot)),Y(\cdot;t_0,x_{t_0};\hat{u}(\cdot)),Z(\cdot,\cdot;t_0,x_{t_0};\hat{u}(\cdot))).
\end{equation}
We use the following notation:
\begin{equation}\label{notation 1}
	\varphi(s)=\varphi(s,\hat{u}(s),\hat{X}(s)),\ \varphi_x(s)=\partial_x\varphi(s,\hat{u}(s),\hat{X}(s)),\ \varphi_{xx}(s)=\partial^2_x\varphi(s,\hat{u}(s),\hat{X}(s)),
\end{equation}
for $\varphi=b,\sigma$, and
\begin{equation}\label{notation 2}
	\begin{cases}
		h(t)=h(t,\hat{X}(T)),\ h_x(t)=\partial_xh(t,\hat{X}(T)),\ h_{xx}(t)=\partial^2_xh(t,\hat{X}(T)),\\
		f(t,s)=f(t,s,\hat{u}(s),\hat{X}(s),\hat{Y}(s),\hat{Z}(t,s)),\\
		f_\alpha(t,s)=\partial_\alpha f(t,s,\hat{u}(s),\hat{X}(s),\hat{Y}(s),\hat{Z}(t,s)),\ \alpha=x,y,z,\\
		D^2f(t,s)=D^2f(t,s,\hat{u}(s),\hat{X}(s),\hat{Y}(s),\hat{Z}(t,s)).
	\end{cases}
\end{equation}
\par
We introduce the following two EBSVIEs for $(p(\cdot,\cdot),q(\cdot,\cdot))$ and $(P(\cdot,\cdot),Q(\cdot,\cdot))$, respectively:
\begin{equation}\label{EBSVIE 1}
\begin{split}
	&p(t,s)=h_x(t)+\int^T_s\Bigl\{b^\top_x(r)p(t,r)+\sigma^\top_x(r)q(t,r)+f_z(t,r)\bigl(\sigma^\top_x(r)p(t,r)+q(t,r)\bigr)\\
	&\hspace{4cm}+f_y(t,r)p(r,r)+f_x(t,r)\Bigr\}\,dr\\
	&\hspace{2cm}-\int^T_sq(t,r)\,dW(r),\ \ s\in[t_0,T],\ t\in[t_0,T],
\end{split}
\end{equation}
and
\begin{equation}\label{EBSVIE 2}
\begin{split}
	&P(t,s)\\
	&=h_{xx}(t)\\
	&\hspace{0.5cm}+\int^T_s\Bigl\{b^\top_x(r)P(t,r)+P(t,r)b_x(r)+\sigma^\top_x(r)P(t,r)\sigma_x(r)+\sigma^\top_x(r)Q(t,r)+Q(t,r)\sigma_x(r)\\
	&\hspace{0.8cm}+f_z(t,r)\bigl(\sigma^\top_x(r)P(t,r)+P(t,r)\sigma_x(r)+Q(t,r)\bigr)+f_y(t,r)P(r,r)\\
	&\hspace{0.8cm}+b^\top_{xx}(r)p(t,r)+\sigma^\top_{xx}(r)\bigl(f_z(t,r)p(t,r)\mathalpha{+}q(t,r)\bigr)\\
	&\hspace{0.8cm}+[I_{n\times n},p(r,r),\sigma^\top_x(r)p(t,r)\mathalpha{+}q(t,r)]D^2f(t,r)[I_{n\times n},p(r,r),\sigma^\top_x(r)p(t,r)\mathalpha{+}q(t,r)]^\top\Bigr\}\,dr\\
	&\hspace{0.5cm}-\int^T_sQ(t,r)\,dW(r),\ \ s\in[t_0,T],\ t\in[t_0,T],
\end{split}
\end{equation}
where $b^\top_{xx}(r)p(t,r):=\sum^n_{i=1}p^i(t,r)b^i_{xx}(r)$ and $\sigma^\top_{xx}(r)(f_z(t,r)p(t,r)+q(t,r))$ is defined similarly. We call EBSVIE~\eqref{EBSVIE 1} the \emph{first-order adjoint equation} and EBSVIE~\eqref{EBSVIE 2} the \emph{second-order adjoint equation} in the spirit of the stochastic maximum principle. The above adjoint equations are natural generalizations of that of time-consistent problems~\cite{a_Peng_90,b_Yong-Zhou_99,a_Hu_17} and a time-inconsistent problem with an additive cost functional~\cite{a_Yan-Yong_19} (Section 4). These equations become EBSVIEs due to the dependency on $f_y(t,r)p(r,r)$ and $f_y(t,r)P(r,r)$ of the generators. Note that the coefficients of EBSVIEs~\eqref{EBSVIE 1} and \eqref{EBSVIE 2} are not continuous with respect $t$ in the pointwise sense, and hence they are beyond the literature \cite{a_Wang_20}. Alternatively, we can easily check that they satisfy the weak continuity assumption~\eqref{weak continuity} and any other conditions in Assumption~\ref{assumption: EBSVIE} for any $p\geq2$; see Remark~\ref{remark: t-continuity}. Moreover, by Assumption~\ref{assumption: cost BSVIE}~(\rnum{4}), we see that the coefficients of the first-order adjoint equation~\eqref{EBSVIE 1} satisfy Assumption~\ref{assumption: EBSVIE derivative}. Thus, by Theorems~\ref{theorem: EBSVIE} and \ref{theorem: EBSVIE derivative}, together with Lemma~\ref{lemma: Z diagonal}, we obtain the following proposition.

%% Proposition

\begin{prop}\label{proposition: EBSVIE}
EBSVIEs \eqref{EBSVIE 1} and \eqref{EBSVIE 2} have unique adapted C-solutions
\begin{equation*}
\begin{split}
	&(p(\cdot,\cdot),q(\cdot,\cdot))\in\bigcap_{p\geq2}\mathfrak{H}^p_\mathbb{F}(t_0,T;\mathbb{R}^n\times\mathbb{R}^n)\ \text{and}\\
	&(P(\cdot,\cdot),Q(\cdot,\cdot))\in\bigcap_{p\geq2}\mathfrak{H}^p_\mathbb{F}(t_0,T;\mathbb{R}^{n\times n}\times\mathbb{R}^{n\times n}),
\end{split}
\end{equation*}
respectively. Moreover, for any $p\geq2$, the map $t\mapsto q(t,\cdot)$ is in $C^1([t_0,T];L^{p,2}_\mathbb{F}(t_0,T;\mathbb{R}^n))$, and there exists a unique process $\mathrm{Diag}[q](\cdot)\in \bigcap_{p\geq2}L^{p,2}_\mathbb{F}(t_0,T;\mathbb{R}^n)$ such that, for any $p'\geq1$ and $1\leq q'\leq2$, it holds that
\begin{equation}\label{q-diagonal estimate}
	\mathbb{E}\Bigl[\Bigl(\int^{t+\ep}_t\bigl|q(t,s)-\mathrm{Diag}[q](s)\bigr|^{q'}\,ds\Bigr)^{p'/q'}\Bigr]^{1/p'}=o(\ep^{1/2+1/q'}),\ \forall\,t\in[t_0,T).
\end{equation}
\end{prop}

Define the $\mathcal{H}$-function $[t_0,T]\times\Omega\times U\ni(s,\omega,v)\mapsto\mathcal{H}(s,\omega,v;t_0,x_{t_0},\hat{u}(\cdot))\in\mathbb{R}$ by
\begin{equation}\label{H-function}
\begin{split}
	&\mathcal{H}(s,\omega,v;t_0,x_{t_0};\hat{u}(\cdot))\\
	&:=\bigl\langle p(s,s),b\bigl(s,v,\hat{X}(s)\bigr)\bigr\rangle+\bigl\langle\mathrm{Diag}[q](s),\sigma\bigl(s,v,\hat{X}(s)\bigr)\bigr\rangle\\
	&\hspace{1cm}+f\Bigl(s,s,v,\hat{X}(s),\hat{Y}(s),\mathrm{Diag}[\hat{Z}](s)+\bigl\langle p(s,s),\sigma\bigl(s,v,\hat{X}(s)\bigr)-\sigma\bigl(s,\hat{u}(s),\hat{X}(s)\bigr)\bigr\rangle\Bigr)\\
	&\hspace{1cm}+\frac{1}{2}\Bigl\langle P(s,s)\bigl(\sigma\bigl(s,v,\hat{X}(s)\bigr)-\sigma\bigl(s,\hat{u}(s),\hat{X}(s)\bigr)\bigr),\sigma\bigl(s,v,\hat{X}(s)\bigr)-\sigma\bigl(s,\hat{u}(s),\hat{X}(s)\bigr)\Bigr\rangle
\end{split}
\end{equation}
for $(s,\omega,v)\in[t_0,T]\times\Omega\times U$ (where we suppressed the dependency on $\omega\in\Omega$ in the right-hand side). Note that the state process $\hat{X}(\cdot)$, the cost process $(\hat{Y}(\cdot),\hat{Z}(\cdot,\cdot))$, and the adjoint processes $(p(\cdot,\cdot),q(\cdot,\cdot))$, $(P(\cdot,\cdot),Q(\cdot,\cdot))$ are uniquely determined by each initial condition $(t_0,x_{t_0})\in\mathcal{I}$ and control process $\hat{u}(\cdot)\in\mathcal{U}[t_0,T]$. Thus, the $\mathcal{H}$-function is also uniquely determined by them. Note also that the $\mathcal{H}$-function is progressively measurable and continuous in $v\in U$ for $\mathrm{Leb}_{[t_0,T]}\otimes\mathbb{P}$-a.e.\,$(s,\omega)\in[t_0,T]\times\Omega$. Now let us state our main result.

%% Theorem

\begin{theo}\label{theorem: main result}
Let $(t_0,x_{t_0})\in\mathcal{I}$ be given. Then $\hat{u}(\cdot)\in\mathcal{U}[t_0,T]$ is an open-loop equilibrium control with respect to the initial condition $(t_0,x_{t_0})$ if and only if
\begin{equation}\label{characterization}
\begin{split}
	&\mathcal{H}(s,\omega,v;t_0,x_{t_0};\hat{u}(\cdot))\geq\mathcal{H}(s,\omega,\hat{u}(s,\omega);t_0,x_{t_0};\hat{u}(\cdot)),\ \ \forall\,v\in U,\\
	&\hspace{3cm}\text{for}\ \mathrm{Leb}_{[t_0,T]}\otimes\mathbb{P}\text{-a.e.}\,(s,\omega)\in[t_0,T]\times\Omega.
\end{split}
\end{equation}
\end{theo}

%% Remark

\begin{rem}
\begin{enumerate}
\renewcommand{\labelenumi}{(\roman{enumi})}
\item
Our result is an extension of that of Yan--Yong~\cite{a_Yan-Yong_19} (Section 4), where the authors investigated an open-loop equilibrium control in a time-inconsistent stochastic control problem for a cost functional defined by just a conditional expectation of a function of states and controls, that is, an additive cost functional. Our result generalizes their result to the case of a recursive cost functional. On one hand, due to the difficulty to treat the ``diagonal processes'' of $q(\cdot,\cdot)$ and $\hat{Z}(\cdot,\cdot)$, the characterization result stated in Theorem~1 of \cite{a_Yan-Yong_19} remained to include a limit procedure, and hence they did not provide a full characterization in a local form like \eqref{characterization}. On the other hand, we overcame the difficulty by introducing the operator $\mathrm{Diag}[\cdot]$ which we defined in Section~\ref{subsection: two time-parameters}. We emphasize that Assumption~\ref{assumption: cost BSVIE}~(\rnum{4}) guarantees the well-definedness of $\mathrm{Diag}[q](\cdot)$ and $\mathrm{Diag}[\hat{Z}](\cdot)$ via Theorem~\ref{theorem: EBSVIE derivative}.
\item
Let us remark on the setting of the problem. The assumptions of uniform boundedness of $b(s,u,0),\,\sigma(s,u,0),\,f(t,s,u,0,0,0)$, with respect to $u$, and/or the Lipschitz continuity of $h,\,f$, with respect to $x$, exclude the case of linear-quadratic control problems. Besides, the requirement of the initial state $x_{t_0}$ being in $L^p_{\mathcal{F}_{t_0}}(\Omega;\mathbb{R}^n)$ for any $p\geq1$ may seem to be too strong. However, since the main goal of this paper is to derive proper forms of the adjoint equations which characterize open-loop equilibrium controls, we do not pursue the most generality here. We remark that Assumptions~\ref{assumption: state SDE} and \ref{assumption: cost BSVIE} are generalizations of the assumptions~(S0)--(S3) in the textbook~\cite{b_Yong-Zhou_99} to our problem. Also, by a careful observation of discussions in Section~\ref{section: proof of main result}, we see that it suffices to assume that the initial state $x_{t_0}$ is in $L^8_{\mathcal{F}_{t_0}}(\Omega;\mathbb{R}^n)$.
\end{enumerate}
\end{rem}

%%%%%%%%%%%%%%%
%% Section
%%%%%%%%%%%%%%%

\section{Proof of Theorem~\ref{theorem: main result}: Variational methods}\label{section: proof of main result}

In this section, we derive the adjoint equations~\eqref{EBSVIE 1} and \eqref{EBSVIE 2}, and prove Theorem~\ref{theorem: main result}. Proofs of some technical estimates are given in Appendix~\ref{appendix}.
\par
Suppose that we are given an initial condition $(t_0,x_{t_0})\in\mathcal{I}$ and a control process $\hat{u}(\cdot)\in\mathcal{U}[t_0,T]$. As in Section~\ref{section: control problem}, we denote by $(\hat{X}(\cdot),\hat{Y}(\cdot),\hat{Z}(\cdot,\cdot))$ the corresponding triplet; see \eqref{equilibrium triplet}. Fix $v(\cdot)\in\mathcal{U}[t_0,T]$ and $\tau\in[t_0,T)$. For each $\ep\in(0,T-\tau)$, define the perturbed triplet by
\begin{equation*}
\begin{split}
	&(X^\pert(\cdot),Y^\pert(\cdot),Z^\pert(\cdot,\cdot))\\
	&:=(X(\cdot;\tau,\hat{X}(\tau);u^\pert(\cdot)),Y(\cdot;\tau,\hat{X}(\tau);u^\pert(\cdot)),Z(\cdot,\cdot;\tau,\hat{X}(\tau);u^\pert(\cdot))),
\end{split}
\end{equation*}
where $u^\pert(\cdot)\in\mathcal{U}[\tau,T]$ is defined by \eqref{perturbed control}. Then we have that
\begin{equation*}
	J(\tau,\hat{X}(\tau);\hat{u}|_{[\tau,T]}(\cdot))=\hat{Y}(\tau)\ \text{and}\ J(\tau,\hat{X}(\tau);u^\pert(\cdot))=Y^\pert(\tau).
\end{equation*}
In the following, in addition to the notations~\eqref{notation 1} and \eqref{notation 2}, we use the following notation. For $\varphi=b,\sigma$,
\begin{equation*}
	\delta\varphi(s)=\varphi(s,v(s),\hat{X}(s))-\varphi(s)\ \text{and}\ \delta\varphi_x(s)=\partial_x\varphi(s,v(s),\hat{X}(s))-\varphi_x(s).
\end{equation*}
For each $\ep\in(0,T-\tau)$, we consider the following SDEs on $[\tau,T]$:
\begin{align*}
	&\begin{cases}
		dX^\pert_1(s)=b_x(s)X^\pert_1(s)\,ds+\bigl(\sigma_x(s)X^\pert_1(s)+\delta\sigma(s)\1_{[\tau,\tau+\ep)}(s)\bigr)\,dW(s),\ s\in[\tau,T],\\
		X^\pert_1(\tau)=0,
	\end{cases}
	\shortintertext{and}
	&\begin{cases}
		dX^\pert_2(s)=\bigl(b_x(s)X^\pert_2(s)+\delta b(s)\1_{[\tau,\tau+\ep)}(s)+\frac{1}{2}b_{xx}(s)X^\pert_1(s)X^\pert_1(s)\bigr)\,ds\\
			\hspace{2cm}+\bigl(\sigma_x(s)X^\pert_2(s)+\delta\sigma_x(s)X^\pert_1(s)\1_{[\tau,\tau+\ep)}(s)+\frac{1}{2}\sigma_{xx}(s)X^\pert_1(s)X^\pert_1(s)\bigr)\,dW(s),\\
			\hspace{7cm}s\in[\tau,T],\\
		X^\pert_2(\tau)=0,
	\end{cases}
\end{align*}
where $b_{xx}(s)X^\pert_1(s)X^\pert_1(s)=\bigl(\text{tr}[b^1_{xx}(s)X^\pert_1(s)X^\pert_1(s)^\top],\mathalpha{\dots},\text{tr}[b^n_{xx}(s)X^\pert_1(s)X^\pert_1(s)^\top]\bigr)^\top$ and similar for $\sigma_{xx}(s)X^\pert_1(s)X^\pert_1(s)$. The above SDEs are the first-order and the second-order variational equations for the sate equation~\eqref{state SDE} obtained by Peng in \cite{a_Peng_90}. The following lemma is well-known; see for example~\cite{a_Peng_90,b_Yong-Zhou_99}.

%% Lemma

\begin{lemm}\label{lemma: variation state}
For any $p\geq1$, it holds that
\begin{align*}
	&\mathbb{E}\Biggl[\sup_{s\in[\tau,T]}|X^\pert_1(s)|^{2p}\Biggr]=O(\ep^p),\displaybreak[1]\\
	&\mathbb{E}\Biggl[\sup_{s\in[\tau,T]}|X^\pert_2(s)|^{2p}\Biggr]=O(\ep^{2p}),\displaybreak[1]\\
	&\mathbb{E}\Biggl[\sup_{s\in[\tau,T]}|X^\pert(s)-\hat{X}(s)|^{2p}\Biggr]=O(\ep^p),\displaybreak[1]\\
	&\mathbb{E}\Biggl[\sup_{s\in[\tau,T]}|X^\pert(s)-\hat{X}(s)-X^\pert_1(s)|^{2p}\Biggr]=O(\ep^{2p}),\displaybreak[1]\\
	&\mathbb{E}\Biggl[\sup_{s\in[\tau,T]}|X^\pert(s)-\hat{X}(s)-X^\pert_1(s)-X^\pert_2(s)|^{2p}\Biggr]=o(\ep^{2p}).
\end{align*}
In particular, for any $\varphi\in C^2(\mathbb{R}^n)$ with $\partial^2_x\varphi$ being bounded, it holds that
\begin{align*}
	&\mathbb{E}\Bigl[\Bigl|\varphi(X^\pert(T))-\varphi(\hat{X}(T))\\
	&\hspace{1cm}-\langle\partial_x\varphi(\hat{X}(T)),X^\pert_1(T)+X^\pert_2(T)\rangle-\frac{1}{2}\langle\partial^2_x\varphi(\hat{X}(T))X^\pert_1(T),X^\pert_1(T)\rangle\Bigr|^2\Bigr]=o(\ep^2).
\end{align*}
\end{lemm}

Now we derive the first-order and the second-order adjoint equations \eqref{EBSVIE 1} and \eqref{EBSVIE 2}. To do so, let us consider two BSDEs parametrized by $t\in[t_0,T]$ of the following forms:
\begin{align*}
	&\begin{cases}
		dp(t,s)=-k(t,s)\,ds+q(t,s)\,dW(s),\ s\in[t_0,T],\\
		p(t,T)=h_x(t),
	\end{cases}
	\shortintertext{and}
	&\begin{cases}
		dP(t,s)=-K(t,s)\,ds+Q(t,s)\,dW(s),\ s\in[t_0,T],\\
		P(t,T)=h_{xx}(t),
	\end{cases}
\end{align*}
for some measurable maps
\begin{equation*}
	k(\cdot,\cdot)\in\bigcap_{p\geq1}C([t_0,T];L^{p,1}_\mathbb{F}(t_0,T;\mathbb{R}^n))\ \text{and}\ K(\cdot,\cdot)\in\bigcap_{p\geq1}C([t_0,T];L^{p,1}_\mathbb{F}(t_0,T;\mathbb{R}^{n\times n})).
\end{equation*}
We will determine the precise forms of $k(\cdot,\cdot)$ and $K(\cdot,\cdot)$ later (see \eqref{k} and \eqref{K}). It can be easily shown that there exists a unique solutions $(p(\cdot,\cdot),q(\cdot,\cdot))\in\bigcap_{p\geq1}\mathfrak{H}^p_\mathbb{F}(t_0,T;\mathbb{R}^n\times\mathbb{R}^n)$ and $(P(\cdot,\cdot),Q(\cdot,\cdot))\in\bigcap_{p\geq1}\mathfrak{H}^p_\mathbb{F}(t_0,T;\mathbb{R}^{n\times n}\times\mathbb{R}^{n\times n})$ of the above equations.
\par
For each $t\in[\tau,T]$, by applying It\^{o}'s formula to the processes $\langle p(t,\cdot),X^\pert_1(\cdot)+X^\pert_2(\cdot)\rangle$ and $\langle P(t,\cdot)X^\pert_1(\cdot),X^\pert_1(\cdot)\rangle$ on $[t,T]$, we have that
\begin{align*}
	&\langle h_x(t),X^\pert_1(T)+X^\pert_2(T)\rangle+\frac{1}{2}\langle h_{xx}(t)X^\pert_1(T),X^\pert_1(T)\rangle\\
	&=\langle p(t,T),X^\pert_1(T)+X^\pert_2(T)\rangle+\frac{1}{2}\langle P(t,T)X^\pert_1(T),X^\pert_1(T)\rangle\\
% ds
	&=\langle p(t,t),X^\pert_1(t)+X^\pert_2(t)\rangle+\frac{1}{2}\langle P(t,t)X^\pert_1(t),X^\pert_1(t)\rangle\\
	&\hspace{0.5cm}+\int^T_t\Bigl\{\bigl\langle p(t,s),b_x(s)\bigl(X^\pert_1(s)+X^\pert_2(s)\bigr)+\delta b(s)\1_{[\tau,\tau+\ep)}(s)+\frac{1}{2}b_{xx}(s)X^\pert_1(s)X^\pert_1(s)\bigr\rangle\\
	&\hspace{2cm}-\bigl\langle k(t,s),X^\pert_1(s)+X^\pert_2(s)\bigr\rangle\\
	&\hspace{2cm}+\bigl\langle q(t,s),\sigma_x(s)\bigl(X^\pert_1(s)+X^\pert_2(s)\bigr)+\delta\sigma(s)\1_{[\tau,\tau+\ep)}(s)\\
	&\hspace{4cm}+\delta\sigma_x(s)X^\pert_1(s)\1_{[\tau,\tau+\ep)}(s)+\frac{1}{2}\sigma_{xx}(s)X^\pert_1(s)X^\pert_1(s)\bigr\rangle\\
	&\hspace{2cm}+\frac{1}{2}\bigl\langle\bigl(P(t,s)b_x(s)+b^\top_x(s)P(t,s)-K(t,s)\bigr)X^\pert_1(s),X^\pert_1(s)\bigr\rangle\\
	&\hspace{2cm}+\frac{1}{2}\bigl\langle P(t,s)\bigl(\sigma_x(s)X^\pert_1(s)+\delta\sigma(s)\1_{[\tau,\tau+\ep)}(s)\bigr),\sigma_x(s)X^\pert_1(s)+\delta\sigma(s)\1_{[\tau,\tau+\ep)}(s)\bigr\rangle\\
	&\hspace{2cm}+\frac{1}{2}\bigl\langle \bigl(Q(t,s)+Q^\top(t,s)\bigr)X^\pert_1(s),\sigma_x(s)X^\pert_1(s)+\delta\sigma(s)\1_{[\tau,\tau+\ep)}(s)\bigr\rangle\Bigr\}\,ds\\
% dW(s)
	&\hspace{0.5cm}+\int^T_t\Bigl\{\bigl\langle p(t,s),\sigma_x(s)\bigl(X^\pert_1(s)+X^\pert_2(s)\bigr)+\delta\sigma(s)\1_{[\tau,\tau+\ep)}(s)\\
	&\hspace{3.5cm}+\delta\sigma_x(s)X^\pert_1(s)\1_{[\tau,\tau+\ep)}(s)+\frac{1}{2}\sigma_{xx}(s)X^\pert_1(s)X^\pert_1(s)\bigr\rangle\\
	&\hspace{2cm}+\bigl\langle q(t,s),X^\pert_1(s)+X^\pert_2(s)\bigr\rangle\\
	&\hspace{2cm}+\frac{1}{2}\bigl\langle \bigl(P(t,s)+P^\top(t,s)\bigr)X^\pert_1(s),\sigma_x(s)X^\pert_1(s)+\delta\sigma(s)\1_{[\tau,\tau+\ep)}(s)\bigr\rangle\\
	&\hspace{2cm}+\frac{1}{2}\bigl\langle Q(t,s)X^\pert_1(s),X^\pert_1(s)\bigr\rangle\Bigr\}\,dW(s)\displaybreak[1]\\
% ds
	&=\langle p(t,t),X^\pert_1(t)+X^\pert_2(t)\rangle+\frac{1}{2}\langle P(t,t)X^\pert_1(t),X^\pert_1(t)\rangle\\
	&\hspace{0.5cm}+\int^T_t\Bigl\{\alpha(t,s)\1_{[\tau,\tau+\ep)}(s)+\langle A_1(t,s),X^\pert_1(s)+X^\pert_2(s)\rangle+\frac{1}{2}\langle A_2(t,s)X^\pert_1(s),X^\pert_1(s)\rangle\\
	&\hspace{2cm}+\langle A_3(t,s),X^\pert_1(s)\rangle\1_{[\tau,\tau+\ep)}(s)\Bigr\}\,ds\\
% dW(s)
	&\hspace{0.5cm}+\int^T_t\Bigl\{\beta(t,s)\1_{[\tau,\tau+\ep)}(s)+\langle B_1(t,s),X^\pert_1(s)+X^\pert_2(s)\rangle+\frac{1}{2}\langle B_2(t,s)X^\pert_1(s),X^\pert_1(s)\rangle\\
	&\hspace{2cm}+\langle B_3(t,s),X^\pert_1(s)\rangle\1_{[\tau,\tau+\ep)}(s)\Bigr\}\,dW(s),
\end{align*}
where
\begin{align*}
	&\alpha(t,s):=\langle p(t,s),\delta b(s)\rangle+\langle q(t,s),\delta\sigma(s)\rangle+\frac{1}{2}\langle P(t,s)\delta\sigma(s),\delta\sigma(s)\rangle\ \in\mathbb{R},\\
	&A_1(t,s):=b^\top_x(s)p(t,s)+\sigma^\top_x(s)q(t,s)-k(t,s)\ \in\mathbb{R}^n,\\
	&A_2(t,s):=b^\top_{xx}(s)p(t,s)+\sigma^\top_{xx}(s)q(t,s)+P(t,s)b_x(s)+b^\top_x(s)P(t,s)\\
	&\hspace{2cm}+Q(t,s)\sigma_x(s)+\sigma^\top_x(s)Q(t,s)+\sigma^\top_x(s)P(t,s)\sigma_x(s)-K(t,s)\ \in\mathbb{R}^{n\times n},\\
	&A_3(t,s):=\delta\sigma^\top_x(s)q(t,s)+\frac{1}{2}\bigl(\sigma^\top_x(s)P(t,s)+\sigma^\top_x(s)P^\top(t,s)+Q(t,s)+Q^\top(t,s)\bigr)\delta\sigma(s)\ \in\mathbb{R}^n,\displaybreak[1]\\
	&\beta(t,s):=\langle p(t,s),\delta\sigma(s)\rangle\ \in\mathbb{R},\\
	&B_1(t,s):=\sigma^\top_x(s)p(t,s)+q(t,s)\ \in\mathbb{R}^n,\\
	&B_2(t,s):=\sigma^\top_{xx}(s)p(t,s)+P(t,s)\sigma_x(s)+\sigma^\top_x(s)P(t,s)+Q(t,s)\ \in\mathbb{R}^{n\times n},\\
	&B_3(t,s):=\delta\sigma^\top_x(s)p(t,s)+\frac{1}{2}\bigl(P(t,s)+P^\top(t,s)\bigr)\delta\sigma(s)\ \in\mathbb{R}^n.
\end{align*}

%% Remark

\begin{rem}
\begin{enumerate}
\renewcommand{\labelenumi}{(\roman{enumi})}
\item
The convergence rates of the terms $\alpha(t,s)\1_{[\tau,\tau+\ep)}(s)$ and $\beta(t,s)\1_{[\tau,\tau+\ep)}(s)$ cannot be improved anymore by Taylor expansions. As in the literature~\cite{a_Hu_17}, we include these terms in the variation of the backward equation.
\item
The terms $A_1(\cdot,\cdot)$ and $A_2(\cdot,\cdot)$ include the undetermined processes $k(\cdot,\cdot)$ and $K(\cdot,\cdot)$, respectively, while the terms $B_1(\cdot,\cdot),\,B_2(\cdot,\cdot),\,B_3(\cdot,\cdot)$ do not include either these processes.
\item
We can show that (see Lemma~\ref{lemma: A1} in the appendix)
\begin{equation}\label{estimate A_3}
	\sup_{t\in[\tau,T]}\mathbb{E}\Bigl[\Bigl|\int^T_t\langle A_3(t,s),X^\pert_1(s)\rangle\1_{[\tau,\tau+\ep)}(s)\,ds\Bigr|^2\Bigr]=o(\ep^2).
\end{equation}
\end{enumerate}
\end{rem}

Set
\begin{equation}\label{eta zeta}
	\begin{cases}
		\eta^\pert(t):=\langle p(t,t),X^\pert_1(t)+X^\pert_2(t)\rangle+\frac{1}{2}\langle P(t,t)X^\pert_1(t),X^\pert_1(t)\rangle,\ t\in[\tau,T],\\
		\zeta^\pert(t,s):=\langle B_1(t,s),X^\pert_1(s)+X^\pert_2(s)\rangle+\frac{1}{2}\langle B_2(t,s)X^\pert_1(s),X^\pert_1(s)\rangle\\
		\hspace{2cm}+\langle B_3(t,s),X^\pert_1(s)\rangle\1_{[\tau,\tau+\ep)}(s),\ (t,s)\in[\tau,T]^2,
	\end{cases}
\end{equation}
\begin{equation*}
	\begin{cases}
		\tilde{Y}^\pert(t):=Y^\pert(t)-\eta^\pert(t),\ t\in[\tau,T],\\
		\tilde{Z}^\pert(t,s):=Z^\pert(t,s)-\beta(t,s)\1_{[\tau,\tau+\ep)}(s)-\zeta^\pert(t,s),\ (t,s)\in[\tau,T]^2,
	\end{cases}
\end{equation*}
and
\begin{equation*}
	\begin{cases}
		\bar{Y}^\pert(t):=\tilde{Y}^\pert(t)-\hat{Y}(t),\ t\in[\tau,T],\\
		\bar{Z}^\pert(t,s):=\tilde{Z}^\pert(t,s)-\hat{Z}(t,s),\ (t,s)\in[\tau,T]^2.
	\end{cases}
\end{equation*}
Since $\eta^\pert(\tau)=0$, we have
\begin{equation}\label{perturbation cost}
	J(\tau,\hat{X}(\tau);u^\pert(\cdot))-J(\tau,\hat{X}(\tau);\hat{u}|_{[\tau,T]}(\cdot))=Y^\pert(\tau)-\hat{Y}(\tau)=\bar{Y}^\pert(\tau).
\end{equation}
Furthermore, $(\bar{Y}^\pert(\cdot),\bar{Z}^\pert(\cdot,\cdot))$ satisfies the following BSVIE:
\begin{align*}
	\bar{Y}^\pert(t)=&\psi^\pert_1(t)-\int^T_t\bar{Z}^\pert(t,s)\,dW(s)\\
	&+\int^T_t\Bigl\{f\bigl(t,s,u^\pert(s),X^\pert(s),Y^\pert(s),Z^\pert(t,s)\bigr)-f(t,s)\\
	&\hspace{1.5cm}+\langle A_1(t,s),X^\pert_1(s)+X^\pert_2(s)\rangle+\frac{1}{2}\langle A_2(t,s)X^\pert_1(s),X^\pert_1(s)\rangle\\
	&\hspace{1.5cm}+\alpha(t,s)\1_{[\tau,\tau+\ep)}(s)\Bigr\}\,ds,\hspace{1cm}t\in[\tau,T],
\end{align*}
where
\begin{align*}
	\psi^\pert_1(t):=&h(t,X^\pert(T))-h(t)-\langle h_x(t),X^\pert_1(T)+X^\pert_2(T)\rangle-\frac{1}{2}\langle h_{xx}(t)X^\pert_1(T),X^\pert_1(T)\rangle\\
	&\hspace{1cm}+\int^T_t\langle A_3(t,s),X^\pert_1(s)\rangle\1_{[\tau,\tau+\ep)}(s)\,ds.
\end{align*}
By Lemmas~\ref{lemma: variation state} and the estimate~\eqref{estimate A_3}, we see that
\begin{equation*}
	\sup_{t\in[\tau,T]}\mathbb{E}\bigl[|\psi^\pert_1(t)|^2\bigr]=o(\ep^2).
\end{equation*}
Observe that
\begin{align*}
	&f\bigl(t,s,u^\pert(s),X^\pert(s),Y^\pert(s),Z^\pert(t,s)\bigr)-f(t,s)\\
	&=\Bigl\{f\bigl(t,s,u^\pert(s),X^\pert(s),Y^\pert(s),Z^\pert(t,s)\bigr)\\
	&\hspace{1cm}-f\bigl(t,s,u^\pert(s),\hat{X}(s)+X^\pert_1(s)+X^\pert_2(s),Y^\pert(s),Z^\pert(t,s)\bigr)\Bigr\}\\
	&\hspace{0.5cm}+\Bigl\{f\bigl(t,s,u^\pert(s),\hat{X}(s)+X^\pert_1(s)+X^\pert_2(s),Y^\pert(s),Z^\pert(t,s)\bigr)\\
	&\hspace{1.5cm}-f\bigl(t,s,\hat{u}(s),\hat{X}(s)+X^\pert_1(s)+X^\pert_2(s),Y^\pert(s),\tilde{Z}^\pert(t,s)+\zeta^\pert(t,s)\bigr)\Bigr\}\\
	&\hspace{0.5cm}+\Bigl\{f\bigl(t,s,\hat{u}(s),\hat{X}(s)+X^\pert_1(s)+X^\pert_2(s),\tilde{Y}^\pert(s)+\eta^\pert(s),\tilde{Z}^\pert(t,s)+\zeta^\pert(t,s)\bigr)\\
	&\hspace{1.5cm}-f\bigl(t,s,\hat{u}(s),\hat{X}(s)+X^\pert_1(s)+X^\pert_2(s),\hat{Y}(s)+\eta^\pert(s),\hat{Z}(t,s)+\zeta^\pert(t,s)\bigr)\Bigr\}\\
	&\hspace{0.5cm}+\Bigl\{f\bigl(t,s,\hat{u}(s),\hat{X}(s)+X^\pert_1(s)+X^\pert_2(s),\hat{Y}(s)+\eta^\pert(s),\hat{Z}(t,s)+\zeta^\pert(t,s)\bigr)\\
	&\hspace{1.5cm}-f\bigl(t,s,\hat{u}(s),\hat{X}(s),\hat{Y}(s),\hat{Z}(t,s)\bigr)\Bigr\}\displaybreak[1]\\
	&=:\Lambda^\pert_1(t,s)+\Lambda^\pert_2(t,s)+\Lambda^\pert_3(t,s)+\Lambda^\pert_4(t,s).
\end{align*}
Now let us further observe the terms $\Lambda^\pert_i(\cdot,\cdot)$, $i=1,2,3,4$.
\par
\medskip
%% 1
$\Lambda^\pert_1(\cdot,\cdot)$: By Lemma~\ref{lemma: variation state}, it can be easily shown that
\begin{equation*}
	\sup_{t\in[\tau,T]}\mathbb{E}\Bigl[\Bigl|\int^T_t\Lambda^\pert_1(t,s)\,ds\Bigr|^2\Bigr]=o(\ep^2).
\end{equation*}
\par
\medskip
%% 2
$\Lambda^\pert_2(\cdot,\cdot)$: By the definitions of $u^\pert(\cdot)$ and $\tilde{Z}^\pert(\cdot,\cdot)$, we see that
\begin{align*}
	&\Lambda^\pert_2(t,s)=\Lambda^\pert_2(t,s)\1_{[\tau,\tau+\ep)}(s).
\end{align*}
Furthermore, for $(t,s)\in[\tau,\tau+\ep]^2$, $\Lambda^\pert_2(t,s)$ can be written as
\begin{align*}
	&\Lambda^\pert_2(t,s)\\
	&=f\bigl(t,s,v(s),\hat{X}(s),\hat{Y}(s),\hat{Z}(t,s)\mathalpha{+}\beta(t,s)\bigr)-f(t,s)\\
	&+\bigl\langle\tilde{\mathfrak{f}}^\pert_x(t,s),X^\pert_1(s)\mathalpha{+}X^\pert_2(s)\bigr\rangle+\tilde{\mathfrak{f}}^\pert_y(t,s)\bigl(\bar{Y}^\pert(s)\mathalpha{+}\eta^\pert(s)\bigr)+\tilde{\mathfrak{f}}^\pert_z(t,s)\bigl(\bar{Z}^\pert(t,s)\mathalpha{+}\zeta^\pert(t,s)\bigr),
\end{align*}
where, for $\alpha=x,y,z$,
\begin{align*}
	&\tilde{\mathfrak{f}}^\pert_\alpha(t,s):=\int^1_0\partial_\alpha f\bigl(t,s,v(s),\hat{X}(s)+\mu(X^\pert_1(s)+X^\pert_2(s)),\hat{Y}(s)+\mu(\bar{Y}^\pert(s)+\eta^\pert(s)),\\
	&\hspace{4cm}\hat{Z}(t,s)+\beta(t,s)+\mu(\bar{Z}^\pert(t,s)+\zeta^\pert(t,s))\bigr)\,d\mu\\
	&\hspace{2cm}-\int^1_0\partial_\alpha f\bigl(t,s,\hat{u}(s),\hat{X}(s)+\mu(X^\pert_1(s)+X^\pert_2(s)),\hat{Y}(s)+\mu(\bar{Y}^\pert(s)+\eta^\pert(s)),\\
	&\hspace{4.5cm}\hat{Z}(t,s)+\mu(\bar{Z}^\pert(t,s)+\zeta^\pert(t,s))\bigr)\,d\mu.
\end{align*}
We can show that (see Lemma~\ref{lemma: A2} in the appendix)
\begin{align*}
	&\sup_{t\in[\tau,T]}\mathbb{E}\Bigl[\Bigl|\int^T_t\Bigl(\bigl\langle\tilde{\mathfrak{f}}^\pert_x(t,s),X^\pert_1(s)\mathalpha{+}X^\pert_2(s)\bigr\rangle\mathalpha{+}\tilde{\mathfrak{f}}^\pert_y(t,s)\eta^\pert(s)\mathalpha{+}\tilde{\mathfrak{f}}^\pert_z(t,s)\zeta^\pert(t,s)\Bigr)\1_{[\tau,\tau+\ep)}(s)\,ds\Bigr|^2\Bigr]\\
	&=o(\ep^2).
\end{align*}
\par
\medskip
%% 3
$\Lambda^\pert_3(\cdot,\cdot)$: We have that
\begin{equation*}
	\Lambda^\pert_3(t,s)=\tilde{f}^\pert_y(t,s)\bar{Y}^\pert(s)+\tilde{f}^\pert_z(t,s)\bar{Z}^\pert(t,s)
\end{equation*}
where, for $\alpha=y,z$,
\begin{align*}
	&\tilde{f}^\pert_\alpha(t,s):=\int^1_0\partial_\alpha f\bigl(t,s,\hat{u}(s),\hat{X}(s)+X^\pert_1(s)+X^\pert_2(s),\\
	&\hspace{4cm}\hat{Y}(s)+\eta^\pert(s)+\mu\bar{Y}^\pert(s),\hat{Z}(t,s)+\zeta^\pert(t,s)+\mu\bar{Z}^\pert(t,s)\bigr)\,d\mu.
\end{align*}
\par
\medskip
%% 4
$\Lambda^\pert_4(\cdot,\cdot)$: Observe that
\begin{align*}
	\Lambda^\pert_4(t,s)&=\bigl\langle f_x(t,s),X^\pert_1(s)+X^\pert_2(s)\bigr\rangle+f_y(t,s)\eta^\pert(s)+f_z(t,s)\zeta^\pert(t,s)\\
	&\hspace{0.5cm}+\frac{1}{2}\left\langle D^2\tilde{f}^\pert(t,s)\!\left(\!\!
    \begin{array}{c}
      X^\pert_1(s)\mathalpha{+}X^\pert_2(s)\\
      \eta^\pert(s)\\
      \zeta^\pert(t,s)
    \end{array}
  \!\!\right),\left(\!\!
    \begin{array}{c}
      X^\pert_1(s)\mathalpha{+}X^\pert_2(s)\\
      \eta^\pert(s)\\
      \zeta^\pert(t,s)
    \end{array}
  \!\!\right)\right\rangle\displaybreak[1]\\
	&=\bigl\langle f_x(t,s)+f_y(t,s)p(s,s)+f_z(t,s)B_1(t,s),X^\pert_1(s)+X^\pert_2(s)\bigr\rangle\\
	&\hspace{0.5cm}+\frac{1}{2}\bigl\langle\bigl(f_y(t,s)P(s,s)+f_z(t,s)B_2(t,s)\bigr)X^\pert_1(s),X^\pert_1(s)\bigr\rangle\\
	&\hspace{0.5cm}+f_z(t,s)\langle B_3(t,s),X^\pert_1(s)\rangle\1_{[\tau,\tau+\ep)}(s)\\
	&\hspace{0.5cm}+\frac{1}{2}\left\langle D^2\tilde{f}^\pert(t,s)\!\left(\!\!
    \begin{array}{c}
      X^\pert_1(s)\mathalpha{+}X^\pert_2(s)\\
      \eta^\pert(s)\\
      \zeta^\pert(t,s)
    \end{array}
  \!\!\right),\left(\!\!
    \begin{array}{c}
      X^\pert_1(s)\mathalpha{+}X^\pert_2(s)\\
      \eta^\pert(s)\\
      \zeta^\pert(t,s)
    \end{array}
  \!\!\right)\right\rangle.
\end{align*}
Here we used the following notation:
\begin{align*}
	&D^2\tilde{f}^\pert(t,s):=2\int^1_0\int^1_0\lambda D^2f\bigl(t,s,\hat{u}(s),\hat{X}(s)+\lambda\mu(X^\pert_1(s)+X^\pert_2(s)),\\
	&\hspace{6cm}\hat{Y}(s)+\lambda\mu\eta^\pert(s),\hat{Z}(t,s)+\lambda\mu\zeta^\pert(t,s)\bigr)\,d\lambda\,d\mu.
\end{align*}
Furthermore, we can show that (see Lemma~\ref{lemma: A3} in the appendix)
\begin{equation*}
	\sup_{t\in[\tau,T]}\mathbb{E}\Bigl[\Bigl|\int^T_tf_z(t,s)\langle B_3(t,s),X^\pert_1(s)\rangle\1_{[\tau,\tau+\ep)}(s)\,ds\Bigr|^2\Bigr]=o(\ep^2)
\end{equation*}
and
\begin{align*}
	&\sup_{t\in[\tau,T]}\mathbb{E}\Bigl[\Bigl|\int^T_t\Bigl\{\left\langle D^2\tilde{f}^\pert(t,s)\!\left(\!\!
    \begin{array}{c}
      X^\pert_1(s)\mathalpha{+}X^\pert_2(s)\\
      \eta^\pert(s)\\
      \zeta^\pert(t,s)
    \end{array}
  \!\!\right),\left(\!\!
    \begin{array}{c}
      X^\pert_1(s)\mathalpha{+}X^\pert_2(s)\\
      \eta^\pert(s)\\
      \zeta^\pert(t,s)
    \end{array}
  \!\!\right)\right\rangle\\
	&\hspace{3cm}-\langle G(t,s)X^\pert_1(s),X^\pert_1(s)\rangle\Bigr\}\,ds\Bigr|^2\Bigr]\\
	&=o(\ep^2),
\end{align*}
where
\begin{equation*}	
	G(t,s):=[I_{n\times n},p(s,s),\sigma^\top_x(s)p(t,s)\mathalpha{+}q(t,s)]D^2f(t,s)[I_{n\times n},p(s,s),\sigma^\top_x(s)p(t,s)\mathalpha{+}q(t,s)]^\top.
\end{equation*}
\par
\medskip
By the above observations, we obtain
\begin{align*}
	&\bar{Y}^\pert(t)=\psi^\pert_2(t)-\int^T_t\bar{Z}^\pert(t,s)\,dW(s)\\
	&+\int^T_t\Bigl\{\Bigl(\tilde{f}^\pert_y(t,s)\mathalpha{+}\tilde{\mathfrak{f}}^\pert_y(t,s)\1_{[\tau,\tau+\ep)}(s)\Bigr)\bar{Y}^\pert(s)+\Bigl(\tilde{f}^\pert_z(t,s)\mathalpha{+}\tilde{\mathfrak{f}}^\pert_z(t,s)\1_{[\tau,\tau+\ep)}(s)\Bigr)\bar{Z}^\pert(t,s)\\
	&\hspace{1.3cm}+\Bigl\langle A_1(t,s)\mathalpha{+}f_x(t,s)\mathalpha{+}f_y(t,s)p(s,s)\mathalpha{+}f_z(t,s)B_1(t,s),X^\pert_1(s)\mathalpha{+}X^\pert_2(s)\Bigr\rangle\\
	&\hspace{1.3cm}+\frac{1}{2}\Bigl\langle\Bigl(A_2(t,s)\mathalpha{+}G(t,s)\mathalpha{+}f_y(t,s)P(s,s)\mathalpha{+}f_z(t,s)B_2(t,s)\Bigr)X^\pert_1(s),X^\pert_1(s)\Bigr\rangle\\
	&\hspace{1.3cm}+\Bigl(\alpha(t,s)+f\bigl(t,s,v(s),\hat{X}(s),\hat{Y}(s),\hat{Z}(t,s)\mathalpha{+}\beta(t,s)\bigr)-f(t,s)\Bigr)\1_{[\tau,\tau+\ep)}(s)\Bigr\}\,ds,\\
	&\hspace{10cm}t\in[\tau,T],
\end{align*}
for some $\psi^\pert_2(\cdot)$ satisfying
\begin{equation*}
	\sup_{t\in[\tau,T]}\mathbb{E}\bigl[|\psi^\pert_2(t)|^2\bigr]=o(\ep^2).
\end{equation*}
Recall that $A_1(\cdot,\cdot)$ and $A_2(\cdot,\cdot)$ include the undetermined processes $k(\cdot,\cdot)$ and $K(\cdot,\cdot)$, respectively, while $B_1(\cdot,\cdot)$ and $B_2(\cdot,\cdot)$ do not include either these processes. Therefore, if we set
\begin{equation}\label{k}
\begin{split}
	&k(t,s)=b^\top_x(s)p(t,s)+\sigma^\top_x(s)q(t,s)\\
	&\hspace{1.5cm}+f_x(t,s)+f_y(t,s)p(s,s)+f_z(t,s)(\sigma^\top_x(s)p(t,s)+q(t,s))
\end{split}
\end{equation}
and
\begin{align}\label{K}
\begin{split}
	&K(t,s)=b^\top_{xx}(s)p(t,s)+\sigma^\top_{xx}(s)q(t,s)+P(t,s)b_x(s)+b^\top_x(s)P(t,s)\\
	&\hspace{1.5cm}+Q(t,s)\sigma_x(s)+\sigma^\top_x(s)Q(t,s)+\sigma^\top_x(s)P(t,s)\sigma_x(s)\\
	&\hspace{1.5cm}+[I_{n\times n},p(s,s),\sigma^\top_x(s)p(t,s)\mathalpha{+}q(t,s)]D^2f(t,s)[I_{n\times n},p(s,s),\sigma^\top_x(s)p(t,s)\mathalpha{+}q(t,s)]^\top\\
	&\hspace{1.5cm}+f_y(t,s)P(s,s)+f_z(t,s)\bigl(\sigma^\top_{xx}(s)p(t,s)+P(t,s)\sigma_x(s)+\sigma^\top_x(s)P(t,s)+Q(t,s)\bigr),
\end{split}
\end{align}
then we obtain the EBSVIEs~\eqref{EBSVIE 1} and \eqref{EBSVIE 2}, and it holds that
\begin{align*}
	&\bar{Y}^\pert(t)=\psi^\pert_2(t)-\int^T_t\bar{Z}^\pert(t,s)\,dW(s)\\
	&+\int^T_t\Bigl\{\Bigl(\tilde{f}^\pert_y(t,s)\mathalpha{+}\tilde{\mathfrak{f}}^\pert_y(t,s)\1_{[\tau,\tau+\ep)}(s)\Bigr)\bar{Y}^\pert(s)+\Bigl(\tilde{f}^\pert_z(t,s)\mathalpha{+}\tilde{\mathfrak{f}}^\pert_z(t,s)\1_{[\tau,\tau+\ep)}(s)\Bigr)\bar{Z}^\pert(t,s)\\
	&\hspace{1.5cm}+\Bigl(\alpha(t,s)+f\bigl(t,s,v(s),\hat{X}(s),\hat{Y}(s),\hat{Z}(t,s)\mathalpha{+}\beta(t,s)\bigr)-f(t,s)\Bigr)\1_{[\tau,\tau+\ep)}(s)\Bigr\}\,ds,\\
	&\hspace{10cm}t\in[\tau,T].
\end{align*}
Note that, on $[\tau+\ep,T]$, $(\bar{Y}^\pert(\cdot),\bar{Z}^\pert(\cdot,\cdot))$ satisfies the following BSVIE:
\begin{align*}
	&\bar{Y}^\pert(t)=\psi^\pert_2(t)+\int^T_t\Bigl(\tilde{f}^\pert_y(t,s)\bar{Y}^\pert(s)+\tilde{f}^\pert_z(t,s)\bar{Z}^\pert(t,s)\Bigr)\,ds-\int^T_t\bar{Z}^\pert(t,s)\,dW(s),\\
	&\hspace{10cm}t\in[\tau+\ep,T].
\end{align*}
By the standard estimate~\eqref{BSVIE estimate} of the solution of the BSVIE, we have the following estimate:
\begin{equation*}
	\sup_{t\in[\tau+\ep,T]}\mathbb{E}\Bigl[|\bar{Y}^\pert(t)|^2+\int^T_t|\bar{Z}^\pert(t,s)|^2\,ds\Bigr]\leq C\sup_{t\in[\tau+\ep,T]}\mathbb{E}\bigl[|\psi^\pert_2(t)|^2\bigr]=o(\ep^2).
\end{equation*}
Similarly, from the estimate~\eqref{EBSVIE estimate}, we have that
\begin{equation*}
	\sup_{t\in[\tau,\tau+\ep]}\mathbb{E}\Bigl[\int^T_{\tau+\ep}|\bar{Z}^\pert(t,s)|^2\,ds\Bigr]=o(\ep^2).
\end{equation*}
Therefore, by defining
\begin{equation*}
	\psi^\pert_3(t):=\mathbb{E}_{\tau+\ep}\Bigl[\psi^\pert_2(t)+\int^T_{\tau+\ep}\Bigl(\tilde{f}^\pert_y(t,s)\bar{Y}^\pert(s)+\tilde{f}^\pert_z(t,s)\bar{Z}^\pert(t,s)\Bigr)\,ds\Bigr]
\end{equation*}
for $t\in[\tau,\tau+\ep]$, then we have
\begin{equation*}
	\sup_{t\in[\tau,\tau+\ep]}\mathbb{E}\bigl[|\psi^\pert_3(t)|^2\bigr]=o(\ep^2).
\end{equation*}
Moreover, $(\bar{Y}^\pert(\cdot),\bar{Z}^\pert(\cdot,\cdot))$ satisfies the following BSVIE on $[\tau,\tau+\ep]$:
\begin{equation}\label{bar Y BSVIE}
\begin{split}
	&\bar{Y}^\pert(t)=\psi^\pert_3(t)+\int^{\tau+\ep}_t\Bigl\{\bigl(\tilde{f}^\pert_y(t,s)\mathalpha{+}\tilde{\mathfrak{f}}^\pert_y(t,s)\bigr)\bar{Y}^\pert(s)+\bigl(\tilde{f}^\pert_z(t,s)\mathalpha{+}\tilde{\mathfrak{f}}^\pert_z(t,s)\bigr)\bar{Z}^\pert(t,s)\\
	&\hspace{4.5cm}+\alpha(t,s)+f\bigl(t,s,v(s),\hat{X}(s),\hat{Y}(s),\hat{Z}(t,s)\mathalpha{+}\beta(t,s)\bigr)-f(t,s)\Bigr\}\,ds\\
	&\hspace{2cm}-\int^{\tau+\ep}_t\bar{Z}^\pert(t,s)\,dW(s),\ t\in[\tau,\tau+\ep].
\end{split}
\end{equation}
Motivated by equation~\eqref{bar Y BSVIE}, we introduce the following (trivial) BSVIE for $(\check{Y}^\pert(\cdot),\check{Z}^\pert(\cdot,\cdot))$:
\begin{align*}
	&\check{Y}^\pert(t)=\int^{\tau+\ep}_t\Bigl\{\alpha(t,s)+f\bigl(t,s,v(s),\hat{X}(s),\hat{Y}(s),\hat{Z}(t,s)\mathalpha{+}\beta(t,s)\bigr)-f(t,s)\Bigr\}\,ds\\
	\nonumber&\hspace{2cm}-\int^{\tau+\ep}_t\check{Z}^\pert(t,s)\,dW(s),\ t\in[\tau,\tau+\ep].
\end{align*}
Note that, for $t\in[\tau,\tau+\ep]$, we have
\begin{equation}\label{the conditional expectation}
	\check{Y}^\pert(t)=\mathbb{E}_t\Bigl[\int^{\tau+\ep}_t\Bigl\{\alpha(t,s)+f\bigl(t,s,v(s),\hat{X}(s),\hat{Y}(s),\hat{Z}(t,s)\mathalpha{+}\beta(t,s)\bigr)-f(t,s)\Bigr\}\,ds\Bigr].
\end{equation}
By using the standard estimates of the solutions of BSVIEs, we can show the following estimate (see Lemma~\ref{lemma: A4} in the appendix):
\begin{equation}\label{cost expansion}
	\sup_{t\in[\tau,\tau+\ep]}\mathbb{E}\Bigl[|\bar{Y}^\pert(t)-\check{Y}^\pert(t)|^2+\int^{\tau+\ep}_t|\bar{Z}^\pert(t,s)-\check{Z}^\pert(t,s)|^2\,ds\Bigr]=o(\ep^2).
\end{equation}
Therefore, by the equalities~\eqref{perturbation cost},\,\eqref{the conditional expectation} and the estimate~\eqref{cost expansion}, we obtain the following equality:
\begin{equation}\label{cost representation 0}
\begin{split}
	&J(\tau,\hat{X}(\tau);u^\pert(\cdot))-J(\tau,\hat{X}(\tau);\hat{u}|_{[\tau,T]}(\cdot))\\
	&=\mathbb{E}_\tau\Bigl[\int^{\tau+\ep}_\tau\Bigl\{\langle p(\tau,s),\delta b(s)\rangle+\langle q(\tau,s),\delta\sigma(s)\rangle+\frac{1}{2}\langle P(\tau,s)\delta\sigma(s),\delta\sigma(s)\rangle\\
	&\hspace{2cm}+f\bigl(\tau,s,v(s),\hat{X}(s),\hat{Y}(s),\hat{Z}(\tau,s)\mathalpha{+}\langle p(\tau,s),\delta\sigma(s)\rangle\bigr)-f(\tau,s)\Bigr\}\,ds\Bigr]+\tilde{R}^\pert\ \text{a.s.}
\end{split}
\end{equation}
where $\tilde{R}^\pert$ is an $\mathcal{F}_\tau$-measurable random variable such that $\mathbb{E}\bigl[|\tilde{R}^\pert|^2\bigr]=o(\ep^2)$.
\par
\medskip
Unfortunately, we cannot use the Lebesgue differentiation theorem directly for the integrand in the right-hand side of \eqref{cost representation 0} since it depends on $\tau$. Note that the terms $p(\cdot,\cdot)$ and $P(\cdot,\cdot)$ can be treated easily since they have the reasonable continuity:
\begin{align*}
	&\mathbb{E}\Bigl[\Bigl|\int^{\tau+\ep}_\tau\langle p(\tau,s)-p(s,s),\delta b(s)\rangle\,ds\Bigr|^2\Bigr]=o(\ep^2)
\shortintertext{and}
	&\mathbb{E}\Bigl[\Bigl|\int^{\tau+\ep}_\tau\Bigl\langle\bigl(P(\tau,s)-P(s,s)\bigr)\delta\sigma(s),\delta\sigma(s)\Bigr\rangle\,ds\Bigr|^2\Bigr]=o(\ep^2).
\end{align*}
However, the cases of $q(\cdot,\cdot)$ and $\hat{Z}(\cdot,\cdot)$ are more delicate. We have to define the terms ``$q(s,s)$'' and ``$\hat{Z}(s,s)$'' in rigorous ways. To do so, we introduce the diagonal processes $\mathrm{Diag}[q](\cdot)$ and $\mathrm{Diag}[\hat{Z}](\cdot)$ of $q(\cdot,\cdot)$ and $\hat{Z}(\cdot,\cdot)$, respectively, which (uniquely) exist under our assumptions; see Lemma~\ref{lemma: Z diagonal}. The following lemma plays a key role in our study.

%% Lemma

\begin{lemm}
For any $v(\cdot)\in\mathcal{U}[t_0,T]$ and $\tau\in[t_0,T)$, It holds that
\begin{align*}
	&\mathbb{E}\Biggl[\Biggl|\mathbb{E}_\tau\Bigl[\int^{\tau+\ep}_\tau\Bigl\{\langle p(\tau,s),\delta b(s)\rangle+\langle q(\tau,s),\delta\sigma(s)\rangle+\frac{1}{2}\langle P(\tau,s)\delta\sigma(s),\delta\sigma(s)\rangle\\
	&\hspace{2.5cm}+f\bigl(\tau,s,v(s),\hat{X}(s),\hat{Y}(s),\hat{Z}(\tau,s)\mathalpha{+}\langle p(\tau,s),\delta\sigma(s)\rangle\bigr)-f(\tau,s)\Bigr\}\,ds\Bigr]\\
	&\hspace{1cm}-\mathbb{E}_\tau\Bigl[\int^{\tau+\ep}_\tau\bigl(\mathcal{H}(s,v(s);t_0,x_{t_0};\hat{u}(\cdot))-\mathcal{H}(s,\hat{u}(s);t_0,x_{t_0};\hat{u}(\cdot))\bigr)\,ds\Bigr]\Biggr|^2\Biggr]=o(\ep^2),
\end{align*}
where $\mathcal{H}(s,v;t_0,x_{t_0};\hat{u}(\cdot))$ is the $\mathcal{H}$-function with respect to $(t_0,x_{t_0})\in\mathcal{I}$ and $\hat{u}(\cdot)\in\mathcal{U}[t_0,T]$ defined by \eqref{H-function}.
\end{lemm}

%% Proof

\begin{proof}
We only prove that
\begin{equation}\label{q-diagonal limit}
	\mathbb{E}\Bigl[\Bigl|\int^{\tau+\ep}_\tau\bigl\langle q(\tau,s)-\mathrm{Diag}[q](s),\delta\sigma(s)\bigr\rangle\,ds\Bigr|^2\Bigr]=o(\ep^2).
\end{equation}
Then by using the regularity assumptions of $f$ (Assumption~\ref{assumption: cost BSVIE}), we can easily get the consequence. Concerning the estimate \eqref{q-diagonal limit}, observe that
\begin{align}\label{q estimate}
	\nonumber&\mathbb{E}\Bigl[\Bigl|\int^{\tau+\ep}_\tau\bigl\langle q(\tau,s)-\mathrm{Diag}[q](s),\delta\sigma(s)\bigr\rangle\,ds\Bigr|^2\Bigr]\\
	\nonumber&\leq\mathbb{E}\Bigl[\int^{\tau+\ep}_\tau\bigl|q(\tau,s)-\mathrm{Diag}[q](s)\bigr|^2\,ds\,\int^{\tau+\ep}_\tau\bigl|\delta\sigma(s)\bigr|^2\,ds\Bigr]\displaybreak[1]\\
	&\leq\mathbb{E}\Bigl[\Bigl(\int^{\tau+\ep}_\tau\bigl|q(\tau,s)-\mathrm{Diag}[q](s)\bigr|^2\,ds\Bigr)^2\Bigr]^{1/2}\underbrace{\mathbb{E}\Bigl[\Bigl(\int^{\tau+\ep}_\tau\bigl|\delta\sigma(s)\bigr|^2\,ds\Bigr)^2\Bigr]^{1/2}}_{=O(\ep)}.
\end{align}
By letting $p'=4$ and $q'=2$ in \eqref{q-diagonal estimate}, we obtain
\begin{equation*}
	\mathbb{E}\Bigl[\Bigl(\int^{\tau+\ep}_\tau\bigl|q(\tau,s)-\mathrm{Diag}[q](s)\bigr|^2\,ds\Bigr)^2\Bigr]^{1/4}=o(\ep).
\end{equation*}
Thus, we see that the last term in \eqref{q estimate} is of order $o(\ep^3)$. In particular, the estimate \eqref{q-diagonal limit} holds true.
\end{proof}

Consequently, for any $\hat{u}(\cdot),v(\cdot)\in\mathcal{U}[t_0,T]$, $\tau\in[t_0,T)$ and $\ep\in(0,T-\tau)$, we have that
\begin{equation}\label{cost representation}
\begin{split}
	&J(\tau,\hat{X}(\tau);u^\pert(\cdot))-J(\tau,\hat{X}(\tau);\hat{u}|_{[\tau,T]}(\cdot))\\
	&=\mathbb{E}_\tau\Bigl[\int^{\tau+\ep}_\tau\bigl(\mathcal{H}(s,v(s);t_0,x_{t_0};\hat{u}(\cdot))-\mathcal{H}(s,\hat{u}(s);t_0,x_{t_0};\hat{u}(\cdot))\bigr)\,ds\Bigr]+R^\pert\ \text{a.s.}
\end{split}
\end{equation}
where $R^\pert$ is an $\mathcal{F}_\tau$-measurable random variable such that $\mathbb{E}\bigl[|R^\pert|^2\bigr]=o(\ep^2)$. The classical Lebesgue differentiation theorem can be applied to the integrand of the right-hand side of the above equality, since it does not depend on $\tau$; see Lemma~\ref{lemma: A5} in the appendix.
\par
\medskip
Now we are ready to prove our main result.

%% Proof of the main theorem

\begin{proof}[Proof of Theorem~\ref{theorem: main result}]
\emph{Sufficiency}: Suppose that $\hat{u}(\cdot)\in\mathcal{U}[t_0,T]$ satisfies \eqref{characterization}. Then for any $v(\cdot)\in\mathcal{U}[t_0,T]$, $\tau\in[t_0,T)$, and $\ep\in(0,T-\tau)$, equality~\eqref{cost representation} yields that
\begin{equation*}
	J(\tau,\hat{X}(\tau);u^\pert(\cdot))-J(\tau,\hat{X}(\tau);\hat{u}|_{[\tau,T]}(\cdot))\geq R^\pert\ \text{a.s.}
\end{equation*}
where $R^\pert$ is an $\mathcal{F}_\tau$-measurable random variable such that $\mathbb{E}\bigl[|R^\pert|^2\bigr]=o(\ep^2)$. Therefore, for any nonnegative, bounded and $\mathcal{F}_\tau$-measurable random variable $\xi_\tau$, it holds that
\begin{equation*}
	\mathbb{E}\Biggl[\frac{J(\tau,\hat{X}(\tau);u^\pert(\cdot))-J(\tau,\hat{X}(\tau);\hat{u}|_{[\tau,T]}(\cdot))}{\ep}\,\xi_\tau\Biggr]\geq\frac{1}{\ep}\mathbb{E}\bigl[R^\pert\xi_\tau\bigr]\overset{\ep\downarrow0}{\longrightarrow}0,
\end{equation*}
and hence $\hat{u}(\cdot)$ is an open-loop equilibrium control with respect to $(t_0,x_{t_0})\in\mathcal{I}$.
\par
\emph{Necessity}: Suppose that $\hat{u}(\cdot)\in\mathcal{U}[t_0,T]$ is an open-loop equilibrium control with respect to $(t_0,x_{t_0})\in\mathcal{I}$. Fix an element $v\in U$ and define $v(\cdot)\in\mathcal{U}[t_0,T]$ by $v(\cdot)\equiv v$. By the definition of open-loop equilibrium controls and equality \eqref{cost representation}, we have, for any $\tau\in[t_0,T)$ and any nonnegative, bounded and $\mathcal{F}_\tau$-measurable random variable $\xi_\tau$,
\begin{equation*}
	\liminf_{\ep\downarrow0}\frac{1}{\ep}\mathbb{E}\Bigl[\int^{\tau+\ep}_\tau\bigl(\mathcal{H}(s,v;t_0,x_{t_0};\hat{u}(\cdot))-\mathcal{H}(s,\hat{u}(s);t_0,x_{t_0};\hat{u}(\cdot))\bigr)\,ds\,\xi_\tau\Bigr]\geq0.
\end{equation*}
This implies that (see Lemma~\ref{lemma: A5} in the appendix)
\begin{equation*}
	\mathcal{H}(s,v;t_0,x_{t_0};\hat{u}(\cdot))-\mathcal{H}(s,\hat{u}(s);t_0,x_{t_0};\hat{u}(\cdot))\geq0
\end{equation*}
for $\mathrm{Leb}_{[t_0,T]}\otimes\mathbb{P}$-a.e.\,$(s,\omega)\in[t_0,T]\times\Omega$. Since the control space $(U,d)$ is a separable metric space and the $\mathcal{H}$-function is continuous in $v\in U$, we obtain \eqref{characterization}.
\end{proof}

%% Remark

\begin{rem}\label{remark: last remark}
By Theorem~\ref{theorem: main result} and equality~\eqref{cost representation}, we see that if $\hat{u}(\cdot)\in\mathcal{U}[t_0,T]$ is an open-loop equilibrium control with respect to $(t_0,x_{t_0})\in\mathcal{I}$, then for any $v(\cdot)\in\mathcal{U}[t_0,T]$ and $\tau\in[t_0,T)$, there exists a sequence $\{\ep_k\}_{k\in\mathbb{N}}\subset(0,T-\tau)$ such that $\lim_{k\to\infty}\ep_k=0$ and
\begin{equation*}
	\liminf_{k\to\infty}\frac{J(\tau,\hat{X}(\tau);u^{\tau,\ep_k}(\cdot))-J(\tau,\hat{X}(\tau);\hat{u}|_{[\tau,T]}(\cdot))}{\ep_k}\geq0\ \text{a.s.}
\end{equation*}
We emphasize that this is not trivial from the definition, but a consequence of our analysis. Also, the above is comparable to the original ``definition'' \eqref{definition'} of open-loop equilibrium controls introduced in \cite{a_Hu-Jin-Zhou_12,a_Hu-Jin-Zhou_17}.
\end{rem}

%%%%%%%%%%%%%%%%%%%%%%%%%%%%%%%%%%%%%%%%%%%%%%%%%%%%%%%%%%%%%%%%%%%%%%%%%%%%%%%%%%%%%%%%%%%%%%
%%%% Acknowledgments
%%%%%%%%%%%%%%%%%%%%%%%%%%%%%%%%%%%%%%%%%%%%%%%%%%%%%%%%%%%%%%%%%%%%%%%%%%%%%%%%%%%%%%%%%%%%%%

\section*{Acknowledgments}

The author would like to thank the editor and the referees for their constructive comments and suggestions. This work was supported by JSPS KAKENHI Grant Number JP18J20973.

%%%%%%%%%%%%%%%%%%%%%%%%%%%%%%%%%%%%%%%%%%%%%%%%%%%%%%%%%%%%%%%%%%%%%%%%%%%%%%%%%%%%%%%%%%%%%%
%%%% References
%%%%%%%%%%%%%%%%%%%%%%%%%%%%%%%%%%%%%%%%%%%%%%%%%%%%%%%%%%%%%%%%%%%%%%%%%%%%%%%%%%%%%%%%%%%%%%

\bibliography{reference}

%%%%%%%%%%%%%%%%%%%%%%%%%%%%%%%%%%%%%%%%%%%%%%%%%%%%%%%%%%%%%%%%%%%%%%%%%%%%%%%%%%%%%%%%%%%%%%
%%%% Appendix
%%%%%%%%%%%%%%%%%%%%%%%%%%%%%%%%%%%%%%%%%%%%%%%%%%%%%%%%%%%%%%%%%%%%%%%%%%%%%%%%%%%%%%%%%%%%%%

\appendix
\section{Appendix}\label{appendix}

In this appendix, we prove some technical estimates appearing in Section~\ref{section: proof of main result}. In Lemmas~\ref{lemma: A1},\,\ref{lemma: A2},\,\ref{lemma: A3}, and \ref{lemma: A4}, we use the same notation as in Section~\ref{section: proof of main result}. Lemma~\ref{lemma: A5} is an abstract result which we used in the proof of the necessity part of Theorem~\ref{theorem: main result}.

%% Lemma: A1

\begin{lemm}\label{lemma: A1}
It holds that
\begin{equation*}
	\sup_{t\in[\tau,T]}\mathbb{E}\Bigl[\Bigl|\int^T_t\langle A_3(t,s),X^\pert_1(s)\rangle\1_{[\tau,\tau+\ep)}(s)\,ds\Bigr|^2\Bigr]=o(\ep^2). 
\end{equation*}
\end{lemm}

%% Proof

\begin{proof}
Observe that
\begin{align*}
	&\sup_{t\in[\tau,T]}\mathbb{E}\Bigl[\Bigl|\int^T_t\langle A_3(t,s),X^\pert_1(s)\rangle\1_{[\tau,\tau+\ep)}(s)\,ds\Bigr|^2\Bigr]=\sup_{t\in[\tau,\tau+\ep]}\mathbb{E}\Bigl[\Bigl|\int^{\tau+\ep}_t\langle A_3(t,s),X^\pert_1(s)\rangle\,ds\Bigr|^2\Bigr]\displaybreak[1]\\
	&\leq\sup_{t\in[\tau,\tau+\ep]}\mathbb{E}\Bigl[\sup_{s\in[\tau,T]}|X^\pert_1(s)|^2\Bigl(\int^{\tau+\ep}_t|A_3(t,s)|\,ds\Bigr)^2\Bigr]\displaybreak[1]\\
	&\leq\underbrace{\mathbb{E}\Bigl[\sup_{s\in[\tau,T]}|X^\pert_1(s)|^4\Bigr]^{1/2}}_{=O(\ep)}\sup_{t\in[\tau,\tau+\ep]}\mathbb{E}\Bigl[\Bigl(\int^{\tau+\ep}_t|A_3(t,s)|\,ds\Bigr)^4\Bigr]^{1/2}.
\end{align*}
Concerning the term $\sup_{t\in[\tau,\tau+\ep]}\mathbb{E}\Bigl[\Bigl(\int^{\tau+\ep}_t|A_3(t,s)|\,ds\Bigr)^4\Bigr]$, we have that, for example,
\begin{align*}
	&\sup_{t\in[\tau,\tau+\ep]}\mathbb{E}\Bigl[\Bigl(\int^{\tau+\ep}_t|Q(t,s)||\delta\sigma(s)|\,ds\Bigr)^4\Bigr]\\
	&\leq\sup_{t\in[\tau,\tau+\ep]}\mathbb{E}\Bigl[\Bigl(\int^{\tau+\ep}_t|Q(t,s)|^2\,ds\Bigr)^2\Bigr(\int^{\tau+\ep}_t|\delta\sigma(s)|^2\,ds\Bigr)^2\Bigr]\displaybreak[1]\\
	&\leq\underbrace{\mathbb{E}\Bigr[\Bigr(\int^{\tau+\ep}_\tau|\delta\sigma(s)|^2\,ds\Bigr)^4\Bigr]^{1/2}}_{=O(\ep^2)}\sup_{t\in[\tau,\tau+\ep]}\mathbb{E}\Bigl[\Bigl(\int^{\tau+\ep}_t|Q(t,s)|^2\,ds\Bigr)^4\Bigr]^{1/2},
\end{align*}
and
\begin{align*}
	&\sup_{t\in[\tau,\tau+\ep]}\mathbb{E}\Bigl[\Bigl(\int^{\tau+\ep}_t|Q(t,s)|^2\,ds\Bigr)^4\Bigr]\\
	&\leq128\Biggl(\sup_{t\in[\tau,\tau+\ep]}\mathbb{E}\Bigl[\Bigl(\int^T_{t_0}|Q(t,s)-Q(\tau,s)|^2\,ds\Bigr)^4\Bigr]+\mathbb{E}\Bigl[\Bigl(\int^{\tau+\ep}_\tau|Q(\tau,s)|^2\,ds\Bigr)^4\Bigr]\Biggr)\\
	&\to0\ \text{as}\ \ep\downarrow0.
\end{align*}
Thus, we see that $\sup_{t\in[\tau,\tau+\ep]}\mathbb{E}\Bigl[\Bigl(\int^{\tau+\ep}_t|A_3(t,s)|\,ds\Bigr)^4\Bigr]=o(\ep^2)$, and finish the proof.
\end{proof}

%% Lemma: A2

\begin{lemm}\label{lemma: A2}
It holds that
\begin{align*}
	&\sup_{t\in[\tau,T]}\mathbb{E}\Bigl[\Bigl|\int^T_t\Bigl(\langle\tilde{\mathfrak{f}}^\pert_x(t,s),X^\pert_1(s)\mathalpha{+}X^\pert_2(s)\rangle\mathalpha{+}\tilde{\mathfrak{f}}^\pert_y(t,s)\eta^\pert(s)\mathalpha{+}\tilde{\mathfrak{f}}^\pert_z(t,s)\zeta^\pert(t,s)\Bigr)\1_{[\tau,\tau+\ep)}(s)\,ds\Bigr|^2\Bigr]\\
	&=o(\ep^2).
\end{align*}
\end{lemm}

%% Proof

\begin{proof}
Note that $\tilde{\mathfrak{f}}^\pert_x(\cdot,\cdot)$, $\tilde{\mathfrak{f}}^\pert_y(\cdot,\cdot)$ and $\tilde{\mathfrak{f}}^\pert_z(\cdot,\cdot)$ are uniformly bounded. Recall the definitions~\eqref{eta zeta} of $\eta^\pert(\cdot)$ and $\zeta^\pert(\cdot,\cdot)$. By the same arguments as in the proof of Lemma~\ref{lemma: A1}, we can show the assertion.
\end{proof}

%% Lemma: A3

\begin{lemm}\label{lemma: A3}
It holds that
\begin{equation}\label{estimate B_3}
	\sup_{t\in[\tau,T]}\mathbb{E}\Bigl[\Bigl|\int^T_tf_z(t,s)\langle B_3(t,s),X^\pert_1(s)\rangle\1_{[\tau,\tau+\ep)}(s)\,ds\Bigr|^2\Bigr]=o(\ep^2)
\end{equation}
and
\begin{equation}\label{estimate D^2}
\begin{split}
	&\sup_{t\in[\tau,T]}\mathbb{E}\Bigl[\Bigl|\int^T_t\Bigl\{\left\langle D^2\tilde{f}^\pert(t,s)\!\left(\!\!
    \begin{array}{c}
      X^\pert_1(s)\mathalpha{+}X^\pert_2(s)\\
      \eta^\pert(s)\\
      \zeta^\pert(t,s)
    \end{array}
  \!\!\right),\left(\!\!
    \begin{array}{c}
      X^\pert_1(s)\mathalpha{+}X^\pert_2(s)\\
      \eta^\pert(s)\\
      \zeta^\pert(t,s)
    \end{array}
  \!\!\right)\right\rangle\\
	&\hspace{3cm}-\langle G(t,s)X^\pert_1(s),X^\pert_1(s)\rangle\Bigr\}\,ds\Bigr|^2\Bigr]\\
	&=o(\ep^2),
\end{split}
\end{equation}
where
\begin{equation*}	
	G(t,s):=[I_{n\times n},p(s,s),\sigma^\top_x(s)p(t,s)\mathalpha{+}q(t,s)]D^2f(t,s)[I_{n\times n},p(s,s),\sigma^\top_x(s)p(t,s)\mathalpha{+}q(t,s)]^\top.
\end{equation*}
\end{lemm}

%% Proof

\begin{proof}
The estimate~\eqref{estimate B_3} can be proved by the same arguments as in the proof of Lemma~\ref{lemma: A1}. We prove the estimate~\eqref{estimate D^2}. By the definitions~\eqref{eta zeta} of $\eta^\pert(\cdot)$ and $\zeta^\pert(\cdot,\cdot)$, we have
\begin{equation*}
	\left(\!\!
    \begin{array}{c}
      X^\pert_1(s)\mathalpha{+}X^\pert_2(s)\\
      \eta^\pert(s)\\
      \zeta^\pert(t,s)
    \end{array}
  \!\!\right)
	=\mathcal{X}^\pert_1(t,s)+\mathcal{X}^\pert_2(t,s),
\end{equation*}
where
\begin{equation*}
	\mathcal{X}^\pert_1(t,s):=
	\left(\!\!
    \begin{array}{c}
      X^\pert_1(s)\\
      \langle p(s,s),X^\pert_1(s)\rangle\\
      \langle B_1(t,s),X^\pert_1(s)\rangle
    \end{array}
  \!\!\right)
\end{equation*}
and
\begin{align*}
	&\mathcal{X}^\pert_2(t,s)\\
	&:=
	\left(\!\!
    \begin{array}{c}
      X^\pert_2(s)\\
      \langle p(s,s),X^\pert_2(s)\rangle+\frac{1}{2}\langle P(s,s)X^\pert_1(s),X^\pert_1(s)\rangle\\
      \langle B_1(t,s),X^\pert_2(s)\rangle+\frac{1}{2}\langle B_2(t,s)X^\pert_1(s),X^\pert_1(s)\rangle+\langle B_3(t,s),X^\pert_1(s)\rangle\1_{[\tau,\tau+\ep)}(s)
    \end{array}
  \!\!\right).
\end{align*}
Thus, we obtain
\begin{align*}
	&\left\langle D^2\tilde{f}^\pert(t,s)\!\left(\!\!
    \begin{array}{c}
      X^\pert_1(s)\mathalpha{+}X^\pert_2(s)\\
      \eta^\pert(s)\\
      \zeta^\pert(t,s)
    \end{array}
  \!\!\right),\left(\!\!
    \begin{array}{c}
      X^\pert_1(s)\mathalpha{+}X^\pert_2(s)\\
      \eta^\pert(s)\\
      \zeta^\pert(t,s)
    \end{array}
  \!\!\right)\right\rangle\\
	&=\bigl\langle D^2\tilde{f}^\pert(t,s)\mathcal{X}^\pert_1(t,s),\mathcal{X}^\pert_1(t,s)\bigr\rangle+2\bigl\langle D^2\tilde{f}^\pert(t,s)\mathcal{X}^\pert_1(t,s),\mathcal{X}^\pert_2(t,s)\bigr\rangle\\
	&\hspace{1cm}+\bigl\langle D^2\tilde{f}^\pert(t,s)\mathcal{X}^\pert_2(t,s),\mathcal{X}^\pert_2(t,s)\bigr\rangle.
\end{align*}
By Lemma~\ref{lemma: variation state}, we see that
\begin{equation}\label{X_1 calculation}
\begin{split}
	&\sup_{t\in[\tau,T]}\mathbb{E}\Bigl[\Bigl(\int^T_t\bigl|\mathcal{X}^\pert_1(t,s)\bigr|^2\,ds\Bigr)^2\Bigr]\\
	&\leq\sup_{t\in[\tau,T]}\mathbb{E}\Bigl[\Bigl(\int^T_t\bigl(1+|p(s,s)|^2+|B_1(t,s)|^2\bigr)\bigl|X^\pert_1(s)\bigr|^2\,ds\Bigr)^2\Bigr]\\
	&\leq\sup_{t\in[\tau,T]}\mathbb{E}\Bigl[\sup_{s\in[\tau,T]}\bigl|X^\pert_1(s)\bigr|^4\Bigl(\int^T_t\bigl(1+|p(s,s)|^2+|B_1(t,s)|^2\bigr)\,ds\Bigr)^2\Bigr]\\
	&\leq\underbrace{\mathbb{E}\Bigl[\sup_{s\in[\tau,T]}\bigl|X^\pert_1(s)\bigr|^8\Bigr]^{1/2}}_{=O(\ep^2)}\underbrace{\sup_{t\in[\tau,T]}\mathbb{E}\Bigl[\Bigl(\int^T_t\bigl(1+|p(s,s)|^2+|B_1(t,s)|^2\bigr)\,ds\Bigr)^4\Bigr]^{1/2}}_{<\infty},
\end{split}
\end{equation}
and hence
\begin{equation}\label{estimate X_1}
	\sup_{t\in[\tau,T]}\mathbb{E}\Bigl[\Bigl(\int^T_t\bigl|\mathcal{X}^\pert_1(t,s)\bigr|^2\,ds\Bigr)^2\Bigr]=O(\ep^2).
\end{equation}
Similarly we can show that
\begin{equation}\label{estimate X_2}
	\sup_{t\in[\tau,T]}\mathbb{E}\Bigl[\Bigl(\int^T_t\bigl|\mathcal{X}^\pert_2(t,s)\bigr|^2\,ds\Bigr)^2\Bigr]=o(\ep^2).
\end{equation}
Note that $D^2\tilde{f}^\pert(\cdot,\cdot)$ is uniformly bounded. Therefore, by the estimates~\eqref{estimate X_1} and \eqref{estimate X_2}, we obtain
\begin{align*}
	&\sup_{t\in[\tau,T]}\mathbb{E}\Bigl[\Bigl|\int^T_t\bigl\langle D^2\tilde{f}^\pert(t,s)\mathcal{X}^\pert_1(t,s),\mathcal{X}^\pert_2(t,s)\bigr\rangle\,ds\Bigr|^2\Bigr]=o(\ep^2)
\shortintertext{and}
	&\sup_{t\in[\tau,T]}\mathbb{E}\Bigl[\Bigl|\int^T_t\bigl\langle D^2\tilde{f}^\pert(t,s)\mathcal{X}^\pert_2(t,s),\mathcal{X}^\pert_2(t,s)\bigr\rangle\,ds\Bigr|^2\Bigr]=o(\ep^2).
\end{align*}
On the other hand, a simple calculation shows that
\begin{equation*}
	\bigl\langle G(t,s)X^\pert_1(s),X^\pert_1(s)\bigr\rangle=\bigl\langle D^2f(t,s)\mathcal{X}^\pert_1(t,s),\mathcal{X}^\pert_1(t,s)\bigr\rangle.
\end{equation*}
Thus, it remains to show that
\begin{equation}\label{estimate D^2 X_1}
	\sup_{t\in[\tau,T]}\mathbb{E}\Bigl[\Bigl|\int^T_t\bigl\langle \bigl(D^2\tilde{f}^\pert(t,s)-D^2f(t,s)\bigr)\mathcal{X}^\pert_1(t,s),\mathcal{X}^\pert_1(t,s)\bigr\rangle\,ds\Bigr|^2\Bigr]=o(\ep^2).
\end{equation}
Now we prove \eqref{estimate D^2 X_1}. First of all, by the same calculations as in \eqref{X_1 calculation}, we have, for any $\mathcal{A}\in\mathcal{B}([\tau,T])\otimes\mathcal{F}_T$,
\begin{align*}
	&\sup_{t\in[\tau,T]}\mathbb{E}\Bigl[\Bigl(\int^T_t\bigl|\mathcal{X}^\pert_1(t,s)\bigr|^2\1_\mathcal{A}(s,\omega)\,ds\Bigr)^2\Bigr]\\
	&\leq\mathbb{E}\Bigl[\sup_{s\in[\tau,T]}\bigl|X^\pert_1(s)\bigr|^{8}\Bigr]^{1/2}\sup_{t\in[\tau,T]}\mathbb{E}\Bigl[\Bigl(\int^T_t\bigl(1+|p(s,s)|^2+|B_1(t,s)|^2\bigr)\1_\mathcal{A}(s,\omega)\,ds\Bigr)^{4}\Bigr]^{1/2}.
\end{align*}
For each $\kappa>0$, define $\mathfrak{A}(\kappa):=\{\mathcal{A}\in\mathcal{B}([\tau,T])\otimes\mathcal{F}_T\,|\,\mathrm{Leb}_{[\tau,T]}\otimes\mathbb{P}(\mathcal{A})\leq\kappa\}$. Fix an arbitrary $\gamma>0$. Then, for each $t\in[\tau,T]$, there exists a constant $\kappa_t=\kappa_t(\gamma)>0$ such that
\begin{equation*}
	\mathbb{E}\Bigl[\Bigl(\int^T_\tau\bigl(1+|p(s,s)|^2+|B_1(t,s)|^2\bigr)\1_\mathcal{A}(s,\omega)\,ds\Bigr)^{4}\Bigr]\leq\frac{\gamma^2}{256},\ \forall\,\mathcal{A}\in\mathfrak{A}(\kappa_t).
\end{equation*}
Besides, since the map $[\tau,T]\ni t\mapsto B_1(t,\cdot)\in L^{8,2}_\mathbb{F}(\tau,T;\mathbb{R}^n)$ is (uniformly) continuous, there exists a partition $\tau=t_0<t_1<\dots<t_N=T$ of $[\tau,T]$ such that
\begin{equation*}
	\mathbb{E}\Bigl[\Bigl(\int^T_\tau|B_1(t,s)-B_1(t_n,s)|^2\,ds\Bigr)^{4}\Bigr]\leq\frac{\gamma^2}{256},\ \forall\,t\in[t_{n-1},t_n],\ n=1,\dots,N.
\end{equation*}
Define $\kappa=\kappa(\gamma):=\min\{\kappa_{t_1},\dots,\kappa_{t_N}\}$. Then, for any $\mathcal{A}\in\mathfrak{A}(\kappa)$ and any $t\in[t_{n-1},t_n]$ with $n=1,\dots,N$, it holds that
\begin{align*}
	&\mathbb{E}\Bigl[\Bigl(\int^T_t\bigl(1+|p(s,s)|^2+|B_1(t,s)|^2\bigr)\1_\mathcal{A}(s,\omega)\,ds\Bigr)^{4}\Bigr]\\
	&\leq128\Bigl\{\mathbb{E}\Bigl[\Bigl(\int^T_\tau\bigl(1+|p(s,s)|^2+|B_1(t_n,s)|^2\bigr)\1_\mathcal{A}(s,\omega)\,ds\Bigr)^{4}\Bigr]\\
	&\hspace{4cm}+\mathbb{E}\Bigl[\Bigl(\int^T_\tau|B_1(t,s)-B_1(t_n,s)|^2\,ds\Bigr)^{4}\Bigr]\Bigr\}\\
	&\leq\gamma^2.
\end{align*}
Thus, we obtain
\begin{equation*}
	\sup_{t\in[\tau,T]}\mathbb{E}\Bigl[\Bigl(\int^T_t\bigl|\mathcal{X}^\pert_1(t,s)\bigr|^2\1_\mathcal{A}(s,\omega)\,ds\Bigr)^2\Bigr]\leq\mathbb{E}\Bigl[\sup_{s\in[\tau,T]}\bigl|X^\pert_1(s)\bigr|^{8}\Bigr]^{1/2}\gamma,\ \forall\,\mathcal{A}\in\mathfrak{A}(\kappa).
\end{equation*}
On the other hand, by Assumption~\ref{assumption: cost BSVIE}~(\rnum{2}), there exists a modulus of continuity $\rho:[0,\infty)\to[0,\infty)$ such that
\begin{equation*}
	|D^2\tilde{f}^\pert(t,s)-D^2f(t,s)|\leq\rho\bigl(|X^\pert_1(s)+X^\pert_2(s)|+|\eta^\pert(s)|+|\zeta^\pert(t,s)|\bigr).
\end{equation*}
Furthermore, by using Lemma~\ref{lemma: variation state}, we can easily show that
\begin{equation*}
	\lim_{\ep\downarrow0}\sup_{t\in[\tau,T]}\mathbb{E}\Bigl[\int^T_\tau\bigl(|X^\pert_1(s)+X^\pert_2(s)|+|\eta^\pert(s)|+|\zeta^\pert(t,s)|\bigr)\,ds\Bigr]=0.
\end{equation*}
Hence, there exists a constant $\ep_0=\ep_0(\gamma)>0$ such that, for any $0<\ep<\ep_0$, we have
\begin{equation*}
	\{(s,\omega)\in[\tau,T]\times\Omega\,|\,|D^2\tilde{f}^\pert(t,s)-D^2f(t,s)|\geq\sqrt{\gamma}\}\in\mathfrak{A}(\kappa),\ \forall\,t\in[\tau,T].
\end{equation*}
Therefore, for any $0<\ep<\ep_0$, it holds that
\begin{align*}
	&\sup_{t\in[\tau,T]}\mathbb{E}\Bigl[\Bigl|\int^T_t\bigl\langle \bigl(D^2\tilde{f}^\pert(t,s)-D^2f(t,s)\bigr)\mathcal{X}^\pert_1(t,s),\mathcal{X}^\pert_1(t,s)\bigr\rangle\,ds\Bigr|^2\Bigr]\\
	&\leq 2\sup_{t\in[\tau,T]}\mathbb{E}\Bigl[\Bigl(\int^T_t\bigl|D^2\tilde{f}^\pert(t,s)-D^2f(t,s)\bigr|\bigl|\mathcal{X}^\pert_1(t,s)\bigr|^2\1_{\{|D^2\tilde{f}^\pert(t,s)-D^2f(t,s)|<\sqrt{\gamma}\}}\,ds\Bigr)^2\Bigr]\\
	&\hspace{1cm}+2\sup_{t\in[\tau,T]}\mathbb{E}\Bigl[\Bigl(\int^T_t\bigl|D^2\tilde{f}^\pert(t,s)-D^2f(t,s)\bigr|\bigl|\mathcal{X}^\pert_1(t,s)\bigr|^2\1_{\{|D^2\tilde{f}^\pert(t,s)-D^2f(t,s)|\geq\sqrt{\gamma}\}}\,ds\Bigr)^2\Bigr]\displaybreak[1]\\
	&\leq 2\underbrace{\sup_{t\in[\tau,T]}\mathbb{E}\Bigl[\Bigl(\int^T_t\bigl|\mathcal{X}^\pert_1(t,s)\bigr|^2\,ds\Bigr)^2\Bigr]}_{=O(\ep^2)}\gamma+8\|D^2f\|^2_\infty\underbrace{\mathbb{E}\Bigl[\sup_{s\in[\tau,T]}\bigl|X^\pert_1(s)\bigr|^{8}\Bigr]^{1/2}}_{=O(\ep^2)}\gamma\\
	&\leq C\ep^2\gamma,
\end{align*}
where $C>0$ is a constant which is independent of $\ep>0$ and $\gamma>0$. Since $\gamma>0$ is arbitrary, this implies \eqref{estimate D^2 X_1}. Hence, we obtain \eqref{estimate D^2}.
\end{proof}

%% Lemma: A4

\begin{lemm}\label{lemma: A4}
It holds that
\begin{equation}\label{lemestimate 1}
	\sup_{t\in[\tau,\tau+\ep]}\mathbb{E}\Bigl[|\bar{Y}^\pert(t)-\check{Y}^\pert(t)|^2+\int^{\tau+\ep}_t|\bar{Z}^\pert(t,s)-\check{Z}^\pert(t,s)|^2\,ds\Bigr]=o(\ep^2).
\end{equation}
\end{lemm}

%% Proof

\begin{proof}
Firstly we prove the following estimate:
\begin{equation}\label{lemestimate 0}
	\sup_{t\in[\tau,\tau+\ep]}\mathbb{E}\Bigl[|\bar{Y}^\pert(t)|^2+\int^{\tau+\ep}_t|\bar{Z}^\pert(t,s)|^2\,ds\Bigr]=o(\ep).
\end{equation}
Recall that $(\bar{Y}^\pert(\cdot),\bar{Z}^\pert(\cdot,\cdot))$ satisfies BSVIE~\eqref{bar Y BSVIE}. Thus, by the standard estimate~\eqref{BSVIE estimate} of the solution to the BSVIE, we have
\begin{align*}
	&\sup_{t\in[\tau,\tau+\ep]}\mathbb{E}\Bigl[|\bar{Y}^\pert(t)|^2+\int^{\tau+\ep}_t|\bar{Z}^\pert(t,s)|^2\,ds\Bigr]\\
	&\leq C\sup_{t\in[\tau,\tau+\ep]}\mathbb{E}\Bigl[|\psi^\pert_3(t)|^2+\Bigl(\int^{\tau+\ep}_t|\alpha(t,s)|\,ds\Bigr)^2\\
	&\hspace{1cm}+\Bigl(\int^{\tau+\ep}_t\bigl|f\bigl(t,s,v(s),\hat{X}(s),\hat{Y}(s),\hat{Z}(t,s)\mathalpha{+}\beta(t,s)\bigr)\bigr|\,ds\Bigr)^2+\Bigl(\int^{\tau+\ep}_t|f(t,s)|\,ds\Bigr)^2\Bigr]
\end{align*}
for some constant $C>0$ which is independent of $\tau$ and $\ep$. By the same arguments as in the proof of Lemma~\ref{lemma: A1}, we can show that $\sup_{t\in[\tau,\tau+\ep]}\mathbb{E}\Bigl[\Bigl(\int^{\tau+\ep}_t|\alpha(t,s)|\,ds\Bigr)^2\Bigr]=o(\ep)$. Furthermore, observe that
\begin{align*}
	&\sup_{t\in[\tau,\tau+\ep]}\mathbb{E}\Bigl[\Bigl(\int^{\tau+\ep}_t\bigl|f\bigl(t,s,v(s),\hat{X}(s),\hat{Y}(s),\hat{Z}(t,s)\mathalpha{+}\beta(t,s)\bigr)\bigr|\,ds\Bigr)^2\Bigr]\\
	&\leq C\Bigl\{\sup_{t\in[\tau,\tau+\ep]}\mathbb{E}\Bigl[\Bigl(\int^{\tau+\ep}_t\bigl|f\bigl(t,s,v(s),\hat{X}(s),\hat{Y}(s),\hat{Z}(t,s)\mathalpha{+}\beta(t,s)\bigr)\\
	&\hspace{4.5cm}-f\bigl(t,s,v(s),\hat{X}(s),\hat{Y}(s),\hat{Z}(\tau,s)\mathalpha{+}\beta(\tau,s)\bigr)\bigr|\,ds\Bigr)^2\Bigr]\\
	&\hspace{0.8cm}+\sup_{t\in[\tau,\tau+\ep]}\mathbb{E}\Bigl[\Bigl(\int^{\tau+\ep}_t\bigl|f\bigl(t,s,v(s),\hat{X}(s),\hat{Y}(s),\hat{Z}(\tau,s)\mathalpha{+}\beta(\tau,s)\bigr)\\
	&\hspace{4.5cm}-f\bigl(\tau,s,v(s),\hat{X}(s),\hat{Y}(s),\hat{Z}(\tau,s)\mathalpha{+}\beta(\tau,s)\bigr)\bigr|\,ds\Bigr)^2\Bigr]\\
	&\hspace{0.8cm}+\sup_{t\in[\tau,\tau+\ep]}\mathbb{E}\Bigl[\Bigl(\int^{\tau+\ep}_t\bigl|f\bigl(\tau,s,v(s),\hat{X}(s),\hat{Y}(s),\hat{Z}(\tau,s)\mathalpha{+}\beta(\tau,s)\bigr)\bigr|\,ds\Bigr)^2\Bigr]\Bigr\}\displaybreak[1]\\
	&\leq \ep\,C\Bigl\{\sup_{t\in[\tau,\tau+\ep]}\mathbb{E}\Bigl[\int^T_{t_0}|\hat{Z}(t,s)-\hat{Z}(\tau,s)|^2\,ds+\sup_{s\in[t_0,T]}|p(t,s)-p(\tau,s)|^2\!\int^T_{t_0}|\delta\sigma(s)|^2\,ds\Bigr]\\
	&\hspace{1.5cm}+\rho(\ep)^2\mathbb{E}\Bigl[\int^T_{t_0}\bigl(1+|\hat{X}(s)|^2+|\hat{Y}(s)|^2+|\hat{Z}(\tau,s)\mathalpha{+}\beta(\tau,s)|^2\bigr)\,ds\Bigr]\\
	&\hspace{1.5cm}+\mathbb{E}\Bigl[\int^{\tau+\ep}_\tau\bigl|f\bigl(\tau,s,v(s),\hat{X}(s),\hat{Y}(s),\hat{Z}(\tau,s)\mathalpha{+}\beta(\tau,s)\bigr)\bigr|^2\,ds\Bigr]\Bigr\}
\end{align*}
for some constant $C>0$ which is independent of $\tau$ and $\ep$, and allowed to change from line to line. Therefore, we see that
\begin{equation*}
	\sup_{t\in[\tau,\tau+\ep]}\mathbb{E}\Bigl[\Bigl(\int^{\tau+\ep}_t\bigl|f\bigl(t,s,v(s),\hat{X}(s),\hat{Y}(s),\hat{Z}(t,s)\mathalpha{+}\beta(t,s)\bigr)\bigr|\,ds\Bigr)^2\Bigr]=o(\ep).
\end{equation*}
Similarly we can show that $\sup_{t\in[\tau,\tau+\ep]}\mathbb{E}\bigl[\bigl(\int^{\tau+\ep}_t|f(t,s)|\,ds\bigr)^2\bigr]=o(\ep)$. Thus, the estimate~\eqref{lemestimate 0} holds.
\par
Next, we prove the estimate~\eqref{lemestimate 1}. By the stability estimate~\eqref{BSVIE stability} of the difference of solutions of two BSVIEs, we see that
\begin{align*}
	&\sup_{t\in[\tau,\tau+\ep]}\mathbb{E}\Bigl[|\bar{Y}^\pert(t)-\check{Y}^\pert(t)|^2+\int^{\tau+\ep}_t|\bar{Z}^\pert(t,s)-\check{Z}^\pert(t,s)|^2\,ds\Bigr]\\
	&\leq C\sup_{t\in[\tau,\tau+\ep]}\mathbb{E}\Bigl[|\psi^\pert_3(t)|^2+\Bigl(\int^{\tau+\ep}_t|\tilde{f}^\pert_y(t,s)+\tilde{\mathfrak{f}}^\pert_y(t,s)|\,|\bar{Y}^\pert(s)|\,ds\Bigr)^2\\
	&\hspace{3cm}+\Bigl(\int^{\tau+\ep}_t|\tilde{f}^\pert_z(t,s)+\tilde{\mathfrak{f}}^\pert_z(t,s)|\,|\bar{Z}^\pert(t,s)|\,ds\Bigr)^2\Bigr]
\end{align*}
for some constant $C>0$ which is independent of $\tau$ and $\ep$. From the discussions in Section~\ref{section: proof of main result}, we know that $\sup_{t\in[\tau,\tau+\ep]}\mathbb{E}\bigl[|\psi^\pert_3(t)|^2\bigr]=o(\ep^2)$. Moreover, by using the estimate~\eqref{lemestimate 0}, we obtain
\begin{align*}
	&\sup_{t\in[\tau,\tau+\ep]}\mathbb{E}\Bigl[\Bigl(\int^{\tau+\ep}_t|\tilde{f}^\pert_y(t,s)+\tilde{\mathfrak{f}}^\pert_y(t,s)|\,|\bar{Y}^\pert(s)|\,ds\Bigr)^2\Bigr]\\
	&\leq9\|\partial_yf\|^2_\infty\ep^2\sup_{t\in[\tau,\tau+\ep]}\mathbb{E}\bigl[|\bar{Y}^\pert(t)|^2\bigr]=o(\ep^3)
\end{align*}
and
\begin{align*}
	&\sup_{t\in[\tau,\tau+\ep]}\mathbb{E}\Bigl[\Bigl(\int^{\tau+\ep}_t|\tilde{f}^\pert_z(t,s)+\tilde{\mathfrak{f}}^\pert_z(t,s)|\,|\bar{Z}^\pert(t,s)|\,ds\Bigr)^2\Bigr]\\
	&\leq9\|\partial_zf\|^2_\infty\ep\sup_{t\in[\tau,\tau+\ep]}\mathbb{E}\Bigl[\int^{\tau+\ep}_t|\bar{Z}^\pert(t,s)|^2\,ds\Bigr]=o(\ep^2).
\end{align*}
Consequently, we get the estimate~\eqref{lemestimate 1}.
\end{proof}

\par
\medskip
In order to prove the necessity part of Theorem~\ref{theorem: main result}, we need the following abstract lemma, which is a slight modification of Lemma~3.5 of \cite{a_Hu-Huang-Li_17}. We provide a complete proof here for the sake of self-containedness.

%% Lemma

\begin{lemm}\label{lemma: A5}
Let $\varphi(\cdot)\in L^1_\mathbb{F}(S,T;\mathbb{R})$ with $0\leq S<T<\infty$ be fixed. Assume that
\begin{equation*}
	\liminf_{\ep\downarrow0}\frac{1}{\ep}\mathbb{E}\Bigl[\int^{t+\ep}_t\varphi(s)\,ds\,\xi_t\Bigr]\geq0,
\end{equation*}
for any $t\in[S,T)$ and any nonnegative, bounded and $\mathcal{F}_t$-measurable random variable $\xi_t$. Then it holds that $\varphi(s)\geq0$ for $\mathrm{Leb}_{[S,T]}\otimes\mathbb{P}$-a.e.\,$(s,\omega)\in[S,T]\times\Omega$.
\end{lemm}

%% Proof

\begin{proof}
Since the map $[S,T]\ni t\mapsto \varphi(t)\in L^1_{\mathcal{F}_T}(\Omega;\mathbb{R})$ is Bochner integrable, by Lebesgue's differentiation theorem for Bochner integrable functions (cf.\ Theorem~3.8.5 of \cite{b_Hille-Phillips_57}), we have, for a.e.\ $t\in[S,T)$,
\begin{equation}\label{null set 1}
	\lim_{\ep\downarrow0}\frac{1}{\ep}\int^{t+\ep}_t\|\varphi(s)-\varphi(t)\|_{L^1_{\mathcal{F}_T}(\Omega;\mathbb{R})}\,ds=0.
\end{equation}
Take an arbitrary $t\in[S,T)$ satisfying \eqref{null set 1}. For any $\ep\in(0,T-t)$ and any nonnegative, bounded and $\mathcal{F}_t$-measurable random variable $\xi_t$, we have
\begin{equation*}
	\mathbb{E}\bigl[\varphi(t)\xi_t\bigr]=\frac{1}{\ep}\mathbb{E}\Bigl[\int^{t+\ep}_t\varphi(s)\,ds\,\xi_t\Bigr]-\frac{1}{\ep}\mathbb{E}\Bigl[\int^{t+\ep}_t\bigl(\varphi(s)-\varphi(t)\bigr)\,ds\,\xi_t\Bigr].
\end{equation*}
Note that
\begin{equation*}
	\Bigl|\frac{1}{\ep}\mathbb{E}\Bigl[\int^{t+\ep}_t\bigl(\varphi(s)-\varphi(t)\bigr)\,ds\,\xi_t\Bigr]\Bigr|\leq\|\xi_t\|_\infty\frac{1}{\ep}\int^{t+\ep}_t\|\varphi(s)-\varphi(t)\|_{L^1_{\mathcal{F}_T}(\Omega;\mathbb{R})}\,ds\overset{\ep\downarrow0}{\longrightarrow}0.
\end{equation*}
Furthermore, by the assumption, we have $\liminf_{\ep\downarrow0}\frac{1}{\ep}\mathbb{E}\Bigl[\int^{t+\ep}_t\varphi(s)\,ds\,\xi_t\Bigr]\geq0$. Thus, we see that $\mathbb{E}\bigl[\varphi(t)\xi_t\bigr]\geq0$. Since $\varphi(t)$ is $\mathcal{F}_t$-measurable, we get $\varphi(t)\geq0$ a.s. This completes the proof.
\end{proof}

%%%%%%%%%%%%%%%%%%%%%%%%%%%%%%%%%%%%%%%%%%%%%%%%%%%%%%%%%%%%%%%%%%%%%%%%%%%%%%%%%%%%%%%%%%%%%%
%%%%%%%%%%%%%%%%%%%%%%%%%%%%%%%%%%%%%%%%%%%%%%%%%%%%%%%%%%%%%%%%%%%%%%%%%%%%%%%%%%%%%%%%%%%%%%
%%%%%%%%%%%%%%%%%%%%%%%%%%%%%%%%%%%%%%%%%%%%%%%%%%%%%%%%%%%%%%%%%%%%%%%%%%%%%%%%%%%%%%%%%%%%%%

\end{document}